\documentclass[11pt]{amsart}
\usepackage{amsfonts, amssymb, amsmath}
\usepackage{mathrsfs,mathtools}
\usepackage{tikz-cd}
\usepackage[all]{xy}
\usepackage{url}
\usepackage[utf8]{inputenc}
\usepackage[hidelinks]{hyperref}
\usepackage{enumerate}

\usepackage[left=25mm,right=25mm,top=38mm,bottom=34mm]{geometry}

\newcommand{\spa}{\medskip}
\newcommand{\mr}[1]{\mathrm{#1}}
\newcommand{\m}{\textrm{-}}
\newcommand{\stack}[1]{\cite[\href{https://stacks.math.columbia.edu/tag/#1}{Tag #1}]{Stacks}}

\newcommand{\Q}{\mathbb{Q}}
\newcommand{\F}{\mathbb{F}}
\newcommand{\Z}{\mathbb{Z}}
\newcommand{\N}{\mathbb{N}}
\newcommand{\Zp}{\mathbb{Z}_p}
\newcommand{\Qp}{\mathbb{Q}_p}
\newcommand{\Fp}{\F_p}

\newcommand{\PP}{\mathbb{P}}
\newcommand{\Spec}{\mr{Spec}}
\newcommand{\Spf}{\mr{Spf}}
\newcommand{\Tr}{\mr{Tr}}
\newcommand{\id}{\mr{id}}

\newcommand{\dR}{\mathrm{dR}}
\newcommand{\cris}{\mathrm{cris}}
\newcommand{\et}{{\mathrm{\acute{e}t}}}
\newcommand{\proet}{{\mathrm{pro\acute{e}t}}}
\newcommand{\syn}{{\mr{syn}}}

\newcommand{\tors}{\mathrm{tors}}

\newcommand{\calE}{\mathcal{E}}

\newcommand{\calH}{\mathcal{H}}

\newcommand{\calO}{\mathcal{O}}
\newcommand{\calP}{\mathcal{P}}
\newcommand{\calM}{\mathcal{M}}
\newcommand{\calN}{\mathcal{N}}

\newcommand{\calV}{\mathcal{V}}

\def\bbA{{\mathbb A}}

\def\bbD{{\mathbb D}}
\def\bbE{{\mathbb E}}

\def\bbN{{\mathbb N}}

\def\bbQ{{\mathbb Q}}

\def\bbZ{{\mathbb Z}}

\newcommand{\cA}{\mathcal{A}}
\newcommand{\cB}{\mathcal{B}}
\newcommand{\cC}{\mathcal{C}}
\newcommand{\cD}{\mathcal{D}}
\newcommand{\cE}{\mathcal{E}}
\newcommand{\cF}{\mathcal{F}}
\newcommand{\cG}{\mathcal{G}}
\newcommand{\cH}{\mathcal{H}}

\newcommand{\cK}{\mathcal{K}}

\newcommand{\cM}{\mathcal{M}}
\newcommand{\cN}{\mathcal{N}}
\newcommand{\cO}{\mathcal{O}}

\newcommand{\cU}{\mathcal{U}}
\newcommand{\cV}{\mathcal{V}}
\newcommand{\cW}{\mathcal{W}}

\newcommand{\fA}{\mathfrak{A}}

\newcommand{\fD}{\mathfrak{D}}

\newcommand{\fP}{\mathfrak{P}}

\newcommand{\fS}{\mathfrak{S}}
\newcommand{\fT}{\mathfrak{T}}
\newcommand{\fU}{\mathfrak{U}}

\newcommand{\fX}{\mathfrak{X}}
\newcommand{\fY}{\mathfrak{Y}}

\newcommand{\fa}{\mathfrak{a}}

\newcommand{\TE}{\mathbb{T}_E}

\def\rmb{{\mathrm b}}

\def\rmH{{\mathrm H}}

\DeclareMathOperator{\coh}{coh}
 
\DeclareMathOperator{\Coker}{Coker}
\DeclareMathOperator{\Cocone}{Cocone}

\DeclareMathOperator{\conv}{conv}

\DeclareMathOperator{\Ext}{Ext}

\DeclareMathOperator{\h}{h} 
\DeclareMathOperator{\Hom}{Hom}

\DeclareMathOperator{\im}{Im}

\DeclareMathOperator{\Isoc}{Isoc}
\DeclareMathOperator{\Ker}{Ker}

\DeclareMathOperator{\loc}{loc}

\DeclareMathOperator{\MW}{MW}

\DeclareMathOperator{\rig}{rig}

\DeclareMathOperator{\sep}{sep}

\newcommand{\riso}{ \overset{\sim}{\longrightarrow}\, }

\newcommand{\hdag}{  \phantom{}{^{\dag} }    }

\usepackage{accents}

\newcommand\underrvec[1]{{\underaccent{\rightarrow}{#1}}}

\begin{document}
	
	\newtheorem{theo}[subsubsection]{Theorem}
	\newtheorem*{theo*}{Theorem}
	\newtheorem{ques}[subsubsection]{Question}
	\newtheorem*{ques*}{Question}
	\newtheorem{conj}[subsubsection]{Conjecture}
	\newtheorem{prop}[subsubsection]{Proposition}
	\newtheorem{lemm}[subsubsection]{Lemma}
	\newtheorem*{lemm*}{Lemma}
	\newtheorem{coro}[subsubsection]{Corollary}
	\newtheorem*{coro*}{Corollary}
	
	\theoremstyle{definition}
	\newtheorem{defi}[subsubsection]{Definition}
        \newtheorem{empt}[subsubsection]{}
	\newtheorem*{defi*}{Definition}
	\newtheorem{hypo}[subsubsection]{Hypothesis}
	\newtheorem{rema}[subsubsection]{Remark}
    \newtheorem*{rema*}{Remark}

	\newtheorem{warn}[subsubsection]{Warning}
	\newtheorem{exam}[subsubsection]{Example}
	\newtheorem{nota}[subsubsection]{Notation}
	\newtheorem{cons}[subsubsection]{Construction}
	
	\numberwithin{equation}{section}
	\title{Injectivity failure in crystalline comparisons}
	\date{\today}
	\makeatletter
	\@namedef{subjclassname@2020}{%
		\textup{2020} Mathematics Subject Classification}
	\makeatother
	
	
	
	\author{Daniel Caro}
	\address{University of Caen Normandie, LMNO, 14000 Caen,France}
         \email{daniel.caro@unicaen.fr}
         
         \author{Marco D'Addezio}
	\address{Institut de Recherche Math\'ematique Avanc\'ee (IRMA), Universit\'e de Strasbourg, 7 rue Ren\'e-Descartes, 67000 Strasbourg, France}
	\email{daddezio@unistra.fr}

	\begin{abstract}For smooth affine varieties in positive characteristic, we identify a slope obstruction to the injectivity of the comparison morphism from rigid cohomology to rationalised crystalline cohomology. This yields a negative answer to a question of Esnault--Kisin--Petrov concerning the injectivity of the de Rham-to-crystalline comparison map for smooth affine schemes over the Witt vectors that admit good compactifications. In contrast, we establish injectivity for certain subspaces defined by slope conditions as well as in cohomological degree one. For the latter case, we also prove the result with coefficients in $F$-able overholonomic $D$-modules leveraging a generalisation of Kedlaya's full faithfulness theorem. Beyond injectivity, we obtain various separation results for the affinoid topology on rigid and convergent cohomology. These results allow us to determine integral algebraic de Rham cohomology modulo torsion and to provide a more conceptual explanation for Ertl and Shiho's construction of varieties for which integral Monsky--Washnitzer cohomology modulo torsion is not finitely generated. Along the way, we prove a new integral comparison theorem between Monsky--Washnitzer cohomology and algebraic de Rham cohomology and we define fractional $p$-adic Tate twists, computing non-integral slopes of crystalline cohomology.
	\end{abstract}
	
	\maketitle
	\tableofcontents

	\section{Introduction}
\subsection{Injectivity results}

Let $W$ be the ring of Witt vectors of a perfect field $k$ of positive characteristic $p$ and let $K$ be the fraction field of $W$. For a smooth scheme over $W$, there is a natural comparison morphism \begin{equation}\label{comparison_map}
H^i_{\dR}(X/W)\to H^i_\cris(X_k/W),
\end{equation} from algebraic de Rham cohomology of $X$ to crystalline cohomology of the special fibre of $X$. A cornerstone of $p$-adic Hodge theory is the fact that this map is an isomorphism when $X$ is proper. Esnault--Kisin--Petrov, in their investigation of a certain vanishing conjecture for de Rham cohomology in characteristic $0$ (see \S\ref{ss_generalisedHodgeconjecture}), asked the following question.

\begin{ques}[{\cite[Prob. 5]{Esn25}}]\label{ques_EKP}
    Let $X$ be a smooth scheme over $W$ of the form $Y\setminus D$, where $Y$ is smooth and proper over $W$ and $D$ is a relative normal crossing divisor over $W$. Is the comparison morphism $$H^i_{\dR}(X/W)\to H^i_\cris(X_k/W)$$ injective for $i\geq 0$? If not, does it become injective after inverting $p$?
\end{ques}

In this direction, they observed that \eqref{comparison_map} is injective when $X$ is the complement of a smooth proper divisor of a smooth proper curve and for the complement of the diagonal of $\PP^1_W\times_W \PP^1_W$. Our article gives a complete negative answer to Question \ref{ques_EKP}. We prove the following result.

\begin{theo}[Theorem \ref{thm_counterexampleGHC}]\label{thm_counterexampleGHCintro}
    Let $Y$ be a smooth proper scheme over $W$. For every affine open $X\subseteq Y$, the image of the composition $$H^i_\dR(Y_K/K)\to H^i_\dR(X_K/K)\to H^i_\mr{cris}(X_k/W)_K$$ has dimension at most $2h^{i,0}(Y_K)+h^{i-1,1}(Y_K).$
\end{theo}
If we take  $Y$ to be a smooth cubic fourfold and $X$ to be the complement of a smooth hyperplane section, then $H^4_\dR(Y_K/K)$ has the following Hodge numbers $$h^{4,0}=0,h^{3,1}=1,h^{2,2}=21.$$ By Theorem \ref{thm_counterexampleGHCintro}, the kernel of $H^4_\dR(Y_K/K)\to H^4_\mr{cris}(X_k/W)_K$ is of dimension at least $22$. On the other hand, the kernel of $H^4_\dR(Y_K/K)\to H^4_\dR(X_K/K)$ is generated by the self-intersection of the hyperplane class. Thus $$\mr{Ker}(H^4_\dR(X_K/K)\to H^4_\mr{cris}(X_k/W)_K)$$ contains a subspace of weight $4$ of dimension at least $21.$ All the restrictions of the Chern classes of $Y_K$ are in this kernel and die even before inverting $p$ (Corollary \ref{cor_VanishingChernClasses}). Furthermore, if $Y_k$ is ordinary, we get an integral enhancement as a consequence of Theorem \ref{thm_Ordinary}.

\spa

Thanks to the comparison between rational de Rham cohomology of $X_K$ and rigid cohomology of $X_k$ and the Newton-above-Hodge inequalities, Theorem \ref{thm_counterexampleGHCintro} is a special case of the following more general result for rigid cohomology.

\begin{theo}[Theorem \ref{thm_counterexample}]\label{thm_counterexampleIntro}
Let $X_0$ be a smooth affine $k$-scheme, let $\gamma$ be a rational number, and let $\calM^\dagger$ be an overconvergent $F$-isocrystal over $X_0$ with constant slopes in the interval $[\alpha,\beta]$. Then, for $i\not \in [\gamma-\beta , \gamma -\alpha +1]$, the comparison morphism $$H^{i}_\mr{rig}(X_0,\calM^\dagger) \to H^{i}_{\mr{conv}}(X_0,\calM)$$ kills $H^{i}_\mr{rig}(X_0,\calM^\dagger)^{[\gamma]}$, the slope $\gamma$ subspace.

\end{theo} Plugging $\calM^\dagger\coloneqq\calO_\mr{rig}$ into Theorem \ref{thm_counterexampleIntro}, we get that the comparison morphism $$H^i_\mr{rig}(X_0/K)\to H^i_\mr{conv}(X_0/K)=H^i_\cris(X_0/W)_K$$ kills $H^i_\mr{rig}(X_0/K)^{[\gamma]}$ when $i\notin [\gamma,\gamma+1].$ By taking the previous example of a smooth cubic fourfold, this shows that the rigid-to-convergent comparison morphism fails to be injective because of a slope obstruction\footnote{It seems that the question whether the comparison map from rigid to convergent cohomology is injective is a folklore question we were not able to locate in the literature. The second named author asked himself the question during his work on edged crystalline cohomology \cite{EdgedCrystalline}, and we later learned that Bhatt was also asking this question for applications to Chern classes.}. The result is proved using the \textit{$F$-gauge structure} of crystalline cohomology (see \S\ref{sec_FGauge}). We obtain Theorem \ref{thm_counterexampleIntro} by analysing $p$-adic convergence properties of the powers of the divided Frobenius morphisms. The fundamental technical result is Proposition \ref{Prop_RationalBounds}, which is an extension of \cite[Prop. 3.1.4]{DK17}.

\spa
Furthermore, by analysing this injectivity failure in the context of arithmetic $D$-modules, and thanks to Huyghe's theorems of types A and B (Theorem \ref{Gamma5.3.3-Huyghe}), we produce noncommutative geometric counterexamples to the following well-known fact of homological algebra (Example \ref{CounterexampleExt}).

\begin{prop}\label{prop-ffExtIntro}
Let $A \to B$ be a faithfully flat morphism of coherent commutative rings, let $M$ be a coherent $A$-module, and let $N$ be an $A$-module. Then the canonical map 
	$$\operatorname{Ext}^i_A(M, N) \hookrightarrow \operatorname{Ext}^i_B(B\otimes_A M, B \otimes_A N)$$ is injective for $i\geq 0$.
\end{prop}

On the positive side, we prove the injectivity of the comparison morphism for certain subspaces determined by slope conditions.

\begin{theo}[Theorem \ref{thm_InjectivitySlopei} and Theorem \ref{thm_InjectivitySlopeibis}]\label{thm_InjectivitySlopeiIntro}If $X_0$ is a quasi-compact smooth $k$-scheme, the restriction of the comparison morphism $$H^r_\mr{rig}(X_0/K)^{(r-1,r]}\to H^r_\mr{conv}(X_0/K)$$ is injective for $r\geq 0$. Moreover, for slope $r$ we get the isomorphism $$H^r_\mr{rig}(X_0/K)^{[r]}\riso H^r_\mr{conv}(X_0/K)^{[r]}.$$
\end{theo}
The slope conditions of Theorem \ref{thm_InjectivitySlopeiIntro} allow descent techniques that are not available for the entire cohomology groups. This is because the slope bounds of rigid cohomology prevent cancellations with lower degree cohomological classes.

\spa

In cohomological degree one, we can also prove a more general injectivity result with coefficients in overconvergent $F$-isocrystals. The key tool here is Kedlaya's full faithfulness theorem \cite{Ked04}. 

\begin{theo}[Theorem \ref{thm_RC}]\label{thm_RCIntro}
 If $X_0$ is a quasi-compact smooth $k$-scheme and $\calM^\dagger$ is an overconvergent $F$-isocrystal over $X_0/K$, then $$H^1_{\mr{rig}}(X_0,\calM^\dag)\to H^1_{\mr{conv}}(X_0,\calM) $$ is injective, where $\calM$ is the associated convergent $F$-isocrystal.
\end{theo}
We then generalise Theorem \ref{thm_RCIntro} to $F$-able overholonomic $D$-modules in Corollary \ref{cor_Hol-ff}. The key step is Theorem \ref{Hol-ff}, which generalises Kedlaya's full faithfulness to this broader context. The proof relies on two fundamental facts: first, Kedlaya's semistable reduction theorem \cite{ked11} implies that overconvergent $F$-isocrystals are overholonomic \cite{CT12}; second, $F$-able overholonomic $D$-modules are \textit{devissable} into overconvergent $F$-isocrystals \cite{car06}. This devissability property allows us to reduce the problem to the case where Kedlaya's theorem applies.

\subsection{A glimpse into the obstruction}
It is well-known that crystalline cohomology has a flaw: it is ``too big'' for open varieties, as it yields infinite-dimensional vector spaces. Our results show that it is also ``too small''. This dual criticality confirms the need for rigid cohomology in the motivic setting (see also Corollary \ref{cor_VanishingChernClasses}). Let us further explain the injectivity obstruction in the simplest setting. 

\spa

If $X=\Spec(R)$ is smooth and affine over $W$ and is endowed with a Frobenius lift $F$, the crystalline cohomology of $X_k$ is nothing but the cohomology of the $p$-adically completed de Rham complex $$ \widehat{R}\to \Omega^1_{\widehat{R}/W}\to\Omega^2_{\widehat{R}/W}\to\dots$$  (see Corollary \ref{cor_quasi-iso}). The Frobenius lift $F$ induces an endomorphism on this complex, called the \textit{crystalline Frobenius}, which is unique up to homotopies. The naive truncation $\Omega^{\geq i}_{\widehat{R}/W}$ is further endowed with divided Frobenius endomorphisms $\varphi_r\coloneqq F/p^r$ for $r\leq i$. This follows from the fact that $F(\omega)$ is divisible by $p$ for every $\omega\in \Omega^1_{\widehat{R}/W}$. For the same reason, $\varphi_r$ acting on $\Omega^{\geq r+1}_{\widehat{R}/W}$ is topologically nilpotent at each degree. It follows that $$\varphi_r-1\colon \Omega^{i}_{\widehat{R}/W}\to \Omega^{i}_{\widehat{R}/W} $$ is an invertible operator for $i\geq r+1$. From this we deduce that $$H^i_\cris(X_k/W)^{F=p^r}=0$$ for $i\geq r+2.$. On the other hand, when $k$ is algebraically closed, the group $H^i_\dR(X_K/K)^{F=p^r}$ can be pretty big thanks to Dieudonné--Manin classification. For example, there can be non-trivial classes coming from the restriction $$H^i_\dR(Y_K/K)^{F=p^r}\to H^i_\dR(X_K/K)^{F=p^r}.$$ They can be studied using Gysin isomorphism. Note that $H^i_\dR(Y_K/K)^{F=p^r}$, in the ``good situations'', is a $\Qp$-vector space of dimension $h^{r,i-r}(Y_K)$.  Note also that if $r\geq 2$ and $i=2r$, all Chern classes in $H^{2r}_\dR(Y_K/K)=H^{2r}_\cris(Y_k/W)_K$ die in $H^{2r}_\cris(X_k/W)_K$. This happens even before inverting $p$ thanks to Corollary \ref{cor_VanishingChernClasses}.

\subsection{The generalised Hodge conjecture}
\label{ss_generalisedHodgeconjecture}

Grothendieck's generalised Hodge conjecture, proposed in \cite{Gro69}, has the following special case.

\begin{conj}(Grothendieck)\label{conj_specialGroth}
    Let $Y$ be a smooth projective complex variety of pure dimension $n$ such that $h^{i,0}=0$ for some $i\leq n$. Then there exists a dense open $U\subseteq Y$ such that the restriction map $$H^i_{\mr{Betti}}(Y,\Q)\to H^i_{\mr{Betti}}(U,\Q)$$ vanishes.
\end{conj}
Even this special case is at the moment a far-reaching conjecture. In \cite{EKP}, Esnault--Kisin--Petrov tried to approach Conjecture \ref{conj_specialGroth} via a reduction to positive characteristic. In this direction, they proved the following result.
\begin{theo}[{\cite[Thm. 2]{Esn25}}]\label{thm_EKP}
    Let $Y$ be a smooth proper scheme over $W$ and $i$ an integer such that $$H^0(Y_k,\Omega^i_{Y_k})=0.$$ Then, for every affine open $U\subseteq Y$, the induced morphism $$H^i_\mr{dR}(Y/W)\to H^i_\cris(U_k/W)$$ factors through the subgroup $$N^i_\mr{cris}(U_k/W)\coloneqq \bigcap_{n=1}^\infty p^nH^i_\cris(U_k/W).$$
    
    \end{theo}Their theorem leverages the vanishing of $H^i_\mr{dR}(Y_k/k)\to H^i_{\mr{dR}}(U_k/k)$ arising from the Cartier isomorphism. The group $N^i_\mr{cris}(U_k/W)$ is the kernel of the maximal separated quotient $$ H^i_\cris(U_k/W)\twoheadrightarrow H^i_\cris(U_k/W)^\mr{sep}.$$
    
    Theorem \ref{thm_counterexampleGHCintro} shows that it is not possible to directly reduce Conjecture \ref{conj_specialGroth} to a statement for crystalline cohomology, since there are many de Rham classes that are undetected in the crystalline setting (see also Theorem \ref{thm_Ordinary}).


\subsection{Separation results} 
Spurred by Theorem \ref{thm_EKP}, we undertook a more detailed analysis of the separation properties of rigid and convergent cohomology. This ultimately led us to determine the structure of integral algebraic de Rham cohomology modulo torsion.

\begin{theo}[Theorem {\ref{thm_DeterminationAlgebraicdR}}]\label{thm_DeterminationAlgebraicdRIntro}
     Let $Y$ be a smooth proper scheme over $W$, let $D$ be a relative normal crossing divisor, and let $X\subseteq Y$ be the complement of $D$. If $X$ is affine, then
     $$H^i_\dR(X/W)/\mr{tors}\simeq K^{\oplus a}\oplus W^{\oplus b},$$ where $a$ is the dimension of $H^i_\mr{rig}(X_0/K)^{[0,i)}$ and $b$ is the dimension of $H^i_\mr{rig}(X_0/K)^{[i]}.$
\end{theo}
By faithfully flat descent, Theorem \ref{thm_DeterminationAlgebraicdRIntro} also extends to non-complete rings as $\Z_{(p)}$, enabling applications in more arithmetic contexts.

\spa

The topology on rigid and convergent cohomology that we studied is the so-called \textit{affinoid topology}, which is induced by the $p$-adic topology on differential forms (see \S\ref{Sec_AffinoidTopology}). Writing $N^i_\mr{conv}$ for the $p$-adic closure of $0$ of $H^i_\mr{conv}$, we prove the following result.
\begin{theo}[Theorem \ref{thm_slopeiminus1-variant} and Theorem \ref{thm_sloperseparated}]\label{thm_slopeiminus1-variantIntro} If $X_0$ is a smooth affine scheme over $k$ and $\calE$ is a unit-root convergent $F$-isocrystal over $X_0$, then $$H^{i}_\mr{conv}(X_0,\calE)^{[0,i)}\subseteq N^{i}_{\mr{conv}}(X_0,\calE)$$
for $i \geq 1$. In addition, $H^{i}_\mr{conv}(X_0,\calE)^{[i]}$ intersects trivially $N^{i}_\mr{conv}(X_0,\calE)$ for $i \geq 0$.
\end{theo}

The inclusion $H^{i}_\mr{conv}(X_0,\calE)^{[0,i)}\subseteq N^{i}_{\mr{conv}}(X_0,\calE)$, is established using an argument analogous to that of Theorem \ref{thm_counterexampleIntro}. For the second assertion, instead, one of the key ingredients is the fact that $H^0(X_0,W\Omega_\mr{log}^i)$ is a finite-rank $\Zp$-module (Proposition \ref{prop_Bosco}). By a Beauville--Laszlo argument, we deduce from Theorem \ref{thm_slopeiminus1-variantIntro} the following result for integral Monsky--Washnitzer cohomology.
\begin{coro}[Corollary \ref{cor_ES20} and Corollary \ref{cor_RigidSlopeRSeparated}]\label{cor_ES20Intro}
    For a smooth affine $k$-scheme $X_0$ and $i\geq 1$, there exists a natural $W$-linear injective morphism $$H^{i}_\mr{rig}(X_0/K)^{[0,i)}\hookrightarrow H^{i}_{\mr{MW}}(X_0/W)/\mr{tors},$$ where $H^{i}_{\mr{MW}}$ denotes integral Monsky--Washnitzer cohomology. The cokernel is a free $W$-module of the same rank as the dimension of $H^{i}_\mr{rig}(X_0/K)^{[i]}.$
    \end{coro}
Corollary \ref{cor_ES20Intro} provides a more transparent explanation for the phenomenon discovered by Ertl and Shiho in \cite{ES20}, where they constructed examples such that $H^i_{\mathrm{MW}}(X_0/W)/\mathrm{tors}$ is not finitely generated over $W$. It shows that this non-finiteness is not an ``anomaly'' but a fundamental feature of integral Monsky--Washnitzer cohomology.

\spa

In the presence of good compactifications, we also prove that integral Monsky--Washnitzer cohomology is isomorphic to the algebraic de Rham cohomology of a lift.
\begin{theo}[Theorem \ref{thm_comparisonMWdR}]\label{thm_comparisonMWdRIntro}
    Let $Y$ be a smooth proper scheme over $W$ and $D$ a relative normal crossing divisor. If the complement $X\coloneqq Y\setminus D$ is affine, the weak completion induces a quasi-isomorphism $$R\Gamma_\mr{dR}(X/W)\xrightarrow{\sim} R\Gamma_{\mr{MW}}(X_k/W).$$
\end{theo}

Since integral Monsky–Washnitzer cohomology has a natural Frobenius structure, Theorem \ref{thm_comparisonMWdRIntro} provides integral algebraic de Rham cohomology with a Frobenius structure, at least when good compactifications exist. Combining Theorem \ref{thm_comparisonMWdRIntro} with our separation results, we finally get Theorem \ref{thm_DeterminationAlgebraicdRIntro}. When $Y$ is an elliptic curve and $D$ is the origin, Theorem \ref{thm_DeterminationAlgebraicdRIntro} was previously proved by Esnault--Kisin--Petrov via an explicit computation with the de Rham--Witt complex.

\subsection{Fractional $p$-adic Tate twists}
Illusie and Milne introduced logarithmic de Rham--Witt sheaves $W_n\Omega^i_{\mr{log}}$
whose cohomology recovers integral slopes of crystalline cohomology. These sheaves have a $K$-theoretic reinterpretation as explained in \cite[Thm. 8.3]{GL00} and compute $p$-adic étale motivic cohomology by [\textit{ibid}, Thm. 8.4].
In our work, we introduce and systematically use fractional $p$-adic Tate twists in order to deal with the slopes of crystalline cohomology that are not integral. In the setting of the de Rham--Witt complex, for non-integral slopes, to quotients of the pro-étale sheaves of de Rham--Witt differentials (Section \ref{Sec_FractionalLogarithmicdRW}).

\spa

More precisely for $\gamma\in \Q$ we construct complexes of pro-étale sheaves $\nu(\gamma)$, with cohomology concentrated in degrees $[0,1]$, endowed with a morphism of complexes $\nu(\gamma)\to W\Omega^\bullet[b]$, where $b\coloneqq \lfloor\gamma\rfloor$ (Definition \ref{def_dRWFractionalTate}). We prove the following foundational result.

\begin{theo}[Theorem \ref{thm_FractionalSlopesCrystalline}]\label{thm_FractionalSlopesCrystallineIntro}
Let $k$ be an algebraically closed field of characteristic $p$, let $X_0$ be a smooth proper $k$-scheme, and let $\gamma$ be a rational number with $b\coloneqq\lfloor\gamma\rfloor.$ For every $i\geq 0$, there exists a functorial isomorphism $$H^i_\proet(X_0,\nu(\gamma))_K\riso H^{i+b}_\cris(X_0/W)_K^{[\gamma]}$$ of $K$-vector spaces. Moreover, $\nu(\gamma)$ is quasi-isomorphic to $W\Omega_\mr{log}^b[0]$ when $\gamma\in \Z$ and to a pro-étale sheaf sitting in degree $1$ otherwise.
\end{theo}

\spa

We do not build connections with $K$-theory and étale motivic cohomology, but we wonder whether there are variants of these theories encompassing fractional slopes as in our setting.
\subsection{Overview of the paper}
In §\ref{sec_dRVersusCrystalline}, we recall the de Rham–crystalline comparison and prove the vanishing of $V^i_{\mathrm{dR}}$ in low degrees (Theorem \ref{thm_lowdegree}). In §\ref{Sec_RigidVersusConvergent}, we transition to rigid cohomology, establishing the integral comparison between de Rham and Monsky–Washnitzer cohomology (Theorem \ref{thm_comparisonMWdR}) and the injectivity in degree one (Theorem \ref{thm_RC}). Subsequently, in §\ref{Sec_SlopeBounds}, we prove the main slope bounds for crystalline cohomology (Proposition \ref{Prop_IntegralBounds}) after introducing the fractional Nygaard filtration in \S\ref{Sec_FractionalNygaard}. In \S\ref{Sec_InjectivityFailure} we focus on the counterexamples to Question \ref{ques_EKP}, provided by Theorem \ref{thm_counterexampleGHC}. In \S\ref{Sec_dRW}, we recall the theory of the de Rham--Witt complex with coefficients in unit-root $F$-isocrystals and for log schemes and we develop fractional $p$-adic Tate twists in this setting, which compute non-integral slopes in the proper setting (Theorem \ref{thm_FractionalSlopesCrystalline}). The separation properties are proved in \S\ref{Sec_SeparationDefect}, where we prove Theorems \ref{thm_slopeiminus1-variant} and \ref{thm_sloperseparated}, which determine the structure of integral de Rham cohomology in Theorem \ref{thm_DeterminationAlgebraicdR}. We also get
positive injectivity results under specific slope conditions (Theorems \ref{thm_InjectivitySlopei} and \ref{thm_InjectivitySlopeibis}). One of the ingredients is an extended version of syntomic cohomology (Definition \ref{dfn-TijE}). In \S\ref{Sec_InjectivityArithmeticDmodules} we establish full faithfulness for arithmetic $D$-modules in Theorem \ref{Hol-ff} and we generalise Theorem \ref{thm_RC} to $F$-able overholonomic $D$-modules in Corollary \ref{cor_Hol-ff}.

\subsection*{Acknowledgments}
We warmly thank H\'el\`ene Esnault, Alexander Petrov, and Mark Kisin for asking Question \ref{ques_EKP} and for sharing their work in progress. This greatly motivated and inspired our article. We thank Bruno Chiarellotto, Pierre Colmez, Veronika Ertl, Livia Grammatica, Kiran Kedlaya, Bernard Le Stum, and Atsushi Shiho for helpful conversations on the topic, and Bhargav Bhatt, Richard Crew, Sergey Gorchinskiy, Christopher Lazda, Matthew Morrow, Kay R\"ulling, Peter Scholze, and Alberto Vezzani for the interest shown. We further thank Bruno Chiarellotto for providing the reference to \cite{BC94} and Kay R\"ulling for his interesting comments on a first draft. The second-named author would also like to thank Alexander Petrov for identifying a mistake in an initial attempt to answer Question \ref{ques_EKP}. 

\spa

This research was partially funded by the French National Research Agency (ANR) under the project ANR-25-CE40-7869-01.
	
	\section{De Rham versus crystalline cohomology}
    \label{sec_dRVersusCrystalline}
    \subsection{Preliminaries}
In this section, we first recall the fundamental comparison theorem between algebraic de Rham cohomology and crystalline cohomology. Subsequently, we define the key object of study: the kernel of the comparison map, denoted by $V^i_\dR(X/S)$ (Definition \ref{dfnVi-Hbari}). 

    \begin{nota}
        Let $R$ be a Noetherian $\Z_{(p)}$-algebra endowed with the standard PD-structure on $(p)$. We write $R_n$ for  its reduction modulo $p^{n}$ 
    and we set $S\coloneqq \Spec (R)$ and $\fS \coloneqq \Spf (\widehat{R})$, where $\widehat{R}\coloneqq \lim_n R_n.$  If $\fX$ is a  formal $\fS$-scheme, or if $X$ is an $S$-scheme, we denote by $X_n$ the reduction modulo $p^n$ and by $X_0$ the scheme $(X_1)_\mr{red}$. We say that a formal $\fS$-scheme is smooth if each $X_n/ S_n$ is smooth. We write $R \Gamma_{\dR} (X/S)$ for the algebraic de Rham complex $R \Gamma (X, \Omega^\bullet_{X/S})$ and $R \Gamma_{\dR} (\mathfrak{X}/\mathfrak{S})$ for the de Rham complex $R \Gamma (\mathfrak{X}, \Omega^\bullet_{\mathfrak{X}/\mathfrak{S}})$. When the base is clear, we drop it in the notation.
    \end{nota}
         The fundamental de Rham comparison theorem, proved by Berthelot, can be formulated in the following form.

\begin{theo}[Berthelot]\label{thm_quasi-iso0}
Let $\fX$ be a smooth formal $\fS$-scheme.
The following complexes of $\widehat{R}$-modules are canonically quasi-isomorphic.
	\begin{enumerate}[(1)]
	         \item The crystalline cohomology complex $R\Gamma_{\cris}(X_0/S)$.
             \item The derived limit $R\lim_n R\Gamma_{\cris}(X_0/S_n)$.
             \item The derived limit $R\lim_n R \Gamma_{\dR} ( X_n /S_n)$.
		\item The de Rham complex $R \Gamma_{\dR} ( \fX /\fS)$.
		
\end{enumerate}
\end{theo}

\begin{proof}

The equality $R\Gamma_{\cris}(X_0/S)=R\lim_n R\Gamma_{\cris}(X_0/S_n)$ follows from \stack{07MV}, which justifies the quasi-isomorphism between (1) and (2). By \cite[Thm. 7.1]{BerthelotOgus}, we have $R\Gamma_{\cris}(X_0/S_n)= R \Gamma (X_n, \Omega^\bullet_{X_n /S_n})$, which gives the quasi-isomorphism between (2) and (3). See also \cite{BdJ11} for a different proof.
The quasi-isomorphism between (1) and (4) is \cite[Thm. 7.23]{BerthelotOgus}.
\end{proof}

\begin{empt}\label{adjdR-ntn}
Let $X$ be a smooth $S$-scheme.  There is a natural base change morphism
\begin{equation}\label{adjdR}
R \Gamma_{\dR} (X/S) \otimes_{R}^L  R_n  \to R \Gamma_{\dR} (X_n /S_n),
\end{equation}
which induces the map 
\begin{equation}\label{RlimadjdR}
R \Gamma_{\dR} (X/S)^{\wedge}\coloneqq R \lim_n \left ( R \Gamma_{\dR} (X/S) \otimes_{R}^L  R_n  \right )\to 
R \lim_n  R \Gamma_{\dR} (X_n /S_n).
\end{equation}
Composing with the map $R \Gamma_{\dR} (X/S)\to R \Gamma_{\dR} (X/S)^{\wedge}$, we get the fundamental comparison morphism
\begin{equation}\label{RlimadjdRcris}
R \Gamma_{\dR} (X/S) \to   R\Gamma_{\cris}(X_0/S),
\end{equation}
since $R \lim_n  R \Gamma_{\dR} (X_n /S_n)=  R\Gamma_{\cris}(X_0/S)$ by Theorem \ref{thm_quasi-iso0}.
\end{empt}

\begin{lemm}\label{comp-dR-Rim}
If $X/S$ is either affine or proper, the maps \eqref{adjdR} and \eqref{RlimadjdR} are isomorphism. 
\end{lemm}

\begin{proof}
When $X= \Spec (A)$, then 
$$R \Gamma_{\dR} (X/S) \otimes_{R}^L  R_n = \Omega^{\bullet}_{A/R} \otimes_{R}^L  R_n =  \Omega^{\bullet}_{A_n /R_n}=R \Gamma_{\dR} (X_n/S_n).$$
When $X/S$ is proper, this follows from the proper base change theorem (see \stack{09V6}).
Indeed, let $f\colon X \to S$, $i_S \colon S_n \to S$, $f_n\colon X _n\to S_n$, $i_X \colon X_n \to X$ be the given maps. We get the desired isomorphism by applying the functor $R\Gamma (S_n , -)$ to the composition
\begin{gather*}
    R f_{n*} (\Omega^\bullet_{X_n/S_n})  \riso 
    R f_{n*} (f_n^{-1} \cO_{S_n} \otimes^L_{f_n^{-1}i_S^{-1} \cO_S}i_X^{-1} \Omega^\bullet_{X/S}) \\  \riso \cO_{S_n} \otimes^L_{i_S^{-1} \cO_S} R f_{n*} (i_X^{-1} \Omega^\bullet_{X/S}) 
\riso \cO_{S_n} \otimes^L_{i_S^{-1} \cO_S} i_S^{-1} R f_{*} \Omega^\bullet_{X/S}.
\end{gather*}
\end{proof}

\begin{defi}\label{dfnVi-Hbari}
For a smooth scheme $X$ over $S$ and $i\geq 0$, we define $$V^i_\dR(X/S)\coloneqq \Ker \left(H^i_\dR(X/S)\xrightarrow{\eqref{RlimadjdRcris}} H^i_\cris(X_0/S)\right).$$  We also consider the image $$\bar{H}^i_\dR(X/S)\coloneqq \mr{Im}\left(H^i_\dR(X/S)\xrightarrow{\eqref{RlimadjdRcris}} H^i_\cris(X_0/S)\right).$$ 
\end{defi}
Esnault--Kisin--Petrov question can be reformulated in the following form.
\begin{ques}\label{ques_EKP2}
    Let $X$ be a smooth scheme over $W$ of the form $Y\setminus D$, where $Y$ is smooth and proper over $W$ and $D$ is a relative normal crossing divisor over $W$. Does $V^i_{\dR}(X/W)$ vanish? If not, does $V^i_{\dR}(X/W)_\Q$ vanish?
\end{ques}

In the proper setting, the vanishing of $V^i_{\dR}$ follows from Berthelot's comparison.
\begin{prop}\label{prop_propercomparison}
If $X/S$ proper, then we have the following quasi-isomorphisms 
    $$ R \Gamma_{\dR} (X/S)  \otimes_{R} \widehat{R} \riso R \Gamma_{\dR} (X/S)^{\wedge} \riso       R \Gamma_\cris(X_0/S)$$ of complexes in $D^{\rmb}_{\coh} (\widehat{R})$.
\end{prop}
\begin{proof}By \stack{00MB}, $R\to \widehat{R}$ is flat. Since $X/S$ is proper smooth, then $R \Gamma_{\dR} (X/S)$ is a coherent complex of $R$-modules. 
Hence, $R \Gamma_{\dR} (X/S)  \otimes_{R} \widehat{R}$ is a coherent bounded complex of $\widehat{R}$-modules and is therefore derived $p$-complete\footnote{We remark that the notion of derived $p$-complete bounded complex of $\widehat{R}$-modules is equivalent to the notion quasi-coherence (or quasi-consistence) in the sense of Berthelot.} (see \stack{0A05}), which yields the first isomorphism (see \stack{0922}). Via Theorem \ref{thm_quasi-iso0} and Lemma \ref{comp-dR-Rim}, we conclude.
\end{proof}

\begin{coro}
\label{cor_propervanishing}If $X/S$ is proper and $p\in \mr{Jac}(R),$ then $V^i_\dR(X/S)=0$ and $\bar{H}^i_\dR(X/S)=H^i_\cris(X_0/S)$ for every $i\geq 0$.
\end{coro}
\begin{proof}
By \stack{00MC}, the morphism $R \to \widehat{R}$ is faithfully flat. Thus we get the desired result thanks to Proposition \ref{prop_propercomparison}. 
\end{proof}
\begin{rema}
  When $R$ is a $\Q$-algebra, then $X_0$ is empty and $V^i_\dR(X/S)=H^i_\mr{dR}(X/S).$ This is why in Corollary \ref{cor_propervanishing} we need the assumption $p\in \mr{Jac}(R)$.
\end{rema}

One of the main goals of this article is to study $V^i_\mr{dR}$ in the affine case. In this case, crystalline cohomology can be computed with the following explicit complexes.
\begin{coro}[Affine case]\label{cor_quasi-iso}
Let $A$ be a smooth $R$-algebra and write $A_n$ for its reduction modulo $p^{n}$. 
The following complexes of $\widehat{R}$-modules are quasi-isomorphic.
	\begin{enumerate}[(1)]
		\item The crystalline cohomology complex $R\Gamma_{\cris}(A_0/R)$.
        \item The derived limit $R\lim_n\left(\Omega^\bullet_{A_n/R_n}\right)$.
        \item The derived $p$-completion\footnote{See \stack{0BKF}.} of the algebraic de Rham complex $\Omega^\bullet_{A/{R}}$.
		\item The de Rham complex $\Omega^\bullet_{\widehat{A}/\widehat{R}}$.

\end{enumerate}
\end{coro}

\begin{proof}
Set $\fX\coloneqq \Spf(\widehat{A})$ and $X \coloneqq \Spec (A)$. We have     
$R \Gamma_{\dR} ( X_n /S_n)=\Omega^\bullet_{A_n/R_n}$, $R \Gamma_{\dR} ( X/S)=\Omega^\bullet_{A/R}$, and $R \Gamma_{\dR} ( \fX /\fS)=\Omega^\bullet_{\widehat{A}/\widehat{R}}.$ Then the result follows from Theorem \ref{thm_quasi-iso0} and Lemma \ref{comp-dR-Rim}.
\end{proof}

\begin{rema}
    Thanks to Corollary \ref{cor_quasi-iso}, we can say that in the affine setting the groups $V^i_\mr{dR}$ measure the extent to which the de Rham cohomology fails to be $p$-adically separated in a derived sense.
\end{rema}

\subsection{A fibre square}

Using Beauville--Laszlo gluing theorem, we construct a fundamental homotopy pullback square that relates de Rham cohomology, crystalline cohomology, and their rational versions. This diagram yields long exact sequences that help us analyse the torsion and divisibility properties of the comparison maps, leading to key structural insights about the kernel $V^i_{\dR}$.

	\begin{prop}\label{prop_pullbacksquare}
Let $X$ be a quasi-compact smooth $S$-scheme. We have the following homotopy pullback square
	\begin{equation}
		\label{diag_pullback}
\begin{tikzcd}[row sep=large, column sep=large]
	\mathrm{R\Gamma}_{\mathrm{dR}}(X/S) \arrow[r, "\eqref{RlimadjdRcris}"]  \arrow[d] 
	& \mathrm{R\Gamma}_{\mathrm{cris}}(X_0/S) \arrow[d] \\
	\mathrm{R\Gamma}_{\mathrm{dR}}(X/S)_\Q \arrow[r, "\eqref{RlimadjdRcris}"] 
	& \mathrm{R\Gamma}_{\mathrm{cris}}(X_0/S)_\Q.
\end{tikzcd}
	\end{equation}
in the stable $\infty$-category $D(R).$
	\end{prop}
\begin{proof}Thanks to Zariski descent for de Rham and crystalline cohomology

we can reduce to the affine setting. Suppose $X=\Spec(A)$ affine. Since $A\to \widehat{A}$ is flat, for every $i$, the morphism $$\Omega^i_A[p^\infty]\to\left(\Omega^i_A\otimes_A\widehat{A}\right)[p^\infty]=\Omega^i_{\widehat{A}/\widehat{R}}[p^\infty]$$ is an isomorphism. Thanks to  \stack{0BNW}, we have the exact sequence of $R$-modules $$0 \to \Omega^i_{A/R}\to \Omega^i_{\widehat{A}/\widehat{R}}\oplus \Omega^i_{A/R} [\tfrac{1}{p}] \to \Omega^i_{\widehat{A}/\widehat{R}}[\tfrac{1}{p}]
\to 0$$ for every $i\geq 0,$ compatible with differentials. This proves the result thanks to Corollary \ref{cor_quasi-iso}.
\end{proof}

\begin{coro}\label{cor_longexactsequence}
For a quasi-compact smooth $S$-scheme $X$ we have the long exact sequence
\begin{equation*}
    \dots \to H^i_\dR(X/S)\to H^i_\cris(X_0/S)\oplus H^i_\dR(X/S)_\Q\to H^i_\cris(X_0/S)_\Q\to H^{i+1}_\dR(X/S)\to\dots
\end{equation*} 
\end{coro}

\begin{rema}
It would be interesting to generalise \eqref{diag_pullback} to a Beauville--Laszlo gluing theorem for the associated \textit{analytic stacks} in the sense of Clausen--Scholze.
\end{rema}

\begin{coro}\label{cor_torsion}
	If $X$ is a quasi-compact smooth $S$-scheme, then we have the following exact sequences
    \begin{equation}\label{first_es}
        H^{i-1}_\dR(X/S)_\Q\to H^{i-1}_\cris(X_0/S)\otimes_{\Zp} \Q_p/\Z_p\to V^i_\dR(X/S)[p^\infty]\to 0,
    \end{equation}
    \begin{equation}\label{second_es}
       0\to V^i_\dR(X/S)[p^\infty]\to H^i_\dR(X/S)[p^\infty]\to H^i_\cris(X_0/S)[p^\infty]\to 0,
    \end{equation}
 \begin{equation}\label{third_es}
    0 \to  \bar{H}^i_{\mr{dR}}(X/S)/\mr{tors}\to \left(H^i_\cris(X_0/S)/\mr{tors}\right)\oplus  \bar{H}^i_{\mr{dR}}(X/S)_\Q \to  H^i_{\mr{cris}}(X_0/S)_\Q.
\end{equation}
\end{coro}

\begin{proof}
   We want to deduce the result from Corollary \ref{cor_longexactsequence}. By definition, we have that  $$\Ker \left( H^{i}_\dR(X)\to H^{i}_\cris(X_0)\oplus H^{i}_\dR(X)_\Q\right)=V^i_\dR(X)[p^\infty].$$ This gives the exact sequence $$ H^{i-1}_\cris(X_0)\oplus H^{i-1}_\dR(X)_\Q\to H^{i-1}_\cris(X_0)_\Q \to V^i_\dR(X)[p^\infty]\to 0,$$ 
   which induces \eqref{first_es}. The exact sequence \eqref{second_es} is induced instead by the exact sequence 
  $$0\to  V^i_\dR(X)[p^\infty]\to H^{i}_\dR(X)\to H^i_\cris(X_0)\oplus H^i_\dR(X)_\Q\to H^i_\cris(X_0)_\Q.$$ The same exact sequence produces \eqref{third_es}, modding out $H^i_\cris(X_0)\oplus H^i_\dR(X)_\Q$ by $H^i_\cris(X_0)[p^\infty]\oplus V^i_\mr{dR}(X)_\Q$.\end{proof}
\begin{exam}\label{exam_tors} Even assuming that $A$ is $p$-adically separated, it is easy to construct examples where $V^i_\mr{dR}(A)\neq 0.$ If $A\coloneqq \Z_{(p)}[x,y]/(x^2-1-py)$ over $\Z_{(p)}$, we have $H^0_\dR(A)=\Z_{(p)}$, while $$H^0_\cris(A/p)=H^0_\cris(\F_{p}[y]\times \F_{p}[y])=\Zp\oplus \Zp.$$ Thanks to \eqref{first_es}, this implies that $$V^1_\dR(A)[p^\infty]=\Qp/\Zp.$$
\end{exam}
\begin{rema}
    Note that even under assumptions on the existence of a good compactification of $X$, the $R$-module $\bar{H}^i_{\mr{dR}}(X)/\mr{tors}$
  
    is not in general a finite type $R$-module. A family of counterexamples is provided by the affine curves considered in \cite{ES20}. This phenomenon is related to the lack of finite type lattices in rigid cohomology. See Corollary \ref{cor_ES20} for a generalisation of their result.
\end{rema}

\begin{lemm}\label{lem_lowdegree} Let $X$ be a smooth scheme over $S$ such that $\bar{H}^{i-1}_\dR(X/S)\subseteq H^{i-1}_\cris(X_0/S)$ is $p$-adically dense. 
Then  the maps $H^i_{\mr{dR}}(X)[p^\infty]\to H^i_{\mr{cris}}(X_0)[p^\infty]$ and $V^i_\dR(X)\to V^i_\dR(X)_\Q$ are isomorphisms
and we have the short exact sequence \begin{equation}
	\label{lem_lowdegree1}
0\to H^i_{\mr{dR}}(X/S)[p^\infty]\oplus V^i_\dR{(X/S)}\to H^i_{\mr{dR}}(X/S)\to \bar{H}^i_{\mr{dR}}(X/S)/\mr{tors}\to 0.
\end{equation}
\end{lemm}

\begin{proof}
Since $\bar{H}^{i-1}_\dR(X)\subseteq H^{i-1}_\cris(X_0)$ is $p$-adically dense, then
 $$H^{i-1}_\cris(X_0)\oplus H^{i-1}_\dR(X)_\Q\to H^{i-1}_\cris(X_0)_\Q$$ is surjective.
 Using Corollary \ref{cor_longexactsequence}, this yields the exact sequence
 \begin{equation}\label{eq_lowdegree}
     0\to H^i_\dR(X)\to H^i_\cris(X_0)\oplus H^i_\dR(X)_\Q\to H^i_\cris(X_0)_\Q.
 \end{equation} 
Looking at the exact sequence \eqref{eq_lowdegree}, we get that the map
 $H^i_{\mr{dR}}(X)[p^\infty]\to H^i_{\mr{cris}}(X_0)[p^\infty]$  is bijective.
 This implies that $H^i_{\mr{dR}}(X)[p^\infty] \cap V_\dR^i(X) = 0$, i.e. 
 the map $V^i_\dR(X)\to V^i_\dR(X)_\Q$ is injective.
 Using again \eqref{eq_lowdegree}, we deduce that the map $V^i_\dR(X)\to V^i_\dR(X)_\Q$ is surjective and therefore bijective. To get the exactness of \eqref{eq_lowdegree}, it remains to prove that $V^i_\dR{(X)} $ is the kernel of  $H^i_{\mr{dR}}(X)/\mr{tors}\to \bar{H}^i_{\mr{dR}}(X)/\mr{tors}$, which follows from the isomorphism
 $H^i_{\mr{dR}}(X)[p^\infty]\riso H^i_{\mr{cris}}(X_0)[p^\infty]$.
\end{proof}

\begin{lemm}\label{lem_exrational}
	For $A = \mathbb{Z}_{(p)}\left[t, \tfrac{1}{1+pt}\right],$ then $V^1_\dR(A)$ is a $\Q$-vector space of dimension $1$.
\end{lemm}
\begin{proof} In this case, we have that $H^0_\dR(A)=\Z_{(p)}$ is dense in $H^0_\cris(A/p)=\Zp$. By Lemma \ref{lem_lowdegree}, it follows that $V^1_\dR(A)=V^1_\dR(A)_\Q$. In addition, since  $\Spec(A_\Q)\simeq\mathbb{G}_{m,\mathbb{Q}},$ we get an injective morphism $$V^1_\dR(A)\hookrightarrow H^1_\mr{dR}(\mathbb{G}_{m,\mathbb{Q}})=\Q.$$
	If we take the $1$-form
	$$\omega\coloneqq\frac{p}{1+pt}dt\in \Omega^1_A,$$ its primitive $\log(1+pt)$ is well-defined in the $p$-adic completion of $A$. Thus in ${H^1_{\mr{cris}}(A/p)}$ the differential $1$-form $\omega$ is exact. However, in the algebraic setting, $\log(1+pt)$ is not an element of $A$, thus $[\omega]$ is nontrivial in $H^1_{\mr{dR}}(A)$.
	We deduce that $V^1_\dR(A)$ is a non-trivial $\Q$-vector space, which implies the desired result.
\end{proof}

	\begin{defi}\label{def_infpdiv}
	Let $G$ be a commutative group. The subgroup of \emph{(strongly) $p$-divisible classes} in $G$, denoted by $G_{\mr{div}}$, is defined as the image of the projection morphism
	$$G^\sim \coloneqq\lim \left( \dots \xrightarrow{\cdot p}G\xrightarrow{\cdot p}G\xrightarrow{\cdot p} G \right)\to G,$$ 
    which sends  $( \dots, g_2,g_1,g_0)$ to $g_0$. Equivalently, $G_\mr{div}$ is the maximal subgroup of $G$ such that $pG_{\mr{div}}=G_{\mr{div}}$. It naturally sits in the exact sequence $$0\to {T}_p(G)\to G^\sim\to G_{\mr{div}}\to 0$$ where $T_p(G)$ is the Tate module $ \lim_{\cdot p} G[p^n]$. The group $G_{\mr{div}}$ should not be confused with $\bigcap_{n=1}^\infty p^nG,$ which is in general strictly bigger when $G$ has $p$-torsion.
    
\end{defi}

\begin{prop}[Esnault--Kisin--Petrov]\label{prop_kernel}
If $A$ is a smooth $R$-algebra and $i\in \bbN$, then $V^i_\dR(A/R)= H^i_{\mr{dR}}(A/R)_{\mr{div}}$.
\end{prop}

\begin{proof}The result was obtained by Esnault--Kisin--Petrov as a consequence of the fact that, in the affine case, crystalline cohomology is the derived $p$-completion of algebraic de Rham cohomology (see Corollary \ref{cor_quasi-iso}). We propose here a more direct proof with differential forms.

\spa

	Thanks to Corollary \ref{cor_quasi-iso}, we have that $V^i_\dR(A)=\Ker(H^i_\dR(A)\to  H^i_\dR(\widehat{A})).$	
For $j\geq 0$, write $\alpha\colon \Omega^j_{A} \to \Omega^j_{\widehat{A}}$ for the canonical map. Let $Z^i  (A) \coloneqq \mr{Ker} ( \Omega^i_A \to \Omega^{i+1}_A)$
and $[-] \colon Z^i   (A)\to H^i_\dR(A)$ be the canonical projection.

\spa
Let us prove that $V^i_\dR(A)\subseteq H^i_{\mr{dR}}(A)_{\mr{div}}$, which is equivalent to showing that $p V^i_\dR(A)=V^i_\dR(A).$ 
Let $\omega\in Z^i   (A)$ such that $[\omega]\in V^i_\dR(A)$. 
By the assumption, there exists $\psi\in \Omega^{i-1}_{\widehat{A}}$ such that $\alpha (\omega) =d\psi.$ Choose $N\geq0$ such that the torsion of $\Omega^{i}_{\widehat{A}}$ is killed by $p^N$. Let $\phi_1\in\Omega^{i-1}_{A}$, $\psi_1 \in \Omega^{i-1}_{\widehat{A}}$ such that 
	$\psi =\alpha(\phi_1) + p^{N+1}\psi_1$. 
	We get  $$\alpha(\omega-d\phi_1)= p^{N+1}d \psi_1\in p^{N+1} \Omega^i_{\widehat{A}}.$$ 
	Since  $\Omega^i_{A}\cap \alpha^{-1}( p^{N+1} \Omega^i_{\widehat{A}} )=p^{N+1} \Omega^i_{A}$, then there exists 
	$\omega '_1 \in \Omega^i_{A}$ such that $\omega-d\phi_1=p^{N+1}\omega '_1$. 
	Thanks to the fact that $\omega\in Z^i   (A)$, then $d (\omega '_1)$ is torsion and therefore $\omega_1\coloneqq p^{N}\omega '_1\in Z^i   (A)$ satisfies
	$[\omega]= p [\omega_1]$. Since	$p^{N+1}d\psi_1= p^{N+1}\alpha (\omega '_1)$, then $ d (p^{N}\psi_1)= \alpha ( \omega_1)$ and we get $[\omega_1]\in V^i_\dR(A)$.

    \spa
    
   We prove now the other inclusion. Let $\omega_0\in Z^i   (A)$ such that $[\omega_0]\in H^i_{\mr{dR}}(A)_\mr{div}$. By definition, there exists a sequence $\{\omega_n\}_{n\geq 1}$ of $Z^i   (A)$ with $[\omega_n]= p\,[\omega_{n+1}]$ for every $n\geq 0$. We choose $\psi_n\in \Omega^{i-1}_{A}$ such that $\omega_n-p\,\omega_{n+1}=d\psi_{n}$ and we write $\psi\coloneqq\sum_{n=0}^\infty p^n\, \alpha(\psi_n) \in \Omega^{i-1}_{\widehat{A}}$. 
   By construction, we have $d\psi=\alpha (\omega_0)$ ($\alpha$ is continuous), which shows that $[\omega_0]\in V^i_\dR(A)$. Hence, we are done.
\end{proof}
Proposition \ref{prop_kernel} in particular gives the following corollary, which allows descent for some special morphisms.
	\begin{coro}\label{coro_finiteflat}
	Let $X\to Y$ be a finite flat morphism of affine smooth $R$-schemes. If $V^i_\dR(X)=0,$ then $V^i_\dR(Y)=0.$
\end{coro}

\begin{proof}
	Let $d$ be the degree of $f\colon X\to Y$. The trace map $\Tr: H^i_{\mr{dR}}(X)\to H^i_{\mr{dR}}(Y)$ satisfies $\Tr\circ f=d\cdot\id$. The kernel of $H^i_{\mr{dR}}(Y)\to H^i_{\mr{dR}}(X)$ is then contained in the $d$-torsion of $H^i_{\mr{dR}}(Y)$. If $V^i_\dR(X)=0,$ then we deduce that $dV^i_\dR(Y)=0$. We conclude thanks to Proposition \ref{prop_kernel}.
\end{proof}

\subsection{The low-degree case}We prove the vanishing of the groups $V^i_\mr{dR}$ under suitable assumptions on the compactification and for $i\leq 1$. The proof involves careful analysis of the behaviour of algebraic de Rham cohomology near the boundary divisor. We reduce ourselves to the case when the scheme retracts onto the divisor at infinity. In this special situation, we prove an injectivity result for log crystalline cohomology for every $i\geq 0$ (Corollary \ref{cor_InjectivityRetraction}).
	\begin{defi}
	A \textit{smooth log pair} $(Y,D)$ over $S$ is a quasi-compact smooth $S$-scheme $Y$ together with the choice of a  relative to $S$ normal crossing divisor $D\subseteq Y$ (which is not assumed to be smooth). We consider $(Y,D)$ as a log scheme over $S$. We denote by $\Omega_{(Y,D)}^\bullet$ the log de Rham complex over $S$ and by $\mathrm{R\Gamma}_{\mathrm{dR}}(Y,D)$ the log de Rham cohomology complex over $S$. We write $j\colon X\hookrightarrow Y$ for the complement of $D$ in $Y$. We also write $\calH^{i}_{(Y,D)}\coloneqq H^i(\Omega^\bullet_{(Y,D)})$ and $\calH^{i}_{X}\coloneqq H^i(j_*\Omega^\bullet_{X})$ for the étale cohomology sheaves over $Y$.

    \spa
    
    If $B^\sharp$ is a ring endowed with a log structure $\alpha\colon M\to B,$ we denote by $(B^\sharp,f)$ the log structure induced by the pre-log structure $M\oplus \N \to B$ such that $(m,n)\mapsto f^n\alpha(m).$ If the log structure on $B^\sharp$ is trivial, we say that $(B^\sharp,f)$ is endowed with a \textit{principal log structure}.

\end{defi}

\begin{theo}\label{thm_lowdegree}Suppose that $p\in \mr{Jac}(R)$ and let $(Y,D)$ be a smooth log pair over $S$ with $Y$ proper, $D$ a strict normal crossing divisor, and $X
\coloneqq Y\setminus D$ affine. Then, the groups $V^i_\dR(X)$ vanish for $i\leq 1$. In addition, we have $$H^i_\dR(X)[p^\infty]=H^i_\cris(X_0)[p^\infty]$$ for $i\leq 1$.
\end{theo}
\begin{empt}
   We prove the result by comparing de Rham cohomology of $X$ with log de Rham cohomology of the pair $(Y,D)$ in low degree. We start with the following lemma. Let $A^\sharp$ be an $R$-algebra endowed with a log structure. We denote by $A^\sharp[t]$ the log tensor product $A^\sharp\otimes_{R}R[t]$ and by $A[t]^\sharp$ the log $R$-algebra $(A^\sharp[t],t).$ We also write $$\Omega_{A[[t]]^\sharp}^\bullet\coloneqq \lim_n\Omega_{A[t]^\sharp/t^n}^\bullet$$ and $\Omega_{A((t))^\sharp}^\bullet\coloneqq \Omega_{A[[t]]^\sharp}^\bullet\left[\tfrac{1}{t}\right].$
\end{empt}

\begin{lemm}\label{lem_infinitesimal_var}
	Let $A^\sharp$ be an $R$-algebra endowed with a log structure. The complex $$Q_{A^\sharp}^\bullet\coloneqq  \Omega_{A((t))^\sharp}^\bullet/\Omega_{A[[t]]^\sharp}^\bullet$$ is isomorphic to  $$\bigoplus_{n=1}^{\infty} \mathrm{Cocone}\left(\frac{1}{t^n}\Omega_{A^\sharp}^\bullet \longrightarrow \frac{\mr{dlog}(t)}{t^{n}}\wedge \Omega_{A^\sharp}^\bullet\right),$$ where $$\frac{1}{t^n}\Omega_{A^\sharp}^\bullet \longrightarrow \frac{\mr{dlog}(t)}{t^{n}}\wedge \Omega_{A^\sharp}^\bullet$$ sends $\frac{1}{t^n}\omega\mapsto-\frac{n\, \mr{dlog}(t)}{t^{n}}\wedge\omega.$ In particular, the group $H^i(Q_{A^\sharp}^\bullet)$ is a direct sum of $R$-modules of finite exponent for every $i$.
\end{lemm}
\begin{proof}
	Any element $\omega\in \Omega_{A((t))^\sharp}^i$ can be written as a Laurent series $$\omega = \sum_{n\leq N}  \tfrac{\omega_n}{t^n} + \sum_{n\leq N} \tfrac{\mr{dlog}(t)}{t^n} \wedge \eta_n,$$ with $\omega_n\in \Omega_{A^\sharp}^i ,\eta_n \in \Omega_{A^\sharp}^{i-1}$, and $N\in\mathbb{Z}$. Note that the subcomplex $\Omega_{A[[t]]^\sharp}^\bullet$ consists of those series with $N=0$. Consequently, the complex $\Omega_{A((t))^\sharp}^\bullet/\Omega_{A[[t]]^\sharp}^\bullet$  can be identified with the negative part of the Laurent series. We get $$Q_{A^\sharp}^i=\bigoplus_{n\geq 1}\left( \frac{1}{t^n}\Omega_{A^\sharp}^i \oplus \frac{\mr{dlog}(t)}{t^{n}}\wedge\Omega_{A^\sharp}^{i-1}\right).$$
	The differential $Q_{A^\sharp}^i\to Q_{A^\sharp}^{i+1}$ sends $\left(\frac{\omega}{t^n} , \frac{\mr{dlog}(t)}{t^{n}}\, \wedge\eta\right)$ to $$\left( \frac{d\omega}{t^n},\; -\frac{n\, \mr{dlog}(t)}{t^{n}}\wedge \omega + \frac{\mr{dlog}(t)}{t^{n}}\wedge(-d\eta) \right).$$ We deduce a natural isomorphism of complexes $$Q_{A^\sharp}^\bullet =  \bigoplus_{n\geq 1}\mathrm{Cocone} \left( \frac{1}{t^n}\Omega_{A^\sharp}^\bullet \to \frac{ \mr{dlog}(t)}{t^{n}}\wedge\Omega_{A^\sharp}^\bullet\right).$$ 
\end{proof}

\begin{lemm}\label{lem_local_var}
Let $B^\sharp$ be a smooth $R$-algebra endowed with a principal log structure and let $f$ be a nonzerodivisor of $B$ such that $B/f$ is smooth. Suppose that $B^\sharp\twoheadrightarrow A^\sharp\coloneqq B^\sharp\otimes_BB/f$ admits a section $\sigma$ as log $R$-algebras. There exists a natural exact sequence $$0\to \Omega^\bullet_{(B^\sharp,f)}\to \Omega^\bullet_{B_f^\sharp}\to Q_{A^\sharp}^\bullet\to 0$$ of differential graded $R$-algebras. Moreover, $\Omega^i_{B_f^\sharp}\to Q_{A^\sharp}^i$ admits an $R$-linear section as differential graded $R$-algebras, compatible with $\sigma$, which sends $1/t\mapsto 1/f$ and $\mr{dlog}(t)\mapsto \mr{dlog}(f).$

	\end{lemm}
\begin{proof} Let $g\in B$ be an element such that $B^\sharp=(B,g)$. Since $B^\sharp\twoheadrightarrow A^\sharp$ admits $\sigma$ as a section, we may assume that $g\in \sigma(A).$ We have an $A^\sharp$-algebra homomorphism 
$A [t]^\sharp \to (B^\sharp,f)$ given by $t\mapsto f$. Thanks to the fact that $f$ is a nonzerodivisor, this induces an isomorphism
$A [t] /t^{n}  \riso B / f^{n}$ of $R$-algebras for every $n\geq 0$. Since $g\in \sigma(B)$, this isomorphism upgrades to an isomorphism $A^\sharp [t] /t^{n}  \riso B^\sharp/ f^{n}$ of log $R$-algebras.  This yields the ring isomorphism 
$A^\sharp[[t]]\riso C^\sharp$, where $C^\sharp$ is the $f$-adic completion of $B^\sharp$. Identifying $C^\sharp$ with $A^\sharp[[t]]$, the $f$-adic completion map $B^\sharp \to A^\sharp[[t]]$ sends $f$ to $t$. 
Using \stack{0BNW}, we get an exact sequence 
$$0\to \Omega^i_{B^\sharp}\to \Omega^i_{A^\sharp[[t]]}\oplus \Omega^i_{B^\sharp_f}\to  \Omega^i_{A^\sharp((t))}\to  0$$ of $R$-modules. The pushout of the sequence with respect to $\Omega^i_{B^\sharp}
\subseteq \Omega^i_{(B^\sharp,f)}$ induces the exact sequence 
$$0\to \Omega^i_{(B^\sharp,f)}\to \Omega^i_{A[[t]]^\sharp}\oplus \Omega^i_{B_f^\sharp}\to  \Omega^i_{A((t))^\sharp}\to  0.$$
This gives the desired exact sequence. The section is well-defined thanks to Lemma \ref{lem_infinitesimal_var}.
\end{proof}

Lemma \ref{lem_local_var} already provides an injectivity result that we will need later in the text.

\begin{coro}\label{cor_InjectivityRetraction}
    Let $Y$ be a smooth affine $R$-scheme and $D$ a principal strict normal crossing divisor. Write $D=Z\cup D'$ with $Z$ irreducible and $D'$ a divisor not containing $Z$. Suppose that the inclusion $(Z,D'\cap Z)\hookrightarrow (Y,D)$ admits a retraction as smooth log pairs. We have a split injection 
    $$H^i_{\mr{log}\m \mr{cris}}((Y,D)/R)\to H^i_{\mr{log}\m \mr{cris}}((Y\setminus Z,D\setminus Z)/R).$$
\end{coro}
\begin{proof}
Thanks to Lemma \ref{lem_local_var}, there exists a decomposition of complexes $$\Omega^\bullet_{(Y\setminus Z,D\setminus Z)}\simeq \Omega^\bullet_{(Y,D)}\oplus {Q}_{(Z,D'\cap Z)} ^\bullet.$$ Taking $p$-adic completions, we deduce the desired result thanks to (the log variant of) Corollary \ref{cor_quasi-iso}. 
\end{proof}

\begin{empt}\label{nota_cohomology_sheaves}
    Let $Y$ be a smooth $S$-scheme and $D$ a relative strict normal crossing divisor. We write $Y^\sharp$ for the log scheme $(Y,D)$. Let $i\colon Z\hookrightarrow Y$ an irreducible component of $D$ and write $j\coloneqq X\hookrightarrow Y$ for $Y\setminus Z.$ We write $Z^\sharp$ and $X^\sharp$ for the log structures induced by the one of $Y$. The following proposition is a generalisation of \cite[Thm. 2.2.5]{AbbottKedlayaRoe}.
\end{empt}
\begin{prop}\label{prop_cohomologysheaves_var}
 The natural morphism of étale sheaves $\Omega^i_{Y^\sharp} \to j_{*}\Omega^i_{X^\sharp}$
is injective for $i\geq 0$ and, étale locally on $Y$, the $R$-module $\Gamma(Y,j_{*}\calH^i_{X^\sharp}/\calH^i_{Y^\sharp})$ is a direct sum of $R$-modules of finite exponent. In addition, $\calH^0_{X^\sharp}/\calH^0_{Y^\sharp}$ is killed by a fixed power of $p$.
\end{prop}
\begin{proof}
    (i) We reduce ourselves to the situation of Lemma \ref{lem_local_var}. It is enough to prove the statement for the global sections of the sheaves over a cofinal family of étale covers of $Y$.
Hence, we can suppose there exists   an \'etale morphism  of affine $R$-schemes $\phi \colon Y \to \bbA ^d_R$
such that $Z  = V (t_1)$, where  $t_1,\dots , t_d$ are the corresponding local coordinates. We denote by  $B \coloneqq  \Gamma (Y ,\cO_Y)$ and $A \coloneqq  B / t_1 B$. 

 \spa

Let $\overline{t}_2, \dots, \overline{t}_d$ be the global sections of $\cO_{Z}$ induced by $t_2, \dots, t_d$ and let  $\phi_1\colon Z \to \bbA^{d-1}_R$ be the \'etale morphism induced by $\overline{t}_2, \dots, \overline{t}_d$.
We also consider the étale morphism
$$\id \times \phi_1\colon  \bbA ^1_R  \times Z \to \bbA_R^1\times  \bbA ^{d-1}_R=\bbA_R^d$$
and the closed immersion $$i ''\colon  Z=\{0\}\times Z \hookrightarrow \bbA ^1_R \times Z.$$
By putting
 $Y'\coloneqq (\bbA_R ^1  \times Z)\times_{\bbA_R^d}Y$,
 we get the closed immersion 
 $i' \colon Z \hookrightarrow 
 Y '$ making the following diagram commute
  \begin{equation}
  \label{diagphiprime}
 \xymatrix @R=0.8cm {
 {Z}
 \ar@/^0,5cm/[rr]^-{i''}
 \ar@{.>} [r]_-{i'}
 \ar[rd]_-{i}
 &
  {Y'}
  \ar[d]^-{f}
  \ar[r]_-{\phi '} 
  \ar@{}[rd]^-{}|\square
  &
  {\bbA_R ^1  \times Z}
  \ar[d]^-{\id \times \phi_1}
    \\
  &
  Y
 \ar[r]
   \ar[r]_-{\phi} 
  &
  \bbA_R^{d},}
 \end{equation}
 where the square is cartesian, where $f$ and $\phi'$ are the canonical projections.  Putting $Z'\coloneqq Z \times_{Y}  Y'$,
 we get a section $( \id, i ') \colon Z \hookrightarrow Z'$  of the \'{e}tale projection $ Z'\to Z$. 
  Hence, by using \cite[Cor. I.5.3]{sga1},   after shrinking $Y'$ while fixing $Y$ we  may assume that $Z\xrightarrow{\sim}Z\times_{Y}Y'$.
 Let $g\colon Y'\rightarrow Z$ be the canonical morphism, i.e. the composition of  $\phi'$ with the projection $\bbA_R ^1  \times Z \to Z$. By construction, $i '$ is  a section of $g$ compatible with the log structures on $Z$ and $Y'$. 
 \spa

 (ii)     We can then use Lemma \ref{lem_local_var} to deduce that the inclusion of complexes of $R$-modules \begin{equation}\label{eq_complexes_local}
        \Gamma(Y,\Omega^\bullet_{Y^\sharp})\to \Gamma(Y,j_{*}\Omega^\bullet_{Z^\sharp})
    \end{equation}admits a retraction. This shows that $H^i_\dR(Y^\sharp)\to H^i_\dR(Z^\sharp)$ is injective. Combining the description of the cokernel of \eqref{eq_complexes_local} given by Lemma \ref{lem_local_var} with Lemma \ref{lem_infinitesimal_var} we also deduce that the cokernel of $H^i_\dR(Y^\sharp)\to H^i_\dR(Z^\sharp)$ is a direct sum of $R$-modules of finite exponent. If $N$ is a positive integer such that $p^NR[p^\infty]=0,$ then for every flat $R$-algebra $A,$ we have that $p^NA[p^\infty]=0.$  Using again the description of cokernel of \eqref{eq_complexes_local}, this implies that $H^0_\dR(Y^\sharp)\to H^0_\dR(Z^\sharp)$ is killed by the multiplication by $p^N.$
\end{proof}

\subsubsection{Proof of Theorem \ref{thm_lowdegree}}

Thanks to \eqref{second_es}, the vanishing of $V^i_\dR(X)$ implies that $H^i_\dR(X)[p^\infty]=H^i_\cris(X_0)[p^\infty]$. Thus, it is enough to prove that $V^i_\dR(X)=0$ for $i\leq 1$. By Proposition \ref{prop_kernel}, this is equivalent to showing that $H^i_\dR(X)_{\mr{div}}=0.$ Thanks to \cite[Prop. 2.2.8]{AbbottKedlayaRoe}, we deduce that the groups $H^{i}_\dR(Y,D)$ are finite $R$-modules for every $i$, hence $p$-adically separated. This shows that $H^{i}_\dR(Y,D)_\mr{div}=0$ for every $i$. The goal is to prove the result by comparing $H^{i}_\dR(Y,D)$ and $H^i_\dR(X)$ in low degree.

\spa

By Proposition \ref{prop_cohomologysheaves_var}, the sheaf $\calH^0_X/\calH^0_{(Y,D)}$ is killed by a fixed power of $p$. Therefore, any element in $H^0_\dR(X)_\mr{div}$ comes from $H^0_\dR(Y,D)_\mr{div}=0.$ We deduce the vanishing of $V^0_\dR(X).$ To prove the vanishing of $V^1_\dR(X)$, we look at the second page of the étale conjugate spectral sequence for the hypercohomology of $\Omega^\bullet_{(Y,D)}$ and $j_*\Omega^\bullet_X$. We get the following commutative diagram with exact rows and coloumns
	\begin{equation*}\begin{tikzcd}
    & 0\arrow[d]& 0\arrow[d]\\
    H^0_\et(Y,\calH^0_X/\calH^0_{(Y,D)})\arrow[r] &H^1_\et(Y,\calH^0_{(Y,D)}) \arrow[d]\arrow[r] & H^1_\et(Y,\calH^0_X)\arrow[d]\arrow[r] & H^1_\et(Y,\calH^0_X/\calH^0_{(Y,D)}) \\
    	 &H^1_\dR(Y,D) \arrow[d]\arrow[r] & H^1_\dR(X)\arrow[d] & \\
			0\arrow[r] &H^0_\et(Y,\calH^1_{(Y,D)}) \arrow[d,"d_2"]\arrow[r] & H^0_\et(Y,\calH^1_X)\arrow[d,"d_2"]\arrow[r] & H^0_\et(Y,\calH^1_X/\calH^1_{(Y,D)}) \\
			H^1_\et(Y,\calH^0_X/\calH^0_{(Y,D)})\arrow[r]& H^2_\et(Y,\calH^0_{(Y,D)})\arrow[r]& H^2_\et(Y,\calH^0_{X}).&
		\end{tikzcd}
	\end{equation*}

Let $(\dots,x_2,x_1,x_0)\in H^1_\dR(X)^\sim$ (see Notation \ref{def_infpdiv}). Then, it induces a sequence $(\dots,\bar{x}_2,\bar{x}_1,\bar{x}_0)\in H^0_\et(Y,\calH^1_X)^\sim$ such that $d_2(\bar{x}_n)=0$ for every $n.$ By virtue of Proposition \ref{prop_cohomologysheaves_var}, we deduce that $H^0_\et(Y,\calH^1_X/\calH^1_{(Y,D)})^\sim=0$. Thus $(\dots,\bar{x}_2,\bar{x}_1,\bar{x}_0)\in H^0_\et(Y,\calH^1_X)^\sim$ lifts to a unique sequence $$(\dots,y_2,y_1,y_0)\in H^0_\et(Y,\calH^1_{(Y,D)})^\sim.$$ From the vanishing $$p^N\left(\calH^0_{X}/\calH^0_{(Y,D)}\right)=0,$$ we deduce that $p^N(d_2(y_n))=0$ for every $n\geq 0.$ On the other hand, by definition, $p^N(d_2(y_{n+N}))=d_2(p^Ny_{n+N})=d_2(y_{n})$, thus $d_2(y_n)=0$ for $n\geq 0$.

\spa

Since $H^1_\dR(Y,D)$ is a finite $R$-module, any quotient is $p$-adically separated. Thus $\bar{x}_n=y_n=0$ for every $n\geq 0.$ This shows that $(\dots,x_2,x_1,x_0)\in H^1_\et(Y,\calH^0_X)^\sim.$ Since $p^NH^1_\et(Y, \calH^0_{X}/\calH^0_{(Y,D)})=0,$ we deduce that each $x_n$ is in the image of $H^1_\et(Y,\calH^0_{(Y,D)})\to H^1_\et(Y,\calH^0_X).$
More precisely, since $p^NH^0_\et(Y, \calH^0_{X}/\calH^0_{(Y,D)})=0,$ 
then $(\dots,x_2,x_1,x_0)$  comes from an element of $H^1_\et(Y,\calH^0_{(Y,D)})^\sim$.
Moreover, $H^1_\et(Y,\calH^0_{(Y,D)}),$ being an $R$-submodule of $H^1_\dR(Y,D),$ is a finite $R$-module. 
Hence, the group $H^1_\et(Y,\calH^0_{(Y,D)})^\sim$ vanishes. This yields the desired result. \qed

\section{Rigid versus convergent cohomology}\label{Sec_RigidVersusConvergent}
In this section, we transition to the rational setting and we explain the relation between rationalised algebraic de Rham cohomology and rigid cohomology. Furthermore, we prove a positive result: the comparison map is injective in cohomological degree $1$ for overconvergent $F$-isocrystals (Theorem \ref{thm_RC}). This is a consequence of Kedlaya's full faithfulness theorem.
\subsection{Rigid comparison}
We define the rigid analogues $V^i_{\mr{rig}}$ of our comparison kernels and show that they coincide with the rational versions of $V^i_\mr{dR}$ in the presence of good compactifications. 
\begin{nota}
    For a local ring $R$, let $\mathrm{CDVR}^\star(R)$ be the category whose objects are complete discretely valued rings $\cV$ of mixed characteristic $(0,p)$ with perfect residue field endowed with a morphism of local rings $R \to \cV$. A morphism $\cV \to\cV'$ is a morphism of local $R$-algebras $\cV \to \cV'$.

\end{nota}

Let $\calV$ be an object of $\mathrm{CDVR}^\star(\Z_p)$ with residue field $k$ and fraction field $K$. 

\begin{defi}
    For a smooth scheme $X_0$ over $k$, we define $$V^i_\mr{rig}(X_0/K)\coloneqq \mr{Ker} \left( H^i_{\mr{rig}}(X_0/K)\to H^i_{\mr{conv}}(X_0/K)\right).$$ If $\calM^\dagger$ is an overconvergent isocrystal over $X_0$ we also write $$V^i_\mr{rig}(X_0,\calM^\dagger)\coloneqq \mr{Ker} \left( H^i_{\mr{rig}}(X_0,\calM^\dagger)\to H^i_{\mr{conv}}(X_0,\calM)\right).$$
\end{defi}
\begin{prop}\label{prop_dRC-RC}
	Let $(Y,D)$ be a smooth log pair over $\calV$ with $Y$ proper and write $X\subseteq Y$ for the complement of $D$. We have the following natural commutative square where the vertical arrows are isomorphisms
	
		\begin{equation*}
		\begin{tikzcd}[row sep=large, column sep=large]
			H^i_{\dR}(X)_K \arrow[r] \arrow["\simeq",d] 
			& H^i_{\dR}(\widehat{X})_K \arrow["\simeq",d] \\
			 H^i_{\mathrm{rig}}(X_0/K) \arrow[r] 
			& H^i_{\mathrm{conv}}(X_0/K).
		\end{tikzcd}
	\end{equation*}
	
	In particular, for every $i \geq 0$, we have $V^i_\dR(X/\calV)_K=V^i_\mr{rig}(X_0/K)$. 
\end{prop}

\begin{proof}
	Thanks to formal GAGA, the log de Rham cohomology of $(Y,D)$ over $\calV$ is isomorphic to log crystalline cohomology of $(Y_0,D_0)$ over $\calV$. After inverting $p$, we deduce from \cite{BC94} (as explained in \cite[Prop. 4.2]{LST01}) that
	$$H^i_{\dR}(Y,D)_K = H^i_{\mathrm{rig}}(X_0/K),$$ which in turn implies by \cite[Thm. II.3.13]{Del70} that
$$H^i_{\dR}(X)_K=  H^i_{\mathrm{rig}}(X_0/K).$$
On the other hand, since $X_0$ is smooth, we have 
$$H^i_{\dR}(\widehat{X})_K=H^i_{\mathrm{conv}}(X_0/K).$$

This proves the desired result.
\end{proof}
\begin{defi}\label{def_ovcrys}Suppose $X_0=\Spec(A_0)$ affine with a $p$-adically weakly complete flat lift $A^\dagger.$ Write $\Omega^\bullet_{A^\dagger/\calV}$ for the de Rham complex of $A^\dagger$ with continuous differentials. An \textit{overconvergent flat connection} over $X_0/\cV$ is a locally free $A^\dag$-module $E^\dag$ 
    endowed with a flat connection $$\nabla_{E^{\dag}}\colon E^{\dag} \to E^{\dag} \otimes_{A^\dag} \Omega^1_{A^\dag}.$$ Write $\nabla_{E^\dag}^i \colon E \otimes_{A^{\dag}} \Omega_{A^{\dag}}^{i} \to 
    E \otimes_{A^{\dag}} \Omega_{A^{\dag}}^{i+1}$ for the maps induced by $\nabla_{E^\dag}$. We write $$H^i_\mr{MW}(X_0,E^\dag)\coloneqq \Ker \nabla_{E^\dag}^{i}/\im \nabla_{E^\dag}^{i-1}$$ for the $i$th \textit{integral Monsky--Washnitzer cohomolog}y of $E^\dag$. Tensoring $E^\dag$ with $K$, we get an overconvergent isocrystal, denoted by $(E^{\dag}_K , \nabla_{E_K^{\dag}})$ and we have $$H^i_\mr{MW}(X_0,E^\dag)_K  =H^i_\mr{rig}(X_0,E_K^\dag).$$  We denote by $E$ the flat connection $\widehat{E}^\dag$ induced on the $p$-adic completion $\widehat{A}$ of $A^\dagger.$
\end{defi}
\begin{prop}\label{prop_pullbacksquareMW}
If $X_0$ is an affine smooth $k$-scheme, $\mathfrak{X}/\cV$ is a $p$-adic affine flat formal lift, and $E^\dag$ is an overconvergent flat connection over $X_0/\calV$, we have the following distinguished triangle
    $$R\Gamma_{\mr{MW}}(X_0,E^\dag)\to R\Gamma_\mr{dR}(\mathfrak{X},E)_K\oplus R\Gamma_\mr{rig}(X_0,E^\dagger_K)\to R\Gamma_\mr{dR}(\mathfrak{X},E)_K.$$
\end{prop}
\begin{proof}
  Write $X_0=\Spec(A_0)$ and let $A^\dagger/\calV$ be a $p$-adically weakly complete flat lift of $A_0$.  By \cite{Ful69}, the ring $A^\dagger$ is Noetherian. Arguing as in Proposition \ref{prop_pullbacksquare}, we deduce thanks to  \stack{0BNW}, that we have the exact sequence of complexes $$0 \to E^\dagger\otimes \Omega^\bullet_{A^\dagger/\calV}\to \left(E\otimes\Omega^\bullet_{\widehat{A}/\calV}\right)\oplus \left(E^\dagger\otimes\Omega^\bullet_{A^\dagger/\calV}\right) \left[\tfrac{1}{p}\right] \to \left(E\otimes\Omega^\bullet_{\widehat{A}/\calV}\right)\left[\tfrac{1}{p}\right]
\to 0.$$ This yields the desired result.
\end{proof}

\begin{theo}\label{thm_comparisonMWdR}
    Let $(Y,D)$ be a smooth log pair over $\calV$ with $Y$ proper and $X\coloneqq Y\setminus D$ affine. The weak completion induces a quasi-isomorphism $$R\Gamma_\mr{dR}(X/\calV)\xrightarrow{\sim} R\Gamma_{\mr{MW}}(X_0/\calV).$$
\end{theo}
\begin{proof}
    This follows from the combination of Proposition \ref{prop_pullbacksquare}, Proposition \ref{prop_dRC-RC}, and Proposition \ref{prop_pullbacksquareMW}.
\end{proof}

\subsection{The low-degree case}

We now prove the vanishing of $V^i_\mr{rig}$ in low-degree using Kedlaya's full faithfulness theorem. This will be generalised to more general overholonomic $D$-modules in Corollary \ref{cor_Hol-ff}.

\begin{theo}\label{thm_RC}
If $X_0$ is a smooth $k$-scheme and $\calM^\dagger$ is an overconvergent $F$-isocrystal over $X_0$, then $V^1_{\mr{rig}}(X_0,\calM^\dagger)=0$.
\end{theo}

\begin{proof}
By \cite[Prop. 1.1.3]{CLS99}, we have that $$H^1_{\mr{rig}}(X_0,\calM^\dagger)=\mr{Ext}^1(\calO^\dag,\calM^\dag),$$ where the $\mr{Ext}^1$ is taken in the category of overconvergent isocrystals over $X_0$. The injectivity of $$H^1_{\mr{rig}}(X_0,\calM^\dagger)\to H^1_{\mr{conv}}(X_0,\calM)$$ is then equivalent to the fact that every non-trivial extension $$0\to\calM^\dag\to\calE^\dag\to\calO^\dag\to 0 $$ of overconvergent isocrystals is sent, after $p$-adic completion, to a non-trivial extension in the category of convergent isocrystals over $X_0$. This follows from \cite{Ked04}, as proven in \cite[Cor. 5.7]{DE20}. 
\end{proof}

\begin{empt}
    Let $X_0$ be a smooth scheme over $k$ and $\sigma,\tau$ two edge-types with $\sigma\leq \tau$ (see \cite{EdgedCrystalline}). One can similarly look at the natural morphism 
$$H^i_{\sigma_\infty\m \cris}(X_0/W)\to H^i_{\tau_\infty\m \cris}(X_0/W).$$ 
\end{empt}

\begin{lemm}
    The comparison morphism $$H^1_{\theta^1_\infty\m\cris}(\mathbb{A}^1_{\Fp}/\Zp)_{\Qp}\to H^1_{\theta^2_\infty\m\cris}(\mathbb{A}^1_{\Fp}/\Zp)_{\Qp}$$ is not injective. 
\end{lemm}
\begin{proof}
The class represented by $$\omega\coloneqq \sum_{i=0}^\infty p^{3i}t^{p^{2i}}\mr{dlog}(t)$$ in $H^1_{\theta^1_\infty\m\cris}(\mathbb{A}^1_{\Fp}/\Zp)_{\Qp}$ admits as a primitive the series $\sum_{i=0}^\infty p^{i}t^{p^{2i}}$, which is $\theta^2_\infty$-edged but is not $\theta^1_\infty$-edged. This class is not even a torsion class. Note that $pF^2(\omega)\equiv\omega$, hence it is a class of negative slope (cf. Theorem \ref{thm_counterexample}). 
\end{proof}

\section{Slope bounds for crystalline cohomology}
\label{Sec_SlopeBounds}
\subsection{The Nygaard filtration}
\label{SubSec_Nygaard}
To further study the $K$-vector space $V^i_\mr{rig}$ we use the \textit{$F$-gauge structure} of crystalline cohomology, as defined in \cite{FJ21} (see also \cite[Ex. 3.4.7]{Bha22}). We work with \textit{unit-root} coefficients for more generality. 

\begin{defi}\label{Sec_UR}
A \textit{unit-root $F$-crystal} is a locally free $F$-crystal of finite rank over the absolute crystalline site of $X_0$ such that the Frobenius structure is an isomorphism. Thanks to \cite[Thm.2.2]{crew_F-iso_p-adic-rep}, they correspond to locally free \textit{lisse $\Zp$-sheaves} (as in  \stack{03UM}). We denote by $\Zp$ the trivial rank $1$ unit-root $F$-crystal.  
\end{defi}


\begin{empt}\label{sec_FGauge}
Suppose that $X_0$ is a smooth affine scheme over a perfect field $k$ of characteristic $p$. Write $\fX=\Spf(\widehat{A})$ for a formal smooth lift over $W$ and write $F$ for a Frobenius lift $F\colon \fX\to \fX.$ For a unit-root $F$-crystal $E$ over $X_0$, we write $$\mr{dR}_E^\bullet\coloneqq E\otimes_{\widehat{A}}\Omega^\bullet_{\widehat{A}/W} $$ for the de Rham complex of $E$ associated to the formal lift $\widehat{A}.$ It is endowed with a $\sigma$-linear Frobenius structure $F\colon \mr{dR}_E^\bullet\to \mr{dR}_E^\bullet.$  Thanks to \cite[Prop. 8.7]{BMS}, the \textit{Nygaard filtration} of crystalline cohomology, defined\footnote{Strictly speaking, Nygaard defined the filtration when $E=\Zp$ is the trivial $F$-crystal. In \S \ref{Sec_dRWUnitRoot} we explain how his construction naturally extends to more general unit-root $F$-crystals.} in \cite{Nyg82}, can be reinterpreted as the filtration on the complex $\mr{dR}_E^\bullet$, defined by $$\mr{dR}^{\bullet,i}_E\coloneqq p^{(i-\bullet)_+}\mr{dR}^\bullet_E = p^iE\to p^{i-1}E\otimes_{\widehat{A}} \Omega^1_{\widehat{A}/W}\to\dots \to p^0E\otimes_{\widehat{A}}\Omega^i_{\widehat{A}/W}\to E\otimes_{\widehat{A}}\Omega^{i+1}_{\widehat{A}/W}\to \dots$$ for $i\in \Z,$ where $n_+\coloneq\max\{n,0\}.$ The complex admits a divided Frobenius morphism $\varphi_i\colon \mr{dR}^{\bullet,i}_E\to\mr{dR}_E^\bullet$, defined as $F/p^i$ for $i\geq 0$, and a canonical morphism $\iota\colon \mr{dR}^{\bullet,i}_E\to \mr{dR}_E^\bullet$.  In addition, there are maps $t\colon \mr{dR}^{\bullet,i+1}_E\to \mr{dR}^{\bullet,i}_E$ induced by the identity on $\mr{dR}_E^\bullet$ and $u\colon \mr{dR}^{\bullet,i}_E\to \mr{dR}^{\bullet,i+1}_E$ induced by the multiplication by $p$. We get the following diagram
$$\xymatrix{
  \cdots \ar@<0.5ex>[r]^{t} 
  & \mr{dR}^{\bullet,2}_E \ar@<0.5ex>[r]^{t} \ar@<0.5ex>[l]^{u}
  & \mr{dR}^{\bullet,1}_E \ar@<0.5ex>[r]^{t} \ar@<0.5ex>[l]^{u}
  & \mr{dR}_E^{\bullet,0} \ar@<0.5ex>[r]^{t} \ar@<0.5ex>[l]^{u}
  & \mr{dR}_E^{\bullet,-1}  \ar@<0.5ex>[r]^{t} \ar@<0.5ex>[l]^{u}
  & \mr{dR}_E^{\bullet,-2}  \ar@<0.5ex>[r]^{t} \ar@<0.5ex>[l]^{u}
  & \cdots \ar@<0.5ex>[l]^{u},
}$$
where $tu=ut=p$ and for non-positive $i,$ we have $\mr{dR}_E^{\bullet,i}=\mr{dR}_E^{\bullet}$.

We write $\widetilde{\mr{dR}}^{\bullet}_E$ for the complex $\varinjlim_{u}\mr{dR}^{\bullet,i}_E$, which by construction is endowed with a natural injective map $\widetilde{\mr{dR}}^{\bullet}_E\hookrightarrow \mr{dR}_{E_K}^\bullet$ via the identifications
$$\widetilde{\mr{dR}}^j_E=p^{-j}E\otimes_{\widehat{A}}\Omega^j_{\widehat{A}} \subseteq E_K\otimes_{\widehat{A}}\Omega^j_{\widehat{A}} $$ for every $j\in \bbN$. The morphisms $\varphi_i\colon \mr{dR}^{\bullet,i}_E\to \mr{dR}^{\bullet}_E$ define a morphism of complexes
$F\colon \widetilde{\mr{dR}}^{\bullet}_E\to\mr{dR}^{\bullet}_E$ extending the Frobenius action $F\colon\mr{dR}^{\bullet}_E \to \mr{dR}^{\bullet}_E$. The $F$-gauge property follows from the following well-known result (see \cite[Lem. 3.2.3]{DK17} or \cite[\S 6]{FJ21} for a reinterpretation with the syntomic site).

\end{empt}

\begin{prop}[Mazur, Berthelot--Ogus]\label{thm_Mazur}
 With notation as in \S \ref{sec_FGauge}, the morphism $$F \colon \widetilde{\mr{dR}}^{\bullet}_E\to\mr{dR}^{\bullet}_E$$ is a quasi-isomorphism. In addition, for every $i\geq 0$, the morphism $$F\colon \tau_{\leq i} \mr{dR}^{\bullet,i}_E\to\tau_{\leq i} p^i\mr{dR}^{\bullet}_E$$ is a quasi-isomorphism.
\end{prop}
\begin{proof}
  After passing to a finite étale cover of $\widehat{A}$, we may assume that $E/p$ is trivial. In this case, as noticed in \cite{Maz73}, the morphism $H^i (\varphi/p\varphi)$ decomposes a direct sum of isomorphisms $$\underline{p}^iC^{-1}\colon 
\widetilde{\Omega}^{i}_{\widehat{A}}/ p \widetilde{\Omega}^{i}_{\widehat{A}} \xrightarrow{\underline{p}^i}  \Omega^i_{A/pA} \xrightarrow{C^{-1}} H^i (\Omega^{\bullet}_{A/pA}),$$ where $C^{-1}$ is the inverse of the Cartier operator
and $\widetilde{\Omega}^{i}_{\widehat{A}} \coloneqq \widetilde{\mr{dR}}^{i}_{\widehat{A}}= p^{-i} \Omega^i_{\widehat{A}}$ and $\underline{p}^i$ is the isomorphism induced by the multiplication by $p^i$. This shows that $\varphi$ is a quasi-isomorphism. For the second part, since  $\Omega^{j}_{\widehat{A}}$ are $p$-torsion free, then we get the quasi-isomorphism  $F\colon p^i p^{-\bullet}\mr{dR}_E^\bullet \to p^i \mr{dR}_E^\bullet$ induced by $F$,
since $ \tau_{\leq i} ( p^{(i-\bullet)_+}\mr{dR}_E^\bullet) = \tau_{\leq i} ( p^i p^{-\bullet}\mr{dR}_E^\bullet )$ we are done.
\end{proof}

\subsection{A fractional variant}\label{Sec_FractionalNygaard}In this section, we introduce the notion of \textit{fractional $p$-adic Tate twists} $\Zp(\gamma)$ and their associated Nygaard filtration. Let $X_0$ be an affine smooth $k$-scheme and $E$ a unit-root $F$-crystal over $X_0$. Fix $r\in \Z$ and $s\in \Z_{>0}$ such that $\mr{gcd}(r,s)=1$ and write $\Lambda^s\coloneqq \Zp[\pi]/(\pi^s-p)$. 
\begin{empt}\label{fracNygaardfilt}
For a unit-root $F$-crystal $E$ (over $W$) we define complexes \begin{equation}
    \label{dfn-dREs}\mr{dR}_{E,s}^\bullet\coloneqq\mr{dR}_E^\bullet\otimes_{\Zp} \Lambda^s
\end{equation} endowed  with 
the action of Frobenius (still denoted by $F$) given by $F\otimes \id$ where $F$ is the one on $\mr{dR}_E^\bullet$
and with the $\Z$-filtration

$$\mr{dR}^{\bullet,r}_{E,s}\coloneqq \pi^{a_s(\bullet,r)}\mr{dR}^\bullet_{E,s}$$

where $a_s(i,r)\coloneqq (r-is)_+$. Note that if $r=a+bs$ with $0\leq a <s$ and $b\in \Z$ then we have the compatibility
\begin{equation}\label{eq_FractionalCompatibility}
    \mr{dR}^{\bullet,r}_{E,s}=\left(\bigoplus_{i=0}^{a-1}\pi^i\mr{dR}^{\bullet,b+1}_{E}\right)\oplus \left(\bigoplus_{i=a}^{s-1}\pi^i\mr{dR}^{\bullet,b}_{E}\right)
\end{equation}

The complex then admits a divided Frobenius morphism $\varphi_{r,s}\colon \mr{dR}^{\bullet,r}_{E,s}\to\mr{dR}_{E,s}^\bullet$, defined as $F/\pi^r$, and a canonical morphism $\iota\colon \mr{dR}^{\bullet,r}_{E,s}\to \mr{dR}_{E,s}^\bullet$.  In addition, there are maps $t\colon \mr{dR}^{\bullet,r+1}_{E,s}\to \mr{dR}^{\bullet,r}_{E,s}$ induced by the identity on $\mr{dR}_{E,s}^\bullet$ and $u\colon \mr{dR}^{\bullet,r}_{E,s}\to \mr{dR}^{\bullet,r+1}_{E,s}$ induced by the multiplication by $\pi$. We get the following diagram
$$\xymatrix{
  \cdots \ar@<0.5ex>[r]^{t} 
  & \mr{dR}^{\bullet,2}_{E,s} \ar@<0.5ex>[r]^{t} \ar@<0.5ex>[l]^{u}
  & \mr{dR}^{\bullet,1}_{E,s}\ar@<0.5ex>[r]^{t} \ar@<0.5ex>[l]^{u}
  & \mr{dR}_{E,s}^{\bullet,0} \ar@<0.5ex>[r]^{t} \ar@<0.5ex>[l]^{u}
  & \mr{dR}_{E,s}^{\bullet,-1}  \ar@<0.5ex>[r]^{t} \ar@<0.5ex>[l]^{u}
 & \mr{dR}_{E,s}^{\bullet,-2}  \ar@<0.5ex>[r]^{t} \ar@<0.5ex>[l]^{u} & \cdots \ar@<0.5ex>[l]^{u},
}$$
where $tu=ut=\pi$. Again, for $i\leq 0$ we have $\mr{dR}_{E,s}^{\bullet,i}=\mr{dR}_{E,s}^{\bullet}$. 
\end{empt}
We write $\widetilde{\mr{dR}}^{\bullet}_{E,s}$ for the complex $\varinjlim_{u}\mr{dR}^{\bullet,r}_{E,s}$ and $F\colon \widetilde{\mr{dR}}^{\bullet}_{E,s}\to\mr{dR}^{\bullet}_{E,s}$ for the map induced by $\varphi_{r,s}\colon \mr{dR}^{\bullet,r}_{E,s}\to \mr{dR}^{\bullet}_{E,s}$.
\begin{prop}\label{thm_MazurFractional}
 The morphism $$F \colon \widetilde{\mr{dR}}^{\bullet}_{E,s}\to\mr{dR}^{\bullet}_{E,s}$$ is a quasi-isomorphism. In addition, the morphism $$\varphi_{r,s}\colon \tau_{\leq i} \mr{dR}^{\bullet,r}_{E,s}\to\tau_{\leq i} \mr{dR}^{\bullet}_{E,s}$$ is a quasi-isomorphism for $i\leq r/s$.
\end{prop}
\begin{proof}
We have $\widetilde{\mr{dR}}^{\bullet}_{E,s}=p^{-\bullet}\mr{dR}_{E,s}^\bullet= \widetilde{\mr{dR}}^{\bullet}_{E}\otimes_{\Zp} \Lambda^s$
and $F= F\otimes \id$, where $F$ on the right side is the quasi-isomorphism \ref{thm_Mazur}, which yields the first assertion. For the second part, since  $\Omega^{j}_{\widehat{A}}$ are $p$-torsion free, then we get the quasi-isomorphism  $F\colon \pi^r p^{-\bullet}\mr{dR}_{E,s}^\bullet \to \pi^r \mr{dR}_{E,s}^\bullet$ induced by $F$.
Since $ \tau_{\leq i} ( \pi^{a_s(\bullet,r)}\mr{dR}^\bullet_{E,s}) = \tau_{\leq i} ( \pi^r p^{-\bullet}\mr{dR}_{E,s}^\bullet )$, then we are done.
\end{proof}

\subsection{Syntomic vanishing}
In this section we prove the main slope bounds that are used to give a negative answer to Question \ref{ques_EKP2}. We work with the following assumption.
\begin{hypo}\label{Hypo}
Until the end of \S \ref{Sec_SlopeBounds}, we suppose $\calV=W$ unramified.
\end{hypo} 
Let $X_0$ be an affine smooth $k$-scheme and $E$ a unit-root $F$-crystal over $X_0$. Fix $r\in \Z$ and $s\in \Z_{>0}$ such that $\mr{gcd}(r,s)=1$ and write $\gamma \coloneqq r/s$.
\begin{defi}\label{synE(r,s)}
As in \cite[\S 4]{KatoSyntomic} and \cite{Bha22}, 
the \textit{syntomic cohomology} of $X_0$ with $E(\gamma)$-coefficients is defined as the complex 
$$R\Gamma_\mr{syn}(X_0,E(\gamma))\coloneqq \Cocone(\dR^{\bullet,r}_{E,s}\xrightarrow{\varphi_{r,s}-\iota} \dR^{\bullet}_{E,s}).$$
\end{defi}
\spa

The first result we prove is an integral extension of \cite[Prop. 3.1.4]{DK17}.

\begin{prop}\label{Prop_IntegralBounds}
   The following assertions hold.
   \begin{enumerate}
       \item For $q\not\in[\gamma,\gamma +1]$, we have $H^q_\mr{syn}(X_0,E(\gamma))=0.$
       \item For $q\neq \gamma+1$, we have $$H^{q}_\syn(X_0,E(\gamma))=\Ker(\mr{dR}_{E,s}^{q,r}\xrightarrow{\varphi_{r,s}-\iota }\mr{dR}_{E,s}^{q}).$$
   \end{enumerate}
   \end{prop}
   \begin{proof}
      (i) For $i < \gamma=r/s$, write $$K^\bullet\coloneqq\Cocone(  \tau_{\leq i} \mr{dR}^{\bullet,r}_{E,s}\xrightarrow{\varphi_{r,s}-\iota}\tau_{\leq i} \mr{dR}^{\bullet}_{E,s}).$$ In order to prove (1) for $q=i$ and (2) for $q=i+1$, it is enough to prove that $K^\bullet$ is acyclic. By construction, $K^\bullet$ is a complex of $\pi$-complete $\pi$-torsion free $\Lambda^s$-modules. By \stack{0G1U}, it is then enough to prove that $K^\bullet \otimes \Lambda^s/\pi$ is acyclic. Since $a_{s}(j,r)\geq 1$ for $j\leq i$, we have that $$ \tau_{\leq i} \mr{dR}^{\bullet,r}_{E,s}\otimes \Lambda^s/\pi\xrightarrow{\iota}\tau_{\leq i} \mr{dR}^{\bullet}_{E,s}\otimes \Lambda^s/\pi$$ is the $0$-morphism. Therefore, $$K^\bullet \otimes \Lambda^s/\pi=\Cocone\left(\tau_{\leq i} \mr{dR}^{\bullet,r}_{E,s}\otimes \Lambda^s/\pi\xrightarrow{\varphi_{r,s}}\tau_{\leq i} \mr{dR}^{\bullet}_{E,s}\otimes \Lambda^s/\pi\right),$$ which is acyclic by Proposition \ref{thm_MazurFractional}.

      \spa
      
(ii) For $i >\gamma+1$ and $j\geq i-1$, we have $\mr{dR}^{j,r}_{E,s}=\mr{dR}^{j}_{E,s}$ because $a_s(j,r)=(r-js)_+=0.$ In addition, we have  $\varphi_r ( \mr{dR}^{j}_{E,s} ) \subseteq  \pi\mr{dR}^{j}_{E,s} $ because $js-r\geq 1$.
Therefore, the operator $\varphi_{r,s}-1\colon \mr{dR}^{j}_{E,s}\to \mr{dR}^{j}_{E,s}$ is invertible for $j\geq i-1.$ This shows that $$\varphi_{r,s}-1\colon H^i(\mr{dR}^{\bullet}_{E,s})\to  H^i(\mr{dR}^{\bullet}_{E,s})$$ is an isomorphism and $$\varphi_{r,s}-1\colon H^{i-1}(\mr{dR}^{\bullet}_{E,s})\to  H^{i-1}(\mr{dR}^{\bullet}_{E,s})$$ is surjective. This implies (1) and (2) for $q=i$.
   \end{proof}

   \begin{rema}
    Note that in Step (ii) of the proof of Proposition \ref{Prop_IntegralBounds}, we only used the fact that the operator $\varphi_{r,s}$ is well-defined at an integral level and is topologically nilpotent. Therefore, that part generalises to other situations where a complex of $p$-torsion free $p$-adically complete modules is endowed with a Frobenius structure satisfying suitable $p$-divisibility conditions. Note also that, for the slope $0$ case, it is enough that the Frobenius structure is topologically nilpotent. Thus, it applies directly to more general settings in \textit{prismatic cohomology}, defined in \cite{BS22}. Step (i), instead, relies crucially on the $F$-gauge structure on crystalline cohomology provided by Proposition \ref{thm_Mazur}. For the prismatic analogues of this structure, we refer to \cite{Bha22}.
    \end{rema}
    \begin{coro}\label{cor_VanishingChernClasses}
    If 
    $\calE$ is a vector bundle over $X_0$, the $n$th crystalline Chern class $c_n^\mr{cris}(\calE)\in H^{2n}_\mr{cris}(X/W)$ vanishes for $n\geq 2$.
    \end{coro}
    \begin{proof}
        Thanks to \cite[Rem. 9.2.6]{BL22}, the crystalline Chern class is compatible with the syntomic Chern class $c_n^\mr{syn}(\calE)$ (constructed in [\textit{ibid.}, Cons. 9.2.1]) via the comparison morphism $$\gamma^\mr{cris}_\mr{syn}\colon H^{2n}_\syn(X_0,\Zp(n))\to H^{2n}_\mr{cris}(X_0/W).$$ 
        By Proposition \ref{Prop_IntegralBounds}, the group $H^{2n}_\syn(X_0,\Zp(n))$ vanishes for $n\geq 2.$ This implies the desired result.
    \end{proof}
    \begin{empt}
    
Let us explain how $R\Gamma_\mr{syn}(X_0,E(\gamma))$ can be used to control generalised Frobenius eigenvectors of crystalline cohomology $H^i_\mr{cris}(X_0,E)$. We take $(r,s)$ as above with $\mr{gcd}(r,s)=1$ and we suppose that $r/s\leq i-1$. The naive truncation $\mr{dR}^{\geq i-1}_E$ is endowed with an operator $$F^s/p^r\colon \mr{dR}^{\geq i-1}_E\to \mr{dR}^{\geq i-1}_E.$$ This endows $H^i_\mr{cris}(X_0,E)$  with an endomorphism $F^s/p^r.$ In addition, we have $\mr{dR}_{E,s}^{r,\geq i-1}=\mr{dR}_{E,s}^{\geq i-1}.$

\spa

We consider the morphism $$\beta\colon H^i_\cris(X_0,E)\to H^i(\mr{dR}_{E,s}^\bullet)$$ defined as $\sum_{n=0}^{s-1}\varphi_{r,s}^n$. We get the following variant of Lemma \ref{lem_Vrs}.
\end{empt}

\begin{lemm}\label{lem_GeneralisedEV1}For $r/s\leq i-1$ and $r,s$ coprime we have the following commutative diagram with injective horizontal arrows
        \begin{equation}\label{(r,s)-comp-diagbis}  \xymatrix {H^i_\mr{cris}(X_0,E) \ar[d]^{F^s/p^r-1} \ar@{^{(}->}[r]^-\beta & H^i(\mr{dR}_{E,s}^\bullet)\ar[d]^-{\varphi_{r,s}-1} 
\\  H^i_\mr{cris}(X_0,E) \ar@{^{(}->}[r]^-\iota & H^i(\mr{dR}_{E,s}^\bullet).}
\end{equation}
\end{lemm}
\begin{proof}
For $\omega\in \dR_E^j$ with $j\geq i-1$ we have $\varphi_{r,s}^n(\omega)=\pi^{njs-r}F^n(\omega)/p^{nj}$ and $\mr{gcd}(njs-r,s)=1.$ It follows that both $\beta$ and $\iota$ have $\Zp$-linear retractions provided by $\pi\mapsto 0$. The result then follows from the identity $$F^s/p^r-1=(\varphi_{r,s}-1)\circ \left(\sum_{n=0}^{s-1}\varphi_{r,s}^n\right).$$
\end{proof}

In order to go beyond the constraints of $\mr{gcd}(r,s)=1$ it is useful to consider the following construction.

\begin{cons}
    Let $E$ be a unit-root $F$-crystal, for an integer $d\geq 1,$ we write $E^{(d)}$ for $E^{\oplus d}$ with Frobenius structure defined by $$(v_0,\dots,v_{d-1})\mapsto (F_E(v_{d-1}),F_E(v_{0}),\dots, F_E(v_{d-2})),$$ where $F_E$ is the Frobenius structure on $E$. Let $\iota \colon E\to E^{(d)}$ be the inclusion of the first summand.
\end{cons}

\begin{lemm}\label{lem_GeneralisedEV2}For $d>0$, we have the following commutative diagram with injective horizontal arrows
\begin{equation}\xymatrix {H^i_\mr{cris}(X_0,E) \ar[d]^{F^{ds}/p^{dr}-1} \ar@{^{(}->}[r]^\iota & H^i_\mr{cris}(X_0,E^{(ds)})\ar[d]^{F^s/p^r-1} 
\\  H^i_\mr{cris}(X_0,E) \ar@{^{(}->}[r]^\iota & H^i_\mr{cris}(X_0,E^{(ds)}).}
\end{equation}
\end{lemm}
\begin{proof}The commutativity follows from the construction. For the injectivity, note that both $\alpha$ and $\iota$ admit a retraction (which does not commute with the Frobenius).
\end{proof}
    \begin{coro}\label{cor_GeneralisedEV}
For $(t,u) \in \Z\times \Z_{>0}$ such that 
$t/u< i-1,$ we have  $$H^i_\mr{cris}(X_0,E)^{F^u/p^t-1}=0.$$    
    \end{coro}
    \begin{proof}
    This follows from the combination of Proposition \ref{Prop_IntegralBounds} , Lemma \ref{lem_GeneralisedEV1}, and Lemma \ref{lem_GeneralisedEV2}. Alternatively, one can readapt more directly the proof of Proposition \ref{Prop_IntegralBounds}.
    \end{proof}

\section{Injectivity failure}\label{Sec_InjectivityFailure}
In this section, we prove Theorem \ref{thm_counterexample}, which shows the failure of injectivity of the comparison map between rigid and convergent cohomology, thereby answering Question \ref{ques_EKP} negatively.
\subsection{Rationalised $p$-adic Tate twists} We first introduce a rationalised version of the $p$-adic Tate twists of \S \ref{Sec_FractionalNygaard} for convergent $F$-isocrystals. Suppose that $\calV\in \mr{CDVR}^\star(\Zp)$ is endowed with a Frobenius lift $\sigma\colon \calV\to \calV.$

\begin{defi} \label{dfn-Q(r,s)}
For $r\in \Z$ and $s>0$, we write $\Qp{(r,s)}$ for the
convergent $F$-isocrystal on $\Spec(k)/K$ defined as $K[t]/(t^s-p^{-r})$ endowed with
the $\sigma$-linear map $F \colon \Qp{(r,s)} \to \Qp{(r,s)}$ induced by 
$\sum_i a_i t^i \mapsto \sum \sigma (a_i) t^{i+1}$.
\end{defi}

\begin{empt}\label{splip-injQ(0,s)}
    Let $r,r'\in \bbZ$, $s,s'\in \bbN$ and write $m\coloneqq\mr{lcm}(s,s')$, $d\coloneqq \mr{gcd}(s,s')$, and $n\coloneqq\frac{m}{s}r+\frac{m}{s'}r'$. In $\Qp(r,s) \otimes_K \Qp(r',s')$ we have $(t\otimes t)^m=p^{-n}.$ We deduce a (non-canonical) decomposition
 \begin{equation}\label{eq_Decomposition}
    \Qp(r,s) \otimes \Qp(r',s')=\bigoplus_{i=0}^{d-1}\Qp\left(n,m\right)
 \end{equation}
 where the $i$th copy of $\Qp\left(n,m\right)$ is generated as an $F$-isocrystal by $t^i\otimes 1.$ Moreover, there is a map $$\beta\colon\Qp(-r,s)\otimes \Qp(r,s)\to\Qp(0,1)$$ of $F$-isocrystals defined by the assignment $t^i\otimes t^j\mapsto \delta_{ij}.$ This defines an isomorphism $$\Qp(-r,s)\riso \Qp(r,s)^* \coloneqq \Hom_{K} (\Qp(r,s), \Qp(0,1)).$$

 The unit of the duality is the morphism $$\Qp(0,1)\to \Qp(-r,s)\otimes  \Qp(r,s)$$
 given by $$1\mapsto \id =  1\otimes 1 + t\otimes t + \dots + t^{s-1}\otimes t^{s-1}.$$
 This morphism has a retraction provided by $\beta/s$.
 \end{empt}

\begin{empt}\label{(r,s)-comp}
    Let $V$ be a $K$-vector space endowed with a $\sigma$-linear bijective morphism    $F\colon V\to V$. 
    For $r\in \bbZ$, $s\in \Z_{>0}$, set $V(r,s)\coloneqq V \otimes_{K}\Qp(r,s) $, which is endowed with the $\sigma$-linear map given as usual by the tensor product, i.e. by the formula
$F (x \otimes e)= F (x) \otimes F (e)$ for any $x \in V$ and $e \in \Qp(r,s)$. Hence, $V(r,s)= V^{\oplus s}$ as $K$-vector space and we might denote its elements simply by  
$\sum_{i=0}^{s-1} v_i t^i$ instead of $\sum_{i=0}^{s-1} v_i \otimes t^i$, with $v_i \in V$ (and the writing is unique).
We consider the morphism $\alpha\colon V\to V(r,s)$ given by $$v\mapsto v+F(v)t+\dots+F^{s-1}(v)t^{s-1}$$  and $\iota\colon V\to V(r,s)$ such that $v\mapsto v\otimes 1.$
\end{empt}

\begin{lemm}\label{lem_Vrs} 
With notation as in \S\ref{(r,s)-comp}, the following properties hold.
\begin{enumerate}
    \item The map $\alpha$ induces the isomorphism
$$\Ker \left(V \stackrel{F^s -p^r}\longrightarrow V\right)=\Ker \left(V \stackrel{F^s/p^r-1}\longrightarrow V\right)\xrightarrow[\simeq]{\alpha} \Ker \left(V (r,s) \stackrel{F -1} \longrightarrow V(r,s)\right).$$ 

\item The map $V \stackrel{F^s -p^r}\longrightarrow V$ is a bijection if and only so is $V (r,s) \stackrel{F -1} \longrightarrow V(r,s)$.

\end{enumerate}
\end{lemm}
\begin{proof}
The map  $V (r,s) \stackrel{F -1} \longrightarrow V(r,s)$ sends $$\sum_{i=0}^{s-1} v_i t^i\mapsto p^{-r} F (v_{s-1}) -v_0 +\sum_{i=1}^{s-1} (F (v_{i-1}) -v_i ) t^i.$$
This yields the following commutative diagram of $K$-vector spaces
    \begin{equation}\label{(r,s)-comp-diag}  \xymatrix {0\ar[r] &V \ar[d]^{F^s/p^r-1} \ar[r]^-\alpha & {V(r,s)} \ar[d]^-{F-1} \ar[r] & {V(r,s)/\alpha(V)} \ar[r]\ar@{^{(}->}[d] & 0
\\ {0}\ar[r]& V \ar[r]^-\iota & V(r,s)\ar[r] & V(r,s)/\iota(V)\ar[r]& 0.}
\end{equation}
The right vertical arrow is injective by a direct computation.  Hence, we get Part (1) of the Lemma via a snake Lemma.


\spa

If $V (r,s) \stackrel{F -1} \longrightarrow V(r,s)$ is surjective, then via the snake lemma applied to \eqref{(r,s)-comp-diag} we get the surjectivity of $V \stackrel{F^s -p^r}\longrightarrow V$. 
Conversely,  suppose $V \stackrel{F^s -p^r}\longrightarrow V$ is surjective. 
Let $\sum_{i=0}^{s-1} w_i t^i\in V(r,s)$. We are looking for 
$\sum_{i=0}^{s-1} v_i t^i$ such that  
$$p^{-r} F (v_{s-1}) -v_0 +\sum_{i=1}^{s-1} (F (v_{i-1}) -v_i ) t^i =\sum_{i=0}^{s-1} w_i t^i.$$
By hypothesis, there exists $v_0 \in V$ such that $p^{-r} F^s (v_0) -v_0 = w_0 + p^{-r}\sum_{i=1}^{s-1} F^{s-i} (w_i)$.
For $i =1,\dots, s-1$, we define by induction $v_i\coloneqq F (v_{i-1})-w_{i}$. We compute that such a $\sum_{i=0}^{s-1} v_i t^i$ satisfies the desired equation. 
Indeed, for $i =1,\dots, s-1$,  $w_i=F (v_{i-1})-v_{i}$ is clear.
Moreover, $$\sum_{i=1}^{s-1} F^{s-i} (w_i)=\sum_{i=1}^{s-1} F^{s-i} (F (v_{i-1}) -v_i )= F^{s} (v_0)-F(v_{s-1}).$$
Hence, we get $$w_0  = p^{-r} F^s (v_0) -v_0 - p^{-r} \sum_{i=1}^{s-1} F^{s-i} (w_i)= p^{-r} F^s (v_0) -v_0  - p^{-r}( F^{s} (v_0)  - F(v_{s-1}))=p^{-r} F (v_{s-1}) -v_0.$$
\end{proof}

\begin{rema}
    The composition $V (r,s) \stackrel{F -1} \longrightarrow V(r,s)\to V(r,s) /\iota (V)$ sends $$\sum_{i=0}^{s-1} v_i t^i \mapsto \sum_{i=1}^{s-1} (F (v_i-1) -v_i ) t^i .$$
    Hence, the right vertical arrow of \eqref{(r,s)-comp-diag} is an isomorphism if and only if 
    $F -1\colon V \to V$ is surjective. 
\end{rema}

We will also need the following lemma later (in the proof of Theorem \ref{thm_InjectivitySlopei}), which is an easy exercise.
\begin{lemm} \label{lemkercokerF}Let $V, W$ be two $K$-vector spaces endowed with $\sigma$-linear bijective maps $F\colon V\to V$ and $F\colon W\to W$.    
 Let $a\colon V \to W$ be a $K$-linear morphism that commutes with $F$.
 Then $\Ker (a) $, $\Coker (a)$, and $\mr{Im}(a)$ are $K$-vector spaces endowed with a $\sigma$-linear bijective map compatible with that of $V$ and $W$.
\end{lemm}

\subsection{Overconvergent syntomic vanishing}
We extend the slope vanishing results to the convergent and overconvergent settings. We also see the relation between Dieudonné--Manin slopes and syntomic cohomology.
\begin{hypo}\label{Hyporig}
Until the end of \S \ref{Sec_InjectivityFailure}, we suppose $\calV=W$ unramified.
\end{hypo} 
Let $j\colon X_0 \hookrightarrow  Y_0$ be an open immersion of separated schemes of finite type over $k$. Suppose $X_0$ is smooth over $k$.

\begin{defi}\label{twistIsox}
Denoting by $\theta \colon (X_0,Y_0)\to (\Spec(k) , \Spec(k))$ the canonical map, with notation as in \S \ref{dfn-Q(r,s)}, 
we get an overconvergent $F$-isocrystal $\theta^* (\Qp(r,s))$ over $(X_0,Y_0)/K$ that we denote again by $\Qp(r,s).$ 
This overconvergent $F$-isocrystal has pure slope $-r/s$ and rank $s$. For an overconvergent $F$-isocrystal $\cG$ on $(X_0,Y_0)/K$, we get a twisted overconvergent $F$-isocrystal $\cG$ on $(X_0,Y_0)/K$ by setting $\cG(r,s)\coloneqq\cG\otimes_{\Qp(0,1)}\Qp(r,s) $.

\spa

The \textit{(overconvergent) syntomic cohomology} of $(X_0,Y_0)/K$ with $\cG(r,s)$-coefficients is defined as the complex
$$R\Gamma_\mr{syn}(X_0,Y_0,\cG(r,s))\coloneqq \Cocone(R \Gamma_{\mr{rig}}(X_0,Y_0,\cG(r,s)) \stackrel {F-1}\longrightarrow R \Gamma_{\mr{rig}}(X_0,Y_0,\cG(r,s))).$$
\end{defi}
Definition \ref{twistIsox} corresponds to an overconvergent variation of Definition \ref{synE(r,s)}. The following lemma makes the link between the two when the overconvergent singularities are empty.

\begin{lemm}\label{lem_compatibilitysyntomic}
Suppose $X_0$ is an affine smooth $k$-scheme. Let $E$ be a unit-root $F$-crystal on $X_0$ and $\calE\coloneqq E_K$. If $\mr{gcd}(r,s)=1$, then we have the following commutative diagram commute
$$ \xymatrix{ {R\Gamma_\mr{rig}(X_0,X_0,\calE(r,s))} \ar[r]^-{F-1}\ar[d]^-{\sim} & {R\Gamma_\mr{rig}(X_0,X_0,\calE(r,s))}\ar[d]^-{\sim}
\\
{(\dR^{\bullet,r}_{E,s} )_{\bbQ_p} } \ar[r]^-{\varphi_{r,s}-\iota}  & {(\dR^{\bullet}_{E,s} )_{\bbQ_p}, }   }$$ 
where the vertical arrows are isomorphisms sending $t\mapsto 1/\pi^r.$ 
In particular, we get an isomorphism
 $$R\Gamma_\mr{syn}(X_0,X_0,\calE(r,s))\riso R\Gamma_\mr{syn}(X_0,E(\gamma))_{\Qp}.$$
\end{lemm}
\begin{proof}
We have $ \dR^{\bullet}_{\calE(r,s)}\coloneqq \calE(r,s) \otimes \Omega^\bullet_{\fX}= R\Gamma_\mr{rig}(X_0,X_0,\calE(r,s))$.
We have $\dR^{\bullet}_{\calE(r,s)}\riso \dR^{\bullet}_{\calE} (r,s)$ (e.g. use the projection isomorphism) and the action of $F$ on $\dR^{\bullet}_{\calE(r,s)}$ corresponds via this identification to 
$t F_\calE$ where $F_\calE$ is the action on $ \dR^{\bullet}_{\calE} $. Then the result follows from the commutativity of the diagram
$$\xymatrix{ {\dR^{\bullet}_{\calE (r,s)} }  \ar[d]^-{F -1} \ar[r]^-{\sim} & {\dR^{\bullet}_{\calE}(r,s) }  \ar[d]^-{t  F_\calE -1} \ar[r]^-{\sim} &  {(\dR^{r,\bullet}_{E,s})_K }  \ar[d]^-{F/\pi^r -\iota}  
\\ { \dR^{\bullet}_{\calE (r,s)}} \ar[r]^-{\sim} &{ \dR^{\bullet}_{\calE} (r,s)} \ar[r]^-{\sim} & {(\dR^{\bullet}_{E,s})_K}, }$$
where the right horizontal arrows are the isomorphisms which send $t\mapsto 1/\pi^r.$
\end{proof}

\begin{prop}\label{Prop_RationalBoundsvar0}
Fix $r\in \Z$ and $s\in \Z_{>0}$ such that $\mr{gcd}(r,s)=1$ and write $\gamma \coloneqq r/s$. Let $\cG$ be a unit-root overconvergent $F$-isocrystal over $(X_0,Y_0)/K$.
Then the map
 $$F-1\colon H^q_\mr{rig}(X_0,Y_0,\cG(r,s))\to H^q_\mr{rig}(X_0,Y_0,\cG(r,s))$$ 
 is an isomorphism for $q<\gamma$. In particular, $ H^q_\mr{syn}(X_0,Y_0,\cG(r,s))=0$ for $q<\gamma$ ; and if $\gamma \not \in \bbN$ and $n \in ] \gamma ,\gamma +1[$ we have
$$H^{n}_\mr{syn}(X_0,Y_0,\cG(r,s))=\Ker(H^{n} _\mr{rig}(X_0,Y_0,\cG(r,s))\xrightarrow{F-1}H^{n}_\mr{rig}(X_0,Y_0,\cG(r,s))).$$
   \end{prop}
\begin{proof}  
(a) Let us first reduce ourselves to the logarithmic setting.\begin{itemize}
    \item [(i)]  Using Čech spectral sequence and an induction on $q$, we show that the statement is Zariski local in $Y_0$ and in $X_0$. Hence, we can suppose $D _0 \coloneqq Y _0 \setminus X _0$ is the support of a divisor. Thanks to \cite[Thm 1.3.1]{Tsu02}, there exists a smooth scheme $Y_0 '$ of finite type over $k$ and a proper surjective morphism $w \colon Y_0 ' \to Y_0$ over $k$
such that $w$ is generically étale and, if we set $X'_0 \coloneqq  w^{-1} (X_0)$, $D_0'\coloneqq Y _0'\setminus X'_0$, and  $j' \colon X'_0 \to Y _0'$, there exists a unique unit-root $\cE ' \in F  \text{-}\Isoc(Y_0 '/K )$
with $w^*\cG  = (j')^{\dag} \cE'$  as overconvergent $F^a$-isocrystals on $X'_0/K$ along $D'_0$. We can assume $Y'_0$ quasi-projective and $D'_0$ a strict normal crossing divisor of $Y '_0$ (e.g. use de Jong desingularisation theorem).
Using \cite[16.1.11.2]{Car25}, we deduce that $R \Gamma_\mr{rig}(X_0,Y_0,\cG(r,s))$ is a direct summand of $R \Gamma_\mr{rig}(X'_0,Y '_0,w^* (\cG(r,s)))$ and the splitting commutes with the action of Frobenius. Hence, we reduce to the case where
$Y_0$ is a smooth $k$-scheme, there exists $D_0$ a strict normal crossing divisor of $Y_0$ such that  $X _0=Y_0 \setminus D_0$, and 
a unit-root $\cE \in F  \text{-}\Isoc(Y_0/K )$ with $\cG  = j^{\dag} \cE$.

\spa

\item[(ii)]
Since the statement is Zariski local in $Y_0$ and in $X_0$, then we can suppose that there exists an affine formal smooth lift $\fY=\Spf(A)$ of $Y_0$ over $W$ 
and there exists a relative strict normal crossing divisor $\fD$ of $\fY$  lifting $D_0$. 
\end{itemize} 

(b) Let $\fY^\sharp=(\fY , M ({\fD}))$.  We denote by $\Omega^i_{A^\sharp}$ the global section of  $\Omega^i_{\fY^{\sharp}/W}$. 
Let $E$ be a unit-root $F$-crystal  over $Y_0$ such that $E_K = \cE$.
Write $F$ for a Frobenius lift $F\colon \fY\to \fY.$ 
Similarly to \S \ref{sec_FGauge}, \eqref{dfn-dREs}, and Definition \ref{synE(r,s)}, we set 
\begin{gather*}
\mr{dR}_{Y_0^\sharp, E}^\bullet \coloneqq E \otimes_{A}\Omega^\bullet_{A^\sharp} ,\qquad \mr{dR}_{Y_0^\sharp, E,s}^\bullet\coloneqq\mr{dR}_{Y_0^\sharp, E}^\bullet\otimes_{\Zp} \Lambda^s.
\end{gather*}


We have the following commutative diagram with vertical isomorphisms
$$ \xymatrix{ { H^q ( \dR^{\bullet,r}_{Y_0^\sharp, E,s} )} \ar[r]^-{\varphi_{r,s}-\iota} \ar[d]^-{\sim} & {H^q ( \dR^{\bullet,r}_{Y_0^\sharp, E,s} )} \ar[d]^-{\sim} 
 \\ {H ^q _\mr{rig}(Y_0^\sharp,Y_0^\sharp,\calE(r,s))} \ar[r]^-{F-1} \ar[d]^-{\text{\cite[1.3.6]{CT12}}}_-{\sim} & {H ^q _\mr{rig}(Y_0^\sharp,Y_0^\sharp,\calE(r,s))} \ar[d]^-{\text{\cite[1.3.6]{CT12}}}_-{\sim}
\\ {H ^q  _\mr{rig}(X_0,Y_0,\cG (r,s))} \ar[r]^-{F-1} & { H ^q  _\mr{rig}(X_0,Y_0,\cG (r,s)), } }$$ 
where the top vertical isomorphisms are constructed similarly (i.e. it is still available with logarithmic structures) to Lemma \ref{lem_compatibilitysyntomic}.
Arguing as in Part (i) of the proof of Proposition \ref{Prop_IntegralBounds}, for any $q<\gamma$ the top horizontal arrow is an isomorphism.
Hence, we are done.
\end{proof}
\begin{defi}[Dieudonn\'e--Manin classification]\label{DMC-crew}
Since Hypothesis \ref{Hyporig} is in force, by \cite[Thm. 3.2]{Ked22} finite-rank $F$-isocrystals over $\calV=W$ admit a Dieudonn\'e--Manin decomposition
(see \cite[\S 1.9]{crew_F-iso_p-adic-rep} for some ramified cases). 
For an  overconvergent $F$-isocrystal $\cG$ over $(X_0,Y_0)/K$ and $I\subseteq \mathbb{Q}$, we then write $H^i_\mr{rig}(X_0,Y_0,\cG)^{I}$ for the subspace generated by finite-rank sub-$F$-isocrystals of $H^i_\mr{rig}(X_0,Y_0,\cG)$ of slope $\gamma\in I$. If $I$ consists of a single element $\gamma$ we write $I=[\gamma]$. When $k$ is algebraically closed, the $K$-vector space $H^i_\mr{rig}(X_0,Y_0,\cG)^{[\gamma]}$ 
has a basis consisting of elements belonging to 
$H^i_{\rig}(X_0,Y_0,\cM)^{F^s =p^r}$ for $r,s$ two coprime integers such that
$\gamma= r/s$ and $s > 0$. 
When $Y_0$ is proper, we omit it in the notation. When $Y_0= X_0$, we prefer to write  $H^i_\mr{conv}(X_0, \cG)^{I}$.
\end{defi}

\begin{coro}\label{cor_RationalBoundsvar} Let $\cG$ be a unit-root overconvergent $F$-isocrystal over $(X_0,Y_0)/K$.
Let $\gamma \in \bbQ$ and $q\in \bbN$ be an integer such that $q<\gamma$. 
Let $r,s$ be two coprime integers such that $\gamma= r/s$ and $s > 0$.
\begin{enumerate}
    \item The map $F^s -p^r \colon H^q_\mr{rig}(X_0,Y_0,\cG)\to H^q_\mr{rig}(X_0,Y_0,\cG)$ is an isomorphism.
        \item We have  $H^q_\mr{rig}(X_0,Y_0,\cG)^{[\gamma]}=0.$
\end{enumerate}
\end{coro}

\begin{proof}
%
 (i)   Thanks to Proposition \ref{Prop_RationalBoundsvar0} the map $F-1\colon H^q_\mr{rig}(X_0,Y_0,\cG(r,s)) \to H^q_\mr{rig}(X_0,Y_0,\cG(r,s))$ is an isomorphism. Using the projection isomorphism, we get the compatible with Frobenius  isomorphism $$R\Gamma_\mr{rig}(X_0,Y_0,\cG(r,s))\riso \left (R\Gamma_\mr{rig}(X_0,Y_0,\cG) \right) (r,s).$$ Thanks to Lemma \ref{lem_Vrs}.(2), this implies (1).
\spa

(ii) We may assume $k$ algebraically closed by flat descent. 
Thanks to Dieudonn\'e--Manin classification, (2) follows from (1).
\end{proof}

\begin{prop}
\label{Prop_RationalBounds}
  Suppose $X_0$ is an affine smooth $k$-scheme. Let $\calM$ be a convergent $F$-isocrystal on $X_0$ with constant slopes in the interval $[\alpha,\beta].$ Let $(r,s) \in \Z\times \Z_{>0}$ and write $\gamma \coloneqq r/s$. For $i\not\in[\gamma-\beta,\gamma -\alpha +1]$, we have $H^i_\mr{syn}(X_0,\calM(r,s))=0.$
\end{prop}
\begin{proof}
    (i) Suppose first that $\calM$ has pure slope. We choose $(u,v) \in \Z\times \Z_{>0}$ such that $\delta\coloneqq u/v$ is the slope of $\calM$ and we define $\calE\coloneqq \calM(u,v)$. The convergent $F$-isocrystal $\calE$ is unit-root by construction. Thanks to \S \ref{splip-injQ(0,s)}, 
    the unit of the $\otimes$-duality between $\Qp(u,v)$ and $\Qp(-u,v)$ induces 
    a split injective morphism 
    $$\calM\hookrightarrow \calM\otimes\Qp(u,v)\otimes  \Qp(-u,v)= \calE(-u,v).$$

    Combining this with \eqref{eq_Decomposition}, we deduce that  $\calM(r,s)$ admits a split injection into some $\calE'(r',s')$ with $\calE'$ a direct sum of copies of $\calE$ and $\gamma'\coloneqq r'/s'=\gamma-\delta.$ By the assumption, we have that $[\gamma',\gamma'+1]\subseteq [\gamma-\beta,\gamma -\alpha +1]$. This shows that the result for $\calM$ follows from the same result for a unit-root $F$-isocrystal. This is proved in Proposition \ref{Prop_IntegralBounds}.

    \spa

    (ii) For the general case, write $S_1(\calM)\subseteq \calM$ for the subspace of minimal slope. We may assume $S_1(\calM)$ of slope $\alpha$. We have an exact sequence $$0\to S_1(\calM)\to \calM \to \mathcal{N}\to 0,$$ which induces a long exact sequence $$\cdots \to H^i_\mr{syn}(X_0,S_1(\calM)(r,s))\to H^i_\mr{syn}(X_0,\calM(r,s))\to H^i_\mr{syn}(X_0,\mathcal{N}(r,s))\to \cdots.$$ Note that both $S_1(\calM)$ and $\calN$ have slopes in the interval $[\alpha,\beta]$, thus, by the inductive hypothesis, the outer groups vanish. This implies that $H^i_\mr{syn}(X_0,\calM(r,s))=0$, as we wanted.
\end{proof}

\begin{coro}\label{thm_counterexample-lem}
Suppose $X_0$ is an affine smooth $k$-scheme.  Let $\gamma$ be a rational number, and let $\calM$ be a convergent $F$-isocrystal over $X_0$ with constant slopes in the interval $[\alpha,\beta]\subseteq \Q$. For $i\not \in [\gamma-\beta , \gamma -\alpha +1]$, we have 
    $$H^{i}_\mr{conv}(X_0,\calM)^{[\gamma]} =0.$$
\end{coro}
\begin{proof}
Let $r,s$ be two coprime integers such that $\gamma= r/s$ and $s > 0$.
We may assume $k$ algebraically closed by flat descent.
By Lemma \ref{lem_Vrs} and Proposition \ref{Prop_RationalBounds}
we get that $H^i_{\conv}(X_0/\cM)^{F^s =p^r}$ vanishes for $i\not \in [\gamma-\beta , \gamma -\alpha +1]$. 
    By the Dieudonn\'e--Manin classification (see \S \ref{DMC-crew}), this shows that 
    $H^i_\mr{conv}(X_0,\calM)^{[\gamma]}=0$ for $i\not \in [\gamma-\beta , \gamma -\alpha +1]$.\end{proof}

\subsection{The counterexamples}\label{Sec_Counterexamples}

In this section we finally put together the previous results to Theorem \ref{thm_counterexample}. We further construct various concrete counterexamples to Question \ref{ques_EKP}.

\begin{theo}\label{thm_counterexample}
Let $X_0$ be a smooth affine scheme over $k$, let $\gamma$ be a rational number, and let $\calM^\dagger$ be an overconvergent $F$-isocrystal over $X_0$ with constant slopes in the interval $[\alpha,\beta]\subseteq \Q$. Then, for $i\not \in [\gamma-\beta , \gamma -\alpha +1]$, the comparison morphism $$H^{i}_\mr{rig}(X_0,\calM^\dagger) \to H^{i}_{\mr{conv}}(X_0,\calM)$$ kills $H^{i}_\mr{rig}(X_0,\calM^\dagger)^{[\gamma]}$.
\end{theo}
\begin{proof}
Let $r,s$ be two coprime integers such that $\gamma= r/s$ and $s > 0$.
Since the comparison map $$H^{i}_\mr{rig}(X_0,\calM^\dagger)\to H^{i}_\mr{conv}(X_0,\calM)$$ commutes with Frobenius, it is enough to prove that $H^i_\mr{conv}(X_0,\calM)^{[\gamma]}=0$ when $i\not \in [\gamma-\beta , \gamma -\alpha +1]$, which follows from Corollary \ref{thm_counterexample-lem}. 
\end{proof}


\begin{coro}\label{countex}
    Let ${E}_0$ be an ordinary elliptic curve. We write $Y_0$ for the square of $E_0$ and we take $X_0\coloneqq Y_0\setminus D_0$, with $D_0$ the divisor $(E_0\times 0) \cup (0\times E_0)$. The comparison morphism $$H^i_\mr{rig}(X_0/K)\to H^i_\mr{conv}(X_0/K)$$ is not injective.
\end{coro}
\begin{proof}We may assume $k$ algebraically closed. By the assumption that $E_0$ ordinary and by Dieudonn\'e--Manin classification, the vector space $H^1_{\mr{rig}}(E_0/K)$ contains a non-trivial vector $v$ such that $Fv=v.$ Thanks to the K\"unneth formula, \cite[Thm. 3.2]{Be2}, the tensor $v\otimes v$ is a non-trivial element of $H^2_\mr{rig}(Y_0/K)$ fixed by the Frobenius. Using the distinguished localisation triangle with respect to $D_0$, we get the exact sequence $$
    H^2_{D_0,\mr{rig}}(Y_0/K)
    \to H^2_{\mr{rig}}(Y_0/K)\to H^2_{\mr{rig}}(X_0/K).$$ 
Mayer--Vietoris distinguished triangle yields the exact sequence
    $$H^2_{0\times 0,\mr{rig}}(Y_0/K) \to H^2_{E_0 \times 0,\mr{rig}}(Y_0/K)\oplus H^2_{0\times E_0,\mr{rig}}(Y_0/K) \to H^2_{D_0,\mr{rig}}(Y_0/K)\to H^3_{0\times 0,\mr{rig}}(Y_0/K).$$

Using Gysin isomorphism, \cite[Thm. 4.1.1]{Tsu99}, we get the isomorphism
$$ H^0_{\mr{rig}}(E_0\times 0)(-1)\oplus H^0_{\mr{rig}}(0\times E_0)(-1)\xrightarrow{\sim} H^2_{D_0,\mr{rig}}(Y_0/K).$$

    Hence, the kernel of $H^2_{\mr{rig}}(Y_0/K)\to H^2_{\mr{rig}}(X_0/K)$ has slope $1$. We deduce that the image of  $v\otimes v\in H^2_{\mr{rig}}(Y_0/K)$ in $H^2_{\mr{rig}}(X_0/K)$ spans a rank $1$ vector subspace $L\subseteq H^2_{\mr{rig}}(X_0/K)$ of slope $0$. By Theorem \ref{thm_counterexample}, we deduce that the comparison morphism kills $L$, and this yields the desired result.
\end{proof}

This is a special case of the following result.

\begin{coro}
      Let $Y_0$ be smooth projective scheme over $k$ such that $H^i_\mr{rig}(Y_0/K)^{[0,1)}\neq 0$ for $i\geq 2$, then for every $D_0\subseteq Y_0$ ample strict normal crossing divisor, the comparison morphism $$H^i_\mr{rig}(X_0/K)\to H^i_\mr{conv}(X_0/K)$$ is not injective, where $X_0\coloneqq Y_0\setminus D_0.$ 
\end{coro}

\begin{proof} 
When $D_0$ is smooth, by Gysin, we have $H^i_{D_0,\mr{rig}}(Y_0/K)= H^{i-2}_{\mr{rig}}(D_0/K)(-1)$, which has slopes $\geq1$ because the slopes of $H^{i-2}_{\mr{rig}}(D_0/K)$ are\footnote{Rigid cohomology of smooth varieties has always non-negative slopes. See for example \cite[Thm. 8.39]{BerthelotOgus} and \cite[Thm. 5.4.1]{KedWeil2}.} $\geq 0$. Using Mayer--Vietoris exact sequences, this implies that
$H^i_{D_0,\mr{rig}}(Y_0/K)$ has slopes $\geq1$.
Hence, $H^i_\mr{rig}(X_0/K)^{[0,1)}\neq 0$. 
      
\end{proof}

\begin{theo}\label{thm_counterexampleGHC}
    Let $Y$ be a smooth proper scheme over $\calV$. For every affine open $X\subseteq Y$, the image of the composition $$H^i_\dR(Y_K/K)\to H^i_\dR(X_K/K)\to H^i_\mr{conv}(X_0/K)$$ has dimension at most $2h^{i,0}(Y_K)+h^{i-1,1}(Y_K).$ 
\end{theo}
\begin{proof}
    By  \cite{Fal88}, the Newton polygon of the special fibre is above the Hodge polygon of the generic fibre and the endpoints match (see also \cite[Thm. 2]{Yao23}). If we write $d$ for the dimension of $H^i_\mr{rig}(Y_k/K)^{[i-1,i]}$, $h^i\coloneqq h^{i,0}(Y_K)$, and $h^{i-1}\coloneqq h^{i-1,1}(Y_K)$, by looking at the final segment of length $d$ of the two polygons, we deduce that $$(i-1)d\leq ih^i+(i-1) h^{i-1}+(i-2)(d-h^i-h^{i-1}),$$ which implies $d\leq 2h^i+h^{i-1}.$ Combining this bound with Theorem \ref{thm_counterexample}, this shows that the image of the composition $$H^i_\mr{rig}(Y_0/K)\to H^i_\mr{rig}(X_0/K)\to H^i_\mr{conv}(X_0/K)$$ has dimension at most $2h^i+h^{i-1}.$ By Proposition \ref{prop_dRC-RC}, we get the desired result.
\end{proof}
\begin{exam}
    In Theorem \ref{thm_counterexampleGHC}, if we take $Y$ to be a smooth cubic fourfold and $X$ the complement of a smooth hyperplane section, then $Y_K$ has the following Hodge numbers in degree $4$ $$h^{4,0}=0,h^{3,1}=1,h^{2,2}=21.$$ By Theorem \ref{thm_counterexampleGHC}, the kernel of $H^4_\dR(Y_K/K)\to H^4_\mr{conv}(X_0/K)$ has dimension at least $22$. On the other hand, the kernel of $H^4_\dR(Y_K/K)\to H^4_\dR(X_K/K)$ is generated by the Chern class of the self-intersection class of the hyperplane bundle. Thus, in this case $V^4_\mr{dR}(X/\calV)_\Q$ contains a subspace of weight $4$ of dimension at least $21.$ 
\end{exam}

\section{The  de Rham--Witt complex}

\label{Sec_dRW}
\subsection{Unit-root coefficients}\label{Sec_dRWUnitRoot}
We now focus our attention on the de Rham--Witt complex. We analyse the Nygaard filtration with unit-root coefficients. We also compare the cohomology of \'Etesse's twisted logarithmic sheaves $\nu(E,r)$ with syntomic cohomology, as defined in \S\ref{synE(r,s)}.
\begin{empt}\label{Sec_dRWUR}Let $X_0$ be a smooth $k$-scheme. We write $\nu_n(r)$ for the sheaves $W_n\Omega^r_{X_0,\mr{log}}$ and $\nu(r)$ for the pro-étale sheaf $\lim_n \nu_n(r)$. Thanks to \cite[Prop. 8.4]{BMS}, the sequence of complexes of pro-\'etale sheaves
\begin{equation}
    \label{BMS8.7pre} 0\to \nu(r)[-r]\to\calN^{\ge r}W\Omega^\bullet_{X_0}\xrightarrow{\varphi_r-\iota}W\Omega^\bullet_{X_0}\to 0
\end{equation}is exact and $\nu(r)$ coincides with the derived inverse limit $R\lim_n\nu_n(r)$. More generally, if $E$ is a unit-root $F$-crystal, \'Etesse constructed in \cite{Ete88} the \textit{de Rham--Witt complex with coefficients in} $E$, denoted by $E_\bullet\otimes W_\bullet\Omega_{X_0}$ and he constructed in [\textit{ibid.}, Def III.2.1 and Prop. III.2.2] the twisted sheaves $\nu_\bullet (E,r)$. As usual, by taking limits, we get a complex of pro-étale sheaves $E\otimes W\Omega_{X_0}$ and pro-étale sheaves $\nu(E,r)$. We write $$\calN^{\geq i}(E\otimes W\Omega_{X_0}^\bullet)\subseteq E\otimes W\Omega_{X_0}^\bullet$$ for the subcomplex defined by $$\calN^{\geq i}(E\otimes W\Omega_{X_0}^j)\coloneqq p^{i-j-1}V(E\otimes W\Omega^j_{X_0})\quad \mr{for}\ j<i $$  and $$\calN^{\geq i}(E\otimes W\Omega_{X_0}^j)\coloneqq E\otimes W\Omega^j_{X_0} \quad \mr{for}\ j\geq i.$$ We get a twisted variant of \cite[Prop. 8.4]{BMS}.

\end{empt}

\begin{prop}\label{BMS8.7preUR}The complexes of pro-\'etale $\bbZ_p$-sheaves over $X_0$ defined in \S\ref{Sec_dRWUR} sit in the following exact sequence  
$$0\to \nu(E,r)[-r]\to  \calN^{\ge r}(E\otimes W\Omega^\bullet_{X_0})\xrightarrow{\varphi_r-\iota}
    E\otimes W\Omega^\bullet_{X_0}\to 0.$$ 
    Moreover, $\nu(E,r)$ is derived $p$-complete and $\nu(E,r)\otimes^L\Z/p^n=\nu_n(E,r)$. 
\end{prop}

\begin{proof}
 \'Etesse   constructed in  \cite[Lem. III.2.4]{Ete88} the exact sequence
\begin{equation}\label{I.5.7.2-Ilpre}
0\to \nu_\bullet (E,r)\to E_\bullet \otimes W_\bullet\Omega^r_{X_0}\xrightarrow{{\rm F}-1}E_\bullet\otimes W_\bullet\Omega^r_{X_0} \to 0
\end{equation} 
  of étale pro-sheaves. By taking limits in the category of pro-étale sheaves, this implies that $$\nu(E,r)=R\mr{lim}_n\nu_n(E,r).$$  Moreover, thanks to \cite[Lem. III.2.8]{Ete88}, we have the exact sequence 
\begin{equation}\label{BMS8.7preURproof1} 0\to \nu_m (E,r)\xrightarrow{\underline{p}^n} \nu_{n+m} (E,r)\xrightarrow{R^m} \nu_n  (E,r)\to 0
\end{equation}
for $m,n\geq 0.$  Since $\nu_{n+1}(E,r)\xrightarrow{p^n}\nu_{n+1}(E,r)$ is equal to the composition
$$\nu_{n+1}(E,r)\xrightarrow{R^n}\nu_{1}(E,r)\xrightarrow{\underline{p}^n}\nu_{n+1}(E,r)$$ and 
$\nu_{n+1}(E,r)\xrightarrow{p}\nu_{n+1}(E,r)$ is equal to the composition $$\nu_{n+1}(E,r)\xrightarrow{\underline{p}}\nu_{n+2}(E,r)\xrightarrow{R }\nu_{n+1}(E,r),$$ then we compute from \eqref{BMS8.7preURproof1}
that the sequences 
\begin{gather*}
\nu_{n+1}(E,r)\xrightarrow{p^n}\nu_{n+1}(E,r)\xrightarrow{p}\nu_{n+1}(E,r)\xrightarrow{p^n}\nu_{n+1}(E,r),
\\  \nu_{n+1}(E,r)\xrightarrow{p^n}\nu_{n+1}(E,r) \xrightarrow{R} \nu_{n}(E,r) \to 0
\end{gather*}
are exact.
The complex $\nu _{n+1}(E,r) \otimes _{\bbZ /p ^{n+1}} ^{L} \bbZ /p ^{n}$ is isomorphic to
$$\dots \xrightarrow{p} \nu_{n+1}(E,r)\xrightarrow{p^n}\nu_{n+1}(E,r)\xrightarrow{p}\nu_{n+1}(E,r)\xrightarrow{p^n}\nu_{n+1}(E,r)\to 0 \to \dots,$$
where the right last $\nu_{n+1}(E,r)$ term is at degree $0$, 
then we get the canonical isomorphism $$\nu _{n+1}(E,r) \otimes _{\bbZ /p ^{n+1}} ^{L} \bbZ /p ^{n}\riso \nu_n(E,r)$$ 
Therefore, it follows from \cite[Lem. 3.5.5]{BSproet-15} that $$\nu(E,r)\otimes _{\bbZ}^L\Z/p^n=\nu_n(E,r).$$
The result for the other degrees is proved as in \cite[Prop. 8.4]{BMS}.
\end{proof}



\begin{prop}\label{Prop_ComparisondRWUR}
    If $X_0$ is affine, the choice of a flat formal lift of $X_0$ endowed with a Frobenius lift induces a quasi-isomorphism $$\mr{dR}_{E}^{\bullet,r}\xrightarrow{\sim} R \Gamma_{\rm{pro}\et}\left( X_0, \calN^{\geq r}(E\otimes W\Omega_{X_0}^\bullet)\right )$$ for every $r\in \Z.$
\end{prop}
\begin{proof}
    This follows from \cite[Thm. II.2.1]{Ete88} as in \cite[Prop. 8.7]{BMS}.
\end{proof}

\begin{coro}We have a canonical isomorphism
    \begin{equation}\label{BMS8.21} 
R\Gamma_\mr{syn}(X_0,E(r))= R\Gamma_{\rm{pro}\et}(X_0,\nu(E,r))[-r],
\end{equation} where the left hand side is defined in Definition \ref{synE(r,s)}.
\end{coro}

Let us finish the subsection with the following proposition, which extends \cite[Prop. 4.2]{Bos24}. This result will be crucial in the proof of Theorem \ref{thm_sloperseparated}.
\begin{prop} \label{prop_Bosco} If $k$ is algebraically closed, the groups $H^0_\et(X_0,\nu_n(E,r))$  and $H^{r+1}_\mr{syn}(X_0,E(r))[p^\infty]$ are finite and $H^r_\mr{syn}(X_0,E(r))$ is a free $\Zp$-module of finite rank. 
\end{prop}
\begin{proof}
    The finiteness of $H^0_\et(X_0,\nu_n(r))$ and $H^{r+1}_\mr{syn}(X_0,\nu (r))[p^\infty]$ is proved in \cite[Prop. 4.2]{Bos24}. To deduce the finiteness of $H^0_\et(X_0,\nu_n(E,r))$ we can take a finite étale cover $Y_0\to X_0$ trivialising $E_n$. We then have $H^0_\et(X_0,\nu_n(E,r))\hookrightarrow H^0_\et(Y_0,\nu_n(r))^{\oplus m}$ where $m$ is the rank of $E$, which gives the desired finiteness. We also deduce that $H^{r+1}_\mr{syn}(X_0,E(r))[p^\infty]$ is finite as in [\textit{ibid.}].
    
    \spa
    
    We want to prove now that this implies that $$M\coloneqq H^r_\mr{syn}(X_0,E(r))\stackrel{\eqref{BMS8.21}}=H^0_\mr{proet}(X_0,\nu(E,r)) $$  
    is a free $\Zp$-module of finite rank.
    Note that  $$ H^0_{\mr{pro}\et}(X_0,\nu(E,r))=\lim_n H^0_{\et}(X_0,\nu_n(E,r))$$ since     
\begin{equation}\label{proet-Rlim}
R\Gamma_{\mr{pro}\et}(X_0,\nu(E,r))=R\lim_nR\Gamma_{\et}(X_0,\nu_n(E,r)).
\end{equation}
Via the canonical isomorphism $  \nu (E,r) \otimes^L_{\bbZ_p} \bbZ_p / p^{n} = \nu_n (E,r)$ of Proposition \ref{BMS8.7preUR}, we get
    $$  R\Gamma_{\mr{pro}\et}(X_0,\nu (E,r))\otimes_{\Zp}^L\bbZ_p / p^{n}
    =R \Gamma_{\mr{pro}\et}(X_0, \nu (E,r) \otimes^L_{\bbZ_p}\bbZ_p / p^{n} )
    =R \Gamma_{\et}(X_0, \nu_n (E,r) ).$$  
   Thanks to the corresponding spectral sequence and using \eqref{BMS8.21}, we deduce that $M$ is torsion free and we get the exact sequence
    $$0\to M/p^{n}M\to H^0_\et(X_0,\nu_n(E,r))\to H^{r+1}_\mr{syn}(X_0,E(r))[p^n]\to 0,$$ where the transition maps $$H^{r+1}_\mr{syn}(X_0,E(r))[p^{n+1}]\to H^{r+1}_\mr{syn}(X_0,E(r))[p^n]$$ are given by the multiplication-by-$p$ map. Using that $H^{r+1}_\mr{syn}(X_0,E(r))[p^\infty]$ is finite, passing to the limit, we deduce that $$M^\wedge=\lim H^0_\et(X_0,\nu_n(E,r)) =M.$$

     If we pick a lift to $M$ of a basis of $M/p$ of cardinality $d$ (the dimension is finite because so is $H^0_\et(X_0,\nu_1(E,r))$), this defines a morphism $\Zp^{\oplus d}\to M$. 
     Since $M$ is complete and torsion free, this map is therefore an isomorphism, hence $M$ is free of finite rank.
     \end{proof}

     \subsection{Integral vanishing}
In this section, we use the theory of the de Rham--Witt complex to obtain further integral vanishings in the ordinary setting. We also give a different proof of Proposition \ref{Prop_IntegralBounds} using the sheaves $\nu(E,r).$

\begin{theo}\label{thm_Ordinary}
    Let $Y_0$ a smooth proper $k$-scheme which is ordinary in the sense of \cite[Def. IV.4.12]{IR83} and such that $H^i_\mr{cris}(Y_0/W)$ is torsion free for some $i\geq 0$. Then for any affine open $X_0\subseteq Y_0$, the image of the restriction morphism
    $$H^i_\cris(Y_0/W)\to H^i_\cris(X_0/W)$$ is generated as a $W$-module by $h^{i,0}(Y_0)+h^{i-1,1}(Y_0)$ elements, where $h^{i,j}(Y_0)\coloneqq \mr{dim}_k\left(H^j(Y_0,\Omega^i_{Y_0})\right).$
\end{theo}
\begin{proof} We may assume $k$ algebraically closed. By \cite[Thm. IV.4.13]{IR83} and the assumptions, we have the decomposition \begin{equation}\label{eq_DecompositionDRW}
    H^i_\mr{cris}(Y_0/W)=\bigoplus_{r=0}^iH^{i-r}_\proet(Y_0,W\Omega^{r}_{Y_0})
\end{equation} where each $H^{i-r}_\proet(Y_0,W\Omega^{r}_{Y_0})$ is torsion free and generated as a $W$-module by $h^{r,i-r}(Y_0)$ elements $v\in H^{i-r}_\proet(Y_0,W\Omega^{r}_{Y_0})$ such that $F(v)=p^rv.$ In particular, $$H^i_\proet(Y_0,W\Omega_{Y_0}^{\geq r})\to H^i_\mr{cris}(Y_0/W)$$ is a split injection compatible with the Frobenius structure. The complex $W\Omega_{Y_0}^{\geq r}$ is further endowed with a divided Frobenius $$\varphi_r\colon W\Omega_{Y_0}^{\geq r}\to W\Omega_{Y_0}^{\geq r}$$ defined by $p^{j-r}\mr{F}$ in degree $j\geq r$. Thanks to the compatibility with the crystalline Frobenius, $\varphi_r$ is an isogeny of $H^i_\proet(Y_0,W\Omega^{\geq r}_{Y_0}).$ In addition, the tautological morphism of complexes $$W\Omega_{Y_0}^{\geq r}\to W\Omega_{Y_0}^\bullet$$ factors through $\calN^{\geq r}W\Omega_{Y_0}^\bullet$ (notation as in \S \ref{Sec_dRWUR}) and $\varphi_r$ is compatible with $$\varphi_r\colon \calN^{\geq r}W\Omega_{Y_0}^\bullet\to W\Omega_{Y_0}^\bullet.$$ We get the following commutative diagram 

$$\xymatrix{ {H^i_\proet(Y_0,W\Omega_{Y_0}^{\geq r})}  \ar[r] \ar@{^{(}->}[d]^-{\varphi_r} & {H^i_\proet(Y_0,\calN^{\geq r}W\Omega_{Y_0}^\bullet)} \ar[d]^-{\varphi_r} \\ {H^i_\proet(Y_0,W\Omega_{Y_0}^{\geq r})} \ar@{^{(}->}[r] & {H^i_\mr{cris}(Y_0/W).} }$$

In particular, we have the inclusion $$H^i_\proet(Y_0,W\Omega_{Y_0}^{\geq r})^{\varphi_r-1}\subseteq \mr{Ker}\left( H^i_\proet(Y_0,\calN^{\geq r}W\Omega_{Y_0})\xrightarrow{\varphi_r-\iota} H^i_\mr{cris}(Y_0/W)\right).$$

Since $H^i_\proet(Y_0,W\Omega_{Y_0}^{\geq r})$ is torsion free, we have the equality $$H^i_\proet(Y_0,W\Omega_{Y_0}^{\geq r})^{\varphi_r-1}=H^i_\proet(Y_0,W\Omega_{Y_0}^{\geq r})^{F-p^r}.$$
We deduce that $$H^{i-r}_\proet(Y_0,W\Omega^{r}_{Y_0})=H^{i-r}_\proet(Y_0,W\Omega^{r}_{Y_0})^{F-p^r}\otimes_{\Zp}W$$ is generated as a $W$-module by elements in the image of the composition $$H^i_\syn(Y_0,\Zp(r))\to H^i_\proet(Y_0,\calN^{\geq r}W\Omega_{Y_0}^\bullet)\xrightarrow{\iota} H^i_\mr{cris}(Y_0/W).$$ On the other hand, thanks to Proposition \ref{Prop_IntegralBounds}, we have that $H^i_\syn(X_0,\Zp(r))=0$ for  $r\leq i-2.$ Thanks to the commutative diagram

$$\xymatrix{ {H^i_\syn(Y_0,\Zp(r))}  \ar[r] \ar[d] & {H^i_\cris(Y_0/W)} \ar[d] \\ {H^i_\syn(X_0,\Zp(r))} \ar[r] & {H^i_\mr{cris}(X_0/W).} }$$

this implies that the composition $$H^{i-r}_\proet(Y_0,W\Omega^{r}_{Y_0})\to H^i_\cris(Y_0/W)\to H^i_\cris(X_0/W)$$ vanishes for $r\leq i-2$. We conclude by virtue of \eqref{eq_DecompositionDRW}.
\end{proof}
The following is an integral refinement of \cite[\S 3.3]{DK17}.

\begin{prop}\label{BMS8.5-H2}
If $X_0$ is affine, for $n\geq 2$  we have
\begin{equation*}
    H^n_{\mr{pro}\et}(X_0,\nu(E,r))=0.
\end{equation*}
Equivalently, for $i\not \in \{r,r+1\}$ we have
\begin{equation*}\label{vanish-syn}
   H^i_\mr{syn}(X_0,E(r))= 0. 
\end{equation*}
\end{prop}

\begin{proof}

Since  $X_0$ is affine and $E_n\otimes W_n\Omega^r_{X_0}$ is quasi-coherent, then 
$H^i_{\et}(X_0,E_n\otimes W_n\Omega^r_{X_0})=0$ for  $i\geq 1$. 
Since $$H^0_{\et}(X_0,E_n\otimes W_n\Omega^r_{X_0})\to H^0_{\et}(X_0, E_{n-1}\otimes W_{n-1}\Omega^r_{X_0})$$
are surjective, then by Mittag--Leffler we get $$R^1\lim_n H^0_{\et}(X_0,E_n\otimes W_n\Omega^r_{X_0})=0.$$
Thanks to the fact that $$R \Gamma_{\mr{pro}\et}(X_0,E\otimes W\Omega^r_{X_0} )=  R\lim_nR\Gamma_{\et}(X_0,E_n\otimes W_n\Omega^r_{X_0}),$$ we deduce that $H^i_{\mr{pro}\et}(X_0,E\otimes W\Omega^r_{X_0})=0$ for  $i\geq 1$. 
Using Proposition \ref{BMS8.7preUR}, this yields the desired vanishing.
It follows from \eqref{BMS8.21} that for $i\not \in \{r,r+1\}$ we have
\begin{equation*}
   H^i_\mr{syn}(X_0,E(r))= 0. 
\end{equation*}
\end{proof}

\subsection{Fractional logarithmic de Rham--Witt sheaves}\label{Sec_FractionalLogarithmicdRW}In this section we introduce fractional versions of the logarithmic de Rham--Witt sheaves of Illusie and Milne. These compute the slopes of crystalline cohomology beyond the classical integral setting (Theorem \ref{thm_FractionalSlopesCrystalline}).

\spa

Let $s$ be a positive integer and recall that $\Lambda^s\coloneqq \Zp[\pi]/(\pi^s-p)$. As in \S \ref{Sec_dRWUnitRoot}, we define a fractional Nygaard filtration on the ramified de Rham--Witt complex $$W^s\Omega^\bullet_{X_0}\coloneqq W \Omega^\bullet_{X_0} \otimes_{\Zp} \Lambda^s =\bigoplus_{i=0}^{s-1}\pi^iW\Omega^\bullet_{X_0}.$$ If $E$ is a unit-root $F$-crystal, there is also a twisted version $$E\otimes W^s\Omega^\bullet_{X_0}\coloneqq \left( E\otimes W \Omega^\bullet_{X_0}\right) \otimes_{\Zp} \Lambda^s.$$
We denote by $F\colon  E\otimes W^s\Omega^\bullet_{X_0} \to  E\otimes W^s\Omega^\bullet_{X_0}$ the Frobenius action induced by $F\colon E\otimes W\Omega^\bullet_{X_0}\to E\otimes W\Omega^\bullet_{X_0} $ and the identity on $ \Lambda^s$. Similarly, we write $\rm F$ for the action on $E\otimes W^s\Omega^\bullet_{X_0}$ induced by $\mr{F}$ acting on $ E\otimes W\Omega^\bullet_{X_0}$. 
Even in this case, we have  $F = p^j {\rm F}$ on $E\otimes W^s\Omega^j_{X_0}$. We also write $$E_n\otimes W^s_n\Omega^\bullet_{X_0}\coloneqq \left(E_n\otimes W_n \Omega^\bullet_{X_0}\right) \otimes_{\Zp} \Lambda^s$$

and $$R\colon E_{n+1}\otimes W^s_{n+1}\Omega^\bullet_{X_0}\to E_n\otimes W_n^s\Omega^\bullet_{X_0}$$ for the induced reduction map. A local section of $E\otimes W^s\Omega^\bullet_{X_0}$ is also denoted by $(v_0,\dots,v_{s-1})\coloneqq  \sum_{i=0}^{r-1} \pi^i v_i$ where $v_0,\dots,v_{s-1}\in E\otimes  W\Omega^\bullet_{X_0}$.

\begin{nota}\label{nota_EucledeanDivision}
    For $r\in \Z$, we write $r=a+bs$ with $0\leq a <s$ and $b\in \Z$. We also denote by $\gamma$ the fraction $r/s\in \Q$.
\end{nota}

 We define \begin{equation}
    \label{NrWs}\calN^{\geq r}\left(E\otimes W^s\Omega_{X_0}^\bullet\right)\coloneqq \left(\bigoplus_{i=0}^{a-1}\pi^i\calN^{\geq b+1}\left(E\otimes W\Omega^\bullet_{X_0}\right)\right)\oplus \left(\bigoplus_{i=a}^{s-1}\pi^i\calN^{\geq b}\left(E\otimes W\Omega^\bullet_{X_0}\right)\right).
\end{equation}
As recalled in \S \ref{Sec_dRWUnitRoot}, there are well-defined divided Frobenius morphism 
$$\varphi_i \colon \calN^{\ge i}\left(E\otimes W\Omega^\bullet_{X_0}\right)\to E\otimes W\Omega^\bullet_{X_0}.$$ We define $$\varphi_{r,s}\colon \calN^{\geq r}\left(E\otimes W^s\Omega_{X_0}^\bullet\right)\to E\otimes W^s\Omega_{X_0}^\bullet $$
 by setting
$$\varphi_{r,s}(v_0,\dots,v_{s-1})\coloneqq(\varphi_{b}(v_a),\dots,\varphi_b(v_{s-1}),\varphi_{b+1}(v_0),\dots,\varphi_{b+1}(v_{a-1})).$$ 
We have the compatibilities $$   \pi^a \varphi_{r,s}= \varphi_{bs,s}= \varphi_{b} \otimes \id= F / p^b$$ and we have
$\varphi_{r,s}= F/\pi^{r}$.
The complex $\calN^{\geq r}\left(E\otimes W^s\Omega_{X_0}^\bullet\right)$ is the largest subcomplex of 
$E\otimes W^s\Omega_{X_0}^\bullet$ on which $F$ is divisible by $\pi^{r}$. Thanks to 
\eqref{eq_FractionalCompatibility} 
and Proposition \ref{Prop_ComparisondRWUR} we have a quasi-isomorphism 
\begin{equation}\label{eq_ComparisonFractionaldRW}
    \mr{dR}_{E,s}^{\bullet,r}\xrightarrow{\sim}  R \Gamma_{\rm{pro}\et}\left( X_0, \calN^{\geq r}\left(E \otimes W^s\Omega^\bullet_{X_0} \right)\right) ,
\end{equation}
which commutes with the Frobenius action.

\begin{prop}\label{prop_fundamentalFractionaldRW}
    The morphism $$\varphi_{r,s}-\iota \colon \calN^{\geq r}\left(E\otimes W^s\Omega^j_{X_0}\right)\to E\otimes  W^s\Omega^j_{X_0}$$ is an isomorphism for $j\neq b$.
\end{prop}
\begin{proof}We may assume $X_0$ affine. 
If $b<j$ then $\varphi_{r,s}=\pi^{js}\mr{F}/\pi^r=\pi^{js-r}\mr{F}$ is topologically nilpotent. Hence $$\varphi_{r,s}-\iota\colon \calN^{\geq r}\left(E\otimes W^s\Omega^j_{X_0}\right)=E\otimes W^s\Omega^j_{X_0}\to E\otimes W^s\Omega^j_{X_0}$$ is an isomorphism.

\spa

If $b>j$, then we have a commutative diagram
\[\xymatrix@C=2cm{
E\otimes W^s\Omega_{X_0}^j\ar@/_1cm/[rr]_{1-\pi^{r-s(j+1)}V}\ar[r]^{\pi^{r-s(j+1)}V}_{\simeq}&\calN^{\ge r}\left(E\otimes W^s\Omega^j_{X_0}\right)\ar[r]^{\varphi_{r,s}-\iota}&E\otimes W^s\Omega^j_{X_0},
}\]
in which the top left horizontal map is an isomorphism by definition of the filtration $\cN^{\geq r}$. 
Its curved arrow is also an isomorphism because $W^s\Omega^j_{X_0}$ is $p$-adically (and $V$-adically complete for the case where $a=0$ and $b= j-1$).
Hence, so is $\varphi_{r,s}-\iota$.
\end{proof}

\begin{defi}\label{def_dRWFractionalTate}
  If $r=a+bs$ as in Notation \ref{nota_EucledeanDivision} with $\mr{gcd}(r,s)=1$, we define \begin{gather*}
        \nu(E,\gamma)\coloneqq \Cocone \left (\calN^{\geq r}\left(E\otimes W^s\Omega_{X_0}^\bullet\right)\xrightarrow{\varphi_{r,s}-\iota}E\otimes W^s\Omega_{X_0}^\bullet \right ) [b]\\
        \underset{\ref{prop_fundamentalFractionaldRW}}{\riso}  \Cocone \left (\calN^{\geq r}\left(E\otimes W^s\Omega_{X_0}^b\right)\xrightarrow{\varphi_{r,s}-\iota}E\otimes W^s\Omega_{X_0}^b \right )
    \end{gather*} as a complex of pro-étale sheaves. We endow $\nu(E,\gamma)$ with the composition $$\nu(E,\gamma)\to \calN^{\geq r}\left(E\otimes W^s\Omega_{X_0}^\bullet\right)[b]\to E\otimes W\Omega_{X_0}^\bullet[b],$$
    where the last morphism is obtained by setting $\pi^i=0$ for $i>0.$ We also write $\nu(\gamma)$ for $\nu(\Zp,\gamma).$   
\end{defi}


\begin{prop}\label{prop_FractionalLogarithmic}If $E$ is a unit-root $F$-crystal over $X_0$ and $r=a+bs$ with $0\leq a < s$, we have a quasi-isomorphism  $$R\Gamma_\syn(X_0,E(\gamma))\riso \nu(E,\gamma)[-b],$$ where the left hand side is defined in \S \ref{synE(r,s)}.

\end{prop}
\begin{proof}
    This follows from \eqref{eq_ComparisonFractionaldRW}.
\end{proof}

For the next result we need a torsion complement of Proposition \ref{Prop_IntegralBounds}.

\begin{lemm}\label{lem_FractionalInjectivitydRW}For  $a,b,n,s\in \bbN$ with $a>0$, we have the inclusion $$\Ker \left(E_{n+1}\otimes W_{n+1}^s\Omega_{X_0}^b \xrightarrow{\varphi_{bs,s} -\pi^a R} E_n\otimes  W_n^s\Omega_{X_0}^b\right)\subseteq \Ker(R).$$\end{lemm}
\begin{proof}We may assume $X_0$ affine. Let $(v_0,\dots,v_{s-1})$ be a section of $E_{n+1}\otimes W_{n+1}^s\Omega_{X_0}^b.$ We have  $\varphi_{bs} (v_0,\dots,v_{s-1}) 
    = (\mathrm{F} (v_0),\dots ,\mathrm{F} (v_{s-1}))$ and $\pi^a (v_0,\dots,v_{s-1}) =( p v_{s-a},\dots, p v_{s-1}, v_{0},\dots,v_{s-a-1}).$
Therefore, the section is killed by $\varphi_{bs,s} -\pi^a R$ if and only if \begin{equation}\label{eq_FRsystem}
    F(v_0)=pR(v_{s-a}), \dots, F(v_{a-1})=pR(v_{s-1}), F(v_a)=R(v_{0}),\dots, F(v_{s-1})=R(v_{s-a-1}).
\end{equation}
The proof is by induction on $n$. For $n=1$, we have that $$pR(v_{s-a})=\dots=pR(v_{s-1})=0,$$ which implies $F(v_0)=\dots =F(v_{a-1})=0$ (note that we assumed $a>0$). By \cite[(II.1.2.8.1)]{Ete88}, we know that $$\Ker(\mr{F})\subseteq \Ker(p)=\Ker(R),$$ which implies that $v_0,\dots,v_{a-1}\in \Ker(R).$ In turn, by \eqref{eq_FRsystem}, we deduce that $$F(v_a)=\dots=F(v_{2a-1})=0.$$ As above, we get that $v_a,\dots,v_{2a-1}\in \Ker(R).$ Repeating this process, this shows that $$\Ker(\varphi_{bs,s} -\pi^a R)\subseteq \Ker(R),$$ as we wanted.

\spa

Let us suppose now that we know the result for $n-1$. We then have $$\Ker \left(E_{n+1}\otimes W_{n+1}^s\Omega_{X_0}^b \xrightarrow{\varphi_{bs,s} -\pi^a R} E_n\otimes  W_n\Omega_{X_0}^b\right)\subseteq \Ker(R^{2}).$$ Again, let $(v_0,\dots,v_{s-1})\in \Ker(\varphi_{bs} -\pi^a R).$ We have that $$pR(v_{s-a})=\dots=pR(v_{s-1})=0$$ since $\Ker(pR)=\Ker(R^2).$ We can then argue as above and prove the final result.
\end{proof}

\begin{prop}\label{prop_InjectivityFractionaldRW} For $a,r,s$ with $a>0,$ the morphism
    $$\calN^{\geq r}\left( E\otimes  W^s\Omega_{X_0}^b\right)\xrightarrow{\varphi_{r,s}-1}E\otimes W^s\Omega_{X_0}^b.$$ is an injective morphism of pro-étale sheaves.
\end{prop}
\begin{proof}We have $$\pi^a(\varphi_{r,s}-1)=\varphi_{bs,s} -\pi^a.$$ Since $E\otimes W\Omega^b_{X_0}$ is $\pi$-torsion free, we are reduced to compute the kernel of $\varphi_{bs,s} -\pi^a.$  This can be checked by passing to truncations thanks to Lemma \ref{lem_FractionalInjectivitydRW}.
    
\end{proof}

\begin{exam}\label{exam_Coker12}
    Let us consider the case when $X_0=\Spec(A_0)$ is affine and $s=2,r=1$. The morphism $$\varphi_{1,2}-1\colon VW(A_0)\oplus \pi W(A_0)\to W(A_0)\oplus \pi W(A_0)$$ is injective by Proposition \ref{prop_InjectivityFractionaldRW}. Let $\alpha$ be the automorphism of $W(A_0)\oplus \pi W(A_0)$ which sends $(v_0,v_1)$ to $(v_0+F(v_1),v_1).$ Projecting to the first component, we have that $$M\coloneqq\mr{Coker}(\alpha\circ \varphi-\alpha)=\mr{Coker}(W(A_0)\xrightarrow{F-V}W(A_0))$$ is non-trivial. This group can neither be trivialised by a pro-étale cover, since the Verschiebung induces a filtration on $M$ with graded pieces isomorphic to $A_0/A_0^p$.
    \end{exam}
Let us show now the comparison theorem between the pro-étale cohomology of $\nu(\gamma)$ and the slope $\gamma$ piece of crystalline cohomology.

\begin{theo}\label{thm_FractionalSlopesCrystalline}
   Let $k$ be an algebraically closed field of characteristic $p$, let $X_0$ be a smooth proper scheme over $k$, and let $\gamma$ be a rational number with $\lfloor\gamma\rfloor=b.$ The structural morphism $\nu(\gamma)\to  W\Omega_{X_0}^\bullet[b]$ induces an isomorphism $$H^i_\proet(X_0,\nu(\gamma))_K\riso H^{i+b}_\cris(X_0/W)_K^{[\gamma]}$$ of $K$-vector spaces for $i\geq 0$. Moreover, $\nu(\gamma)$ is quasi-isomorphic to $W\Omega_\mr{log}^b[0]$ when $\gamma\in \Z$ and to a pro-étale sheaf sitting in degree $1$ otherwise.
\end{theo}
\begin{proof}Write $\gamma=r/s$ with $(r,s)=1$. By combining Lemma \ref{lem_compatibilitysyntomic} and Proposition \ref{prop_FractionalLogarithmic}, we get an isomorphism 
$$H^i_\mr{syn}(X_0,\Qp(r,s))\riso H^i_\proet(X_0,\nu(\gamma))_{\Qp}.$$

Thus, by Lemma \ref{lem_Vrs}, we deduce that $$H^i_\proet(X_0,\nu(\gamma))_{\Qp}\to H^{i+b}_\cris(X_0/W)_{K}^{F^s-p^r}$$ is surjective. Thanks to Dieudonné--Manin classification, we further infer that $$ H^i_\proet(X_0,\nu(\gamma))_{K}\to H^{i+b}_\cris(X_0/W)_K^{[\gamma]}$$ is surjective -- the field $k$ being algebraically closed. For the injectivity, we use the fact that $H^{i+b-1}_\cris(X_0/W)_K$ is a finite-dimensional $F$-isocrystal, as $X_0$ is proper, which implies that $$H^{i+b-1}_\cris(X_0/W)_K\xrightarrow{F^s-p^r}H^{i+b-1}_\cris(X_0/W)_K$$ is surjective (see for example \cite[Lem. 5.1.3]{Gra25}). The last statement follows instead from Proposition \ref{prop_fundamentalFractionaldRW} and Proposition \ref{prop_InjectivityFractionaldRW}.
\end{proof}

\subsection{The de Rham--Witt complex of log schemes}
We discuss here the de Rham--Witt complex of log schemes. For simplicity, we do not talk about coefficients in this section.
\begin{lemm}\label{SNCD-Cartier}
    Let $Y_0$ be a smooth $S_0$-scheme, $D_0$ be a strict normal crossing divisor of $Y_0$
    and $M (D_0)$ be the induced log structure. If we set $Y_0^\sharp \coloneqq (Y_0, M (D_0))$, then $Y_0^\sharp /S_0$ is log smooth and of Cartier type in the sense of \cite[Def. 4.8]{Kato-logFontaine-Illusie}.  
\end{lemm}

\begin{proof}
Thanks to \cite[Cor. 4.4]{Kato-logFontaine-Illusie} we get that  $Y^\sharp_0/S_0$ is integral.
The property that $Y^\sharp_0/S_0$ is log smooth and of Cartier type is therefore a consequence of  \cite[4.5.2.14.1]{Car25} and of the fact that the map $\bbN^r \stackrel{p} \to \bbN^r$ is exact.
\end{proof}

\begin{empt}
    We keep notation as in Lemma \ref{SNCD-Cartier}.
For any $n\in\bbN$, Lorenzon defined the sheaf $W_n\Omega_{Y^\sharp_0/S_0,\log}^r$ as the étale $W_n$-submodule of $W_n\Omega^r_{Y^\sharp_0/S_0}$  locally generated by $\mr{dlog}(m_1)\wedge\dots\wedge \mr{dlog}(m_r)$ where $m_i$ are local sections of $M(D_0)^\mr{grp}$ (see \cite[Page 257]{Lor02}).
According to \cite[Cor. 2.14]{Lor02} (using Lemma \ref{SNCD-Cartier}) we get the exact sequence of étale pro-sheaves on $Y_0$
\begin{equation}\label{I.5.7.2-Il-logpre}
0\to W_{\bullet} \Omega_{Y^\sharp_0/S_0,\log}^r\to\ W_{\bullet}\Omega^r_{Y^\sharp_0/S_0}\xrightarrow{{\rm F}-1}W_{\bullet}\Omega^r_{Y^\sharp_0/S_0} \to 0.
\end{equation} 
We define a pro-étale $\bbZ_p$-sheaf by setting $W\Omega_{Y^\sharp_0/S_0,\log}^r\coloneqq \lim_n W_n\Omega_{Y^\sharp_0/S_0,\log}^r$ (we view $W_n\Omega_{Y^\sharp_0/S_0,\log}^r$ as a pro-étale $\bbZ_p$-sheaf and the limit is computed in the latter category). By \eqref{I.5.7.2-Il-logpre}, the sheaf
$W\Omega_{Y^{\sharp}_0,\log}^r$ coincides with the derived inverse limit $R\lim_nW_n\Omega^r_{Y^{\sharp}_0,\log}$ 
and we have the exact sequence of derived $p$-complete pro-étale $\bbZ_p$-sheaves on $Y_0$
\begin{equation}\label{I.5.7.2-Il-log}
0\to W\Omega_{Y^\sharp_0/S_0,\log}^r\to\ W\Omega^r_{Y^\sharp_0/S_0}\xrightarrow{{\rm F}-1}W\Omega^r_{Y^\sharp_0/S_0} \to 0.
\end{equation} 
It follows from \cite[Cor. 1.19 and (2.11)]{Lor02} 
that $W\Omega_{Y^\sharp_0/S_0,\log}^r$ is $p$-torsion free and 
$$W\Omega_{Y^\sharp_0/S_0,\log}^r\otimes^L_{\bbZ_p} \bbZ /p^n\to W_n\Omega_{Y^\sharp_0/S_0,\log}^r $$  is a quasi-isomorphism.

\end{empt}

\begin{lemm}\label{lem_compatibilitylogOmega}If $D_0$ is a strict normal divisor of $Y_0$ with complement $j\colon X_0\hookrightarrow Y_0,$ then
    $$j_*W\Omega^r_{X_0,\mr{log}}=W\Omega^r_{Y^{\sharp}_0,\mr{log}}.$$
\end{lemm}
\begin{proof}

By definition, we have that $M(D_0)=j_*\calO^*_{X_0}$ (see \cite[1.5]{Kato-logFontaine-Illusie}).  This implies that $$j_*W_n\Omega^r_{X_0,\mr{log}}=W_n\Omega^r_{Y^{\sharp}_0,\mr{log}}$$ for every $n$. We conclude taking limits (use \stack{099N}).
\end{proof}



\begin{defi}\label{def-W-Nygaard} If $Y_0 = \Spec(A_0)$, we write $W\Omega_{A_0^\sharp}^\bullet\coloneqq \Gamma ( Y_0, W\Omega_{Y_0^\sharp /S_0}^\bullet)$. As in \cite[Def. 8.1]{BMS}, 
let $\calN^{\geq i}W\Omega_{A_0^\sharp}^\bullet\subseteq W\Omega_{A_0^\sharp}^\bullet$ be the subcomplex
\[
p^{i-1}VW(A)\to p^{i-2}VW\Omega^1_{A_0^\sharp}\to\cdots\to pVW\Omega^{i-2}_{A_0^\sharp}\to VW\Omega^{i-1}_{A_0^\sharp}\to W\Omega_{A_0^\sharp}^i\to W\Omega_{A_0^\sharp}^{i+1}\to\cdots.\]
In general, we define $\calN^{\geq i}W\Omega_{Y_0^\sharp}^\bullet\subseteq W\Omega_{Y_0^\sharp}^\bullet$ so that if $U_0 =\Spec (A_0)$ is an affine open of $Y_0$ we have
$\Gamma ( U_0, W\Omega_{Y_0^\sharp /S_0}^\bullet)= W\Omega_{A^\sharp}^\bullet$. We define $\varphi_r\coloneqq F/p^r$ and $\iota$ as in \S \ref{SubSec_Nygaard}.
\end{defi}

\begin{prop}\label{prop_BMS8.7prelog}
    The sequence of complexes of pro-\'etale sheaves
\begin{equation*}
     0\to W\Omega_{Y^{\sharp}_0,\log}^r[-r]\to \calN^{\ge i}W\Omega^\bullet_{Y^{\sharp}_0}\xrightarrow{\varphi_r-\iota}W\Omega^\bullet_{Y^{\sharp}_0}\to 0
\end{equation*}is exact.
\end{prop}
\begin{proof}
   Replacing \cite[Thm. I.5.7.2]{Ill79} by \eqref{I.5.7.2-Il-log},
this follows as in \cite[Prop. 8.4]{BMS}. 
\end{proof}
\begin{defi}\label{Def_logsyn}
    We define $$R\Gamma_\mr{syn}(Y^\sharp_0,\Zp(r))\coloneqq\Cocone (\calN^{\ge r}W\Omega^\bullet_{Y^{\sharp}_0}\xrightarrow{\varphi_r-\iota}W\Omega^\bullet_{Y^{\sharp}_0}).$$
\end{defi}

We have a log variant of Proposition \ref{BMS8.5-H2}.

\begin{prop}\label{BMS8.5-H2-log}With notation as in Lemma \ref{SNCD-Cartier}, 
when $Y_0$ is affine, we have 
\begin{equation*}
    H^n_{\mr{pro}\et}(Y^{\sharp}_0,W\Omega_{Y^{\sharp}_0,\log}^r)=0
\end{equation*}
for $n\geq 2$. Equivalently, for $i\notin \{r,r+1\}$ we have $$H^i_\mr{syn}(Y_0^\sharp,\Zp(r))=0.$$
\end{prop}

\begin{proof}
This is proven as in Proposition \ref{BMS8.5-H2-log} by replacing Proposition \ref{BMS8.7preUR} with Proposition \ref{prop_BMS8.7prelog}.
\end{proof}

\begin{empt}
    Let $\fY$ be a smooth $W$-formal scheme, $\fD$ be a relative to $\fY/W$ strict normal crossing divisor, $\fX$ be the open of $\fY$ complementary to $\fD$. Suppose there exists a lifting $F \colon \fY \to \fY$ of the absolute Frobenius of $Y_0$. Set $\fY^\sharp \coloneqq (\fY, M (\fD))$ where $M (\fD)$ is the log structure induced by $\fD$, and $Y_n^{\sharp}$ for its reduction modulo $p^{n+1}$. Thanks to \cite[Thm. 4.19]{Hyodo-Kato-logcrystcoho}, we have the quasi-isomorphism
    $$W_n \Omega_{Y_0^{\sharp}/S_0}^\bullet \riso R u_{Y_0^{\sharp}/S_n} (\cO_{Y_0^{\sharp}/S_0})$$ compatible with Frobenius. In addition, from \cite[Prop. 2.20]{Hyodo-Kato-logcrystcoho}, we have the quasi-isomophism 
$R u_{Y_0^{\sharp}/S_n} (\cO_{Y_0^{\sharp}/S_0})\riso \Omega^{\bullet}_{Y_n^{\sharp}/S_n}$. Taking limits this yields the quasi-isomorphism 
\begin{equation}
    \label{log-dRW-dR}W\Omega_{Y_0^{\sharp}/S_0}^\bullet\riso \Omega^{\bullet}_{\fY^{\sharp}/\fS},
\end{equation} which is compatible with the Frobenius structure.
This yields 
\begin{gather}\label{log-cris-rig=}
    R \Gamma_{\mr{log\m cris}}((Y_0,D_0)/W)_K\riso R \Gamma_\mr{pro\et}(Y_0,W\Omega^\bullet_{Y_0^\sharp})_K\riso 
R \Gamma (Y_0, \Omega^{\bullet}_{\fY^{\sharp}/\fS})_K\riso R \Gamma_{\rig} (X_0/K),
\end{gather}
where the last isomorphism is \cite[1.3.6]{CT12}.
\spa

Suppose $Y_0= \Spec (A_0)$ is affine and $A$ is a $W$-smooth lift of $A_0$, endowed with a Frobenius lift.  As in \cite[Prop. 8.7]{BMS}, we deduce the existence of a quasi-isomorphism
\begin{equation}
\label{log-Prop. 8.7BMS} p^{(i-\bullet)_+} \Omega^{\bullet}_{\widehat{A} ^{\sharp}}\xrightarrow{\sim} \calN^{\geq i}W\Omega_{A^\sharp_0}^\bullet.
\end{equation}
Thanks to \eqref{log-Prop. 8.7BMS}, the {Nygaard filtration} can be reinterpreted as the filtration on the complex $\Omega^{\bullet}_{\widehat{A} ^{\sharp}}$, defined by 
$p^{(i-\bullet)_+} \Omega^{\bullet}_{\widehat{A} ^{\sharp}} $.
As in \S\ref{sec_FGauge} the action of Frobenius refines to a morphism  $\varphi \colon \widetilde{\Omega}^{\bullet}_{\widehat{A} ^{\sharp}}  \to  \Omega^{\bullet}_{\widehat{A} ^{\sharp}} $. Since the log version of the Cartier isomorphism holds \cite[Thm. V.4.1.3]{Ogus-Logbook}, then we get the following result.
\end{empt}

\begin{prop}\label{thm_Mazur-log}
The morphism $$\varphi\colon \widetilde{\Omega}^{\bullet}_{\widehat{A} ^{\sharp}}\to \Omega^{\bullet}_{\widehat{A} ^{\sharp}} $$ is a quasi-isomorphism.

\end{prop}

\begin{defi}
      For $s\in \Z_{>0}$, we define $$W^s\Omega^\bullet_{Y_0^\sharp}\coloneqq W \Omega^\bullet_{Y_0^\sharp} \otimes_{\Zp} \Lambda^s.$$ We also define $\calN^{\geq r}$ and $\varphi_{r,s}$ as in \S \ref{Sec_FractionalLogarithmicdRW}.
\end{defi}

\begin{prop}\label{prop_fundamentalLogFractionaldRW}If $r=a+bs$ as in Notation \ref{nota_EucledeanDivision} with $\mr{gcd}(r,s)=1$, the morphism $$\varphi_{r,s}-\iota\colon \calN^{\geq r} W^s\Omega^j_{Y_0^\sharp}\to  W^s\Omega^j_{Y_0^\sharp}$$ is an isomorphism for $j\neq b$ and it is injective if $a\neq 0.$
\end{prop}
\begin{proof}
    The proof is the same as in Proposition \ref{prop_fundamentalFractionaldRW} and Proposition \ref{prop_InjectivityFractionaldRW}.
\end{proof}

\begin{defi}
    If $r=a+bs$ and $\gamma=r/s$, as in Notation \ref{nota_EucledeanDivision}, and $\mr{gcd}(r,s)=1$, we define \begin{gather*}
        \nu^\sharp(\gamma)\coloneqq \Cocone \left (\calN^{\geq r} W^s\Omega_{Y_0^\sharp}^\bullet\xrightarrow{\varphi_{r,s}-\iota} W^s\Omega_{Y_0^\sharp}^\bullet \right ) [b] \\\underset{\ref{prop_fundamentalLogFractionaldRW}}{\riso}  \Cocone \left (\calN^{\geq r}W^s\Omega_{Y_0^\sharp}^b\xrightarrow{\varphi_{r,s}-\iota}W^s\Omega_{Y_0^\sharp}^b \right )
    \end{gather*} as a complex of pro-étale sheaves over $Y_0$. By Proposition \ref{prop_BMS8.7prelog} and Proposition \ref{prop_fundamentalLogFractionaldRW}, the complex $\nu^\sharp(\gamma)$ is quasi-isomorphic to a pro-étale sheaf in degree $0$ when $a=0$ and in degree $1$ when $a\neq 0.$
    \end{defi}

\section{Separation properties of rigid cohomology}\label{Sec_SeparationDefect}
Let $\calV$ be an object of $\mathrm{CDVR}^\star(\Z_p)$ with residue field $k$ and fraction field $K$. Assume $\calV$ endowed with a lift of the Frobenius $\sigma \colon  \cV\riso \cV$.

\subsection{The affinoid topology}\label{Sec_AffinoidTopology}
Let $X_0= \Spec(A_0)$ be a smooth affine scheme over $k$.  
Choose $A$ a smooth $\cV$-algebra which is a lifting of $A_0$. 
Let $A^\dag$ be the $p$-adic weak completion of $A$ as $\cV$-algebra. 
Let $F\colon A^\dag \to A^\dag$ be a lifting of the absolute Frobenius of $A_0$ above $\sigma \colon \cV \riso \cV$.

\begin{defi}\label{ntn-Hconv-s-ns} 
Let $\mathcal{M}$ be an $\widehat{A}_K$-module of finite type having an integral connection and a Frobenius structure, i.e. 
an isomorphism of the form $F^* \mathcal{M} \riso \mathcal{M}$ (according to \cite[2.5.7]{Be96}, this corresponds to an convergent $F$-isocrystal on $X_0/\cV$). 
Let $$\nabla_{\mathcal{M}}^i \colon \Omega_{\widehat{A}_K}^{i} \otimes_{\widehat{A}_K} \mathcal{M} \to \Omega_{\widehat{A}_K}^{i+1} \otimes_{\widehat{A}_K} \mathcal{M}$$ be the map induced by the connection of $\mathcal{M}$.
Then $H^i_\mr{conv}(X_0,\mathcal{M})= \Ker \nabla_{\mathcal{M}}^{i}/\im \nabla_{\mathcal{M}}^{i-1}$.
The affinoid topology of $\Omega_{\widehat{A}_K}^{i} \otimes_{\widehat{A}_K} \mathcal{M}$ is 
the topology given by its structure of $\widehat{A}_K$-modules of finite type.
For instance, when $\mathcal{M}$ is the constant coefficient, this topology on the abelian group $\Omega_{\widehat{A}_K}^{i}$ has a basis of open neighborhoods of $0$ given by the subgroups $p^n\Omega_{\widehat{A}}^{i}$
for all $n$.

\spa

We denote by $\overline{\im} \nabla_{\mathcal{M}}^{i-1}$ the closure of $\im \nabla_{\mathcal{M}}^{i-1}$ in $\Omega_{\widehat{A}_K}^{i}$ (or in the closed subspace $\Ker \nabla_{\mathcal{M}}^{i}$) for the affinoid topology and we set 
$$H^i_{\mr{conv}} (X_0,\mathcal{M})^{\sep} \coloneqq \Ker \nabla_{\mathcal{M}}^{i}/\overline{\im} \nabla_{\mathcal{M}}^{i-1}.$$

If $B$ is another smooth $\cV$-algebra which is a lift of $A_0$, then we get a (non-canonical) isomorphism
$\widehat{A} \riso \widehat{B}$. Hence, $H^i_{\mr{conv}} (X_0,\mathcal{M})^{\sep}$ is well-defined and it is the maximal separated quotient of $H^i_{\mr{conv}} (X_0,\mathcal{M})$. Similarly, 
$$N^i_{\rm{conv}} (X_0,\mathcal{M})\coloneqq \overline{\im} \nabla_{\mathcal{M}}^{i-1}/\im \nabla_{\mathcal{M}}^{i-1}$$
is well-defined and it is the $p$-adic closure of $0$ of $H^i_{\mr{conv}} (X_0,\mathcal{M})$.
\end{defi}

\begin{lemm}\label{lemmHiconvpre}
Let $V' \overset{u}{\to} V  \overset{v}{\to} V''$ be two morphisms of $\cV$-modules such that $V$ and $V''$ are $p$-torsion free and $v\circ u =0$. 
Let denote $V^{\prime}_{K}  \overset{u_K}{\to} V_K \overset{v_K}{\to} V^{\prime\prime}_{K}$ the induced map by tensoring with $K$.
Set $H \coloneqq \mr{Ker}(v) /\im (u)$. 
We endow $V_K$ with a locally convex $K$-vector space structure so that a basis of neighborhood of zero is given by 
$\{p^n \alpha ( V) \}_{n\in \bbN}$. Let $\overline{\im}(u_K)$ be the closure of $\im(u_K)$ for the induced topology. Then, we get the equality
$$\overline{\im}(u_K) /\im(u_K) = (H / H_\tors)_{\rm div}= \bigcap_{n\in \bbN} p^n (H / H_\tors) .$$
\end{lemm}

\begin{proof}
The last equality follows from the $p$-torsion freeness of $H / H_\tors$. 
Let $\cH  \coloneqq \mr{Ker}(v_K) /\im(u_K)$, $\cK \coloneqq \Ker v_K$, $K \coloneqq \Ker v $ and $\alpha \colon K \to \cK$ be the canonical map. 
Let $\varpi \colon K \to H / H_\tors$ and  $[-]\colon \cK \to \cH$ be the projections.
The map $\alpha$ induces the inclusion $\beta\colon H/H_\tors\hookrightarrow \cH$ given by  $\beta (\varpi (x) ) = [\alpha ( x)]$.
Let $x \in \cK$. 

\spa

(i) Suppose $[x]\in \overline{\im}(u_K) /\im(u_K)$, which is equivalent to saying 
$x \in \overline{\im}(u_K)$. 
For every $m\in \bbN$, there exist $y_{m}\in \im(u_K)$,
$x_{m}\in V $ such that $x= y_{m}+p^m \alpha (x_{m}).$ Since
$V$ and $V''$ are $p$-torsion free, then $x_m \in K$.
Hence, $[x]=p^m \beta (\varpi (x_m))$.

\spa

(ii)  Suppose $[x] \in (H / H_\tors)_{\rm div}$, i.e., for any $m\in \bbN$, there exists $x_{m}\in K$ such that $[x]=p^m \beta (\varpi (x_m))
=[p^m \alpha ( x_m)]$. Hence, there exists $y_m \in \im(u_K)$
such that $x= y_{m}+p^m \alpha (x_{m})$, and we are done.
\end{proof}

\begin{lemm}\label{lemmHiconv}
Let $E$ be a locally free flat connection of finite rank over $X_0/\cV$.
   We have $$N^i_{\mr{conv}}(X_0,E_K)=( H^i(X_0,\mr{dR}^{\bullet}_{E})/\mr{tors})_\mr{div}.$$
\end{lemm}

\begin{proof}
This follows from Lemma \ref{lemmHiconvpre}.
\end{proof}

\begin{empt}\label{ntn-Hrig-s-ns}
Let $M^\dag$ be an $A^\dag_K$-module of finite type endowed with an integral flat connection and a Frobenius structure, i.e. 
an isomorphism of the form $F^* M^\dag \riso M^\dag$ (according to \cite[2.5.7]{Be96}, this corresponds to an overconvergent $F$-isocrystal on $X_0/K$). 
Let $$\nabla_{M^\dag}^i \colon \Omega_{A^\dag_K}^{i} \otimes_{A^\dag_K} M^\dag \to \Omega_{A^\dag_K}^{i+1} \otimes_{A^\dag_K} M^\dag$$ be the map induced by the connection of $M^\dag$.
Then $H^i_\mr{rig}(X_0,M^\dag)= \Ker \nabla_{M^\dag}^{i}/\im \nabla_{M^\dag}^{i-1}$.
The affinoid topology of $\Omega_{A^\dag_K}^{i} \otimes_{A^\dag_K} M^\dag$ is 
the topology induced by the affinoid topology of $\left (\Omega_{A^\dag_K}^{i} \otimes_{A^\dag_K} M^\dag\right )\otimes_{A^\dag_K}  \widehat{A}_K$ 
via the injection 
$$\Omega_{A^\dag_K}^{i} \otimes_{A^\dag_K} M^\dag \hookrightarrow \left (\Omega_{A^\dag_K}^{i} \otimes_{A^\dag_K} M^\dag\right )\otimes_{A^\dag_K}  \widehat{A}_K.$$ 
For instance, when $M^\dag$ is the constant coefficient, this topology on the abelian group $\Omega_{A^\dag_K}^{i}$ has a basis of open neighborhood of $0$ given by the subgroups $p^n\Omega_{A^\dag}^{i}$
for all $n$ (because $p^n\Omega_{A^\dag}^{i}= p^n\Omega_{\widehat{A}}^{i} \cap \Omega_{A_K^\dag}^{i}$).

\spa

We denote by $\overline{\im} \nabla_{M^\dag}^{i-1}$ the closure of $\im \nabla_{M^\dag}^{i-1}$ in $\Omega_{A^\dag_K}^{i}$ for the affinoid topology (beware that $\nabla_{M^\dag}^{i-1}$ is strict with closed image for the fringe topology (as in \cite{kedlaya06}) and not the affinoid topology, see \cite[Cor. 3.2]{GK04}.

\end{empt}
\begin{defi}
We set 
$$H^i_{\mr{rig}} (X_0,M^\dag)^{\sep}\coloneqq \Ker \nabla_{M^\dag}^{i}/\overline{\im} \nabla_{M^\dag}^{i-1}.$$

Note that if $B$ is another smooth  lift of $A_0$, then we get a (non-canonical) isomorphism
$A^\dag \riso B^\dag$. Hence, $H^i_{\mr{rig}} (X_0,M^\dag)^{\sep}$ is well-defined and is the maximal separated quotient of the $i$th rigid cohomological space of $M^\dag$. Similarly, 
$$N^i_{\rig} (X_0,M^\dag)\coloneqq \overline{\im} \nabla_{M^\dag}^{i-1}/\im \nabla_{M^\dag}^{i-1}$$
is well defined and it is the $p$-adic closure of $0$ ot the $i$th rigid cohomology group of $M^\dag$. They fit into the exact sequence
$$0 \to N^i_{\rig} (X_0,M^\dag)\to H^i_{\rig} (X_0,M^\dag)\to H^i_{\mr{rig}} (X_0,M^\dag)^{\sep} \to 0$$
\end{defi}

\begin{nota}\label{nota-MW}Let $(E^{\dag}, \nabla_{E^{\dag}})$ be an integral overconvergent flat connection over $X_0/\cV$. Taking the $p$-adic completion of $E^\dag$, we get a locally free  flat connection on $X_0/\cV$ of finite rank denoted by $({E}, \nabla_{{E}})$. We denote by $V^i_\mr{MW}(X_0,E^\dag)$ and $\bar{H}^i_\mr{MW}(X_0,E^\dag)$ the kernel and image of the comparison map $$H^i_\mr{MW}(X_0,E^\dag)\to H^i_\dR(\mathfrak{X},{E}).$$ 
    

\end{nota}

\begin{lemm}\label{lemmHirig} With notation as in \S \ref{nota-MW}, we have the equality
     $$N^i_{\mr{rig}}(X_0,E^\dag_K)=(H^i_\mr{MW}(X_0,E^\dag)/\mr{tors})_\mr{div}.$$
\end{lemm}
\begin{proof}
The affinoid topology on the abelian group $E^\dag_K \otimes_{A^\dag_K} \Omega_{A^\dag_K}^{i}$ has a basis of open neighborhood of $0$ given by the subgroups $p^nE^\dag \otimes_{A^\dag_K} \Omega_{A^\dag}^{i}$
for all $n$.
Hence, this follows from Lemma \ref{lemmHiconvpre}.
\end{proof}

\begin{lemm}\label{lem-rg-conv-ns}
Let $M^\dag$ be an $A^\dag_K$-module of finite type having an integral connection and a Frobenius structure.
Let $M\coloneqq M^\dag \otimes_{A^\dag_K}\widehat{A}_K$ be the associated convergent $F$-isocrystal on $X_0/K$. 
For any $i\in \bbN$ the following properties hold.
\begin{enumerate}
\item The topological $K$-vector space $\im \nabla_{M^\dag}^{i}$ is dense in  $\im \nabla_{M} ^{i}$  and the canonical map 
$$H^i_{\mr{rig}} (X_0,M^\dag)^{\sep} \to H^i_{\rm{conv}} (X_0,M)^{\sep}$$
is injective.
\item We have the equality $$\mr{Ker} ( H^i_{\rig} (X_0,M^\dag) \to H^i_{\conv} (X_0,M))
= \mr{Ker} ( N^i_{\rig} (X_0,M^\dag) \to N^i_{\conv} (X_0,M)).$$
\end{enumerate}

\end{lemm}

\begin{proof}
Using $\Omega_{A^\dag_K}^{i} \otimes_{A^\dag_K} M^\dag$ is dense in $\Omega_{\widehat{A}_K}^{i} \otimes_{\widehat{A}_K} M$, we get  
$\im \nabla_{M^\dag}^{i}$ is dense in  $\im \nabla_{M} ^{i}$, the equality $$\Ker \nabla_{M^\dag}^{i} \cap \overline{\im} \nabla_{M}^{i-1} = \overline{\im} \nabla_{M^\dag}^{i-1},$$ which yields the first assertion. We get the second one via the snake lemma. 
\end{proof}

\begin{rema}With notation as \S \ref{lem-rg-conv-ns}, beware that the inclusion $\overline{\im} \nabla_{M^\dag}^{i-1} \subseteq \im \nabla_{M}^{i-1} $ does not hold and
$N^i_{\rig} (X_0,M^\dag) \to N^i_{\conv} (X_0,M)$ is not injective in general.
For instance, if $X_0$ is an affine elliptic curve the containment $V^i_{\rig} (X_0/K) \subseteq N^i_{\rig} (X_0/K)$ is strict.
\end{rema}


\begin{lemm}\label{lem_lowdegreeMW_variant} 
For an integral overconvergent flat connection over $X_0/\calV$, the exact sequence $$0\to\bar{H}^i_\mr{MW}(X_0,E^\dagger)/\mr{tors}\to \left(H^i_\dR(\mathfrak{X},E)/\mr{tors}\right)\oplus\bar{H}^i_\mr{MW}(X_0,E^\dagger)_\Q \to H^i_\dR(\mathfrak{X},E)_\Q .$$
\end{lemm}
\begin{proof}
    This follows from Proposition \ref{prop_pullbacksquareMW}, as in Lemma \ref{third_es}.
\end{proof}

\begin{prop}\label{prop_relationns}
Let $(E^{\dag}, \nabla_{E^{\dag}})$ be an integral overconvergent flat connection over $X_0/\cV$. 
    We have the following cartesian square
    	\begin{equation}
		\label{diag_pullbackMW}
\begin{tikzcd}[row sep=large, column sep=large]
	N^i_{\mr{rig}}(X_0,E^{\dag}_K ) \arrow[r] \arrow[d] 
	& N^i_{\mr{conv}}(X_0,\widehat{E}_K) \arrow[d] \\
	H^i_{\mr{rig}}(X_0, E^{\dag}_K) \arrow[r] 
	& H^i_{\mr{conv}}(X_0,\widehat{E}_K).
\end{tikzcd}
	\end{equation}
\end{prop}

\begin{proof} 
To lighten the notation, we remove $(X_0,E^{\dag}_K)$,
$(X_0,E^{\dag})$, $(X_0,\widehat{E})$ or $(X_0,\widehat{E})$ in the notation. By Lemma \ref{lemmHiconv} and Lemma \ref{lemmHirig}, this is equivalent to check the exactness of the sequence
\begin{equation} \label{diag_pullbackMW-proof1}
0\to (H^i_\mr{MW}/\mr{tors})_\mr{div}\stackrel a \to (H^i_\dR/\mr{tors})_\mr{div} \oplus H^i_{\mr{rig}}\stackrel b \to H^i_{\conv}.
\end{equation}
Set $K^i_\dR\coloneqq \Ker \nabla_{\widehat{E}}^{i}$, $K^i_{\mr{MW}}\coloneqq \Ker \nabla_{E^\dag}^{i}$,
$K^i_{\conv} \coloneqq \Ker \nabla_{\widehat{E}_K}^{i}$, and $K^i_{\rig}\coloneqq \Ker \nabla_{E^\dag_K}^{i}$.
For $\star \in \{ \rm{MW}, \dR\}$, let $\varpi_{\star} \colon K^i_\star \to H^i_\star /\mr{tors}$ be the projections. 
For $\star \in \{ \rm{MW}, \dR,\rig,\conv\}$, let and  $[-]_{\star} \colon K^i_{\star} \to H^i_{\star}$ be the projections.

\spa

(i) Injectivity of $a$. Let $\iota \colon K^i_{\rm{MW}}  \to K^i_{\rig}$, $\iota \colon H^i_{\rm{MW}}  \to H^i_{\rig}$ be the canonical maps. 
Let $x\in K^i_{\rm{MW}}$. If $\iota ([x]_{\MW}) = 0$, then for some integer $N$ large enough, we get $p^N [x]_{\MW}=0$. This yields the injection $H^i_{\rm{MW}} /\tors \to H^i_{\rig}$, and we are done.

\spa

(ii) By construction of \eqref{diag_pullbackMW-proof1}, we have $b\circ a=0$. It remains to check $\mr{Ker} b \subseteq \im a$.
Let $x \in K^i_{\dR}$, $y\in K^i_{\rig}$ such that $\varpi_{\dR} (x) \in (H^i_\dR/\mr{tors})_\mr{div}$ and such that the images of $[x]_{\dR}$ and $[y]_{\rig}$
in $H^i_{\conv}$ agree. It follows from Proposition \ref{prop_pullbacksquareMW} that there exists $z \in K^i_{MW}$ such that the image of $[z]_{\MW}$ 
via $H^i_\mr{MW} \to H^i_\dR \oplus H^i_{\mr{rig}}$ is $([x]_{\dR},[y]_{\rig})$.
Since $\varpi_{\dR} (x) \in (H^i_\dR/\mr{tors})_\mr{div}$, then for any $n\in \bbN$, there exists $x_n \in K^i_\dR$
such that $\varpi_{\dR} (x) = p^n \varpi_{\dR}( x_n)$. Since $[x_n]_{\dR}$ and $[p^{-n} y]_{\rig}$ have the same image in $H^i_{\conv}$, then there exists 
$z_n \in K^i_{\MW}$ such that $[z_n]_{\MW}$ induces $[x_n]_{\dR}$ and $[p^{-n} y]_{\rig}$.
Using the injectivity of the map $a$, this shows that $\varpi_{\MW} (z) = \varpi_{\MW} ( p^n z_n)= p^n \varpi_{\MW} ( z_n)$. Hence, 
$$\varpi_{\MW} (z) \in \bigcap_{n\in \bbN} p^n  (H^i_\mr{MW}/\mr{tors}) =   (H^i_\mr{MW}/\mr{tors})_\mr{div},$$ 
which concludes the proof.
%
\end{proof}

\subsection{Vanishing in the separated quotient}


In this section, we prove that a certain subspace of convergent cohomology is always contained in $N^i_{\mr{conv}}.$ We give two proofs. In the second one, we use an extended version of syntomic cohomology (Definition \ref{dfn-TijE}).

\begin{theo}\label{thm_slopeiminus1-variant} If $X_0$ is a smooth affine scheme over $k$ and  $\cE$ is a convergent $F$-isocrystal over $X_0/W$ of pure slope $\gamma$,
then $$H^i_\mr{conv}(X_0,\calE)^{[\gamma ,\gamma +i)}\subseteq N^i_{\mr{conv}}(X_0,\calE)$$
for $i\geq 1$. 

\end{theo}
\begin{proof}
As in the proof of Proposition \ref{Prop_RationalBounds}, since 
$H^i_\mr{conv}(X_0,\calE(r,s))^{\sep}= H^i_\mr{conv}(X_0,\calE)^{\sep} (r,s)$, then 
we reduce to the case $\gamma= 0$.
It is sufficient to show that for $r\in \bbN$ and $s \in \bbZ_{>0}$ such that $r/s<i$, then the map 
$$F^s -p^r\colon H^i_\mr{conv}(X_0,\calE)^{\sep}  \to H^i_\mr{conv}(X_0,\calE)^{\sep} $$ is an isomorphism. 
Let $E$ be a unit-root $F$-crystal such that $E_K=\calE$. 
Since $\gamma <i$, then for any $j \geq i$, the map $F^s/p^r\colon \mr{dR}^{j}_{E}\to \mr{dR}^{j}_{E}$ is well-defined and its image is included in $p\mr{dR}^{j}_{E}$.    

\spa

Since $\mr{dR}^{j}_{E}$ is complete for the $p$-adic topology, for any $j \geq i$ the map $1-F^s/p^r\colon \mr{dR}^{j}_{E}\to \mr{dR}^{j}_{E}$ is an isomorphism. Tensoring with $K$, this shows that the map $1-F^s/p^r\colon \mr{dR}^{j}_{\cE}\to \mr{dR}^{j}_{\cE}$ is an isomorphism.

\spa

Thanks to the fact that $\nabla^{i-1}_{\calE}$ commutes with $1-F^s/p^r$, then 
$1-F^s/p^r$ sends $\im \nabla^{i-1}_{\calE}$ into   $\im \nabla^{i-1}_{\calE}$.
Since $1-F^s/p^r$ is continuous, then 
$1-F^s/p^r$ sends 
$\overline{\im} \nabla^{i-1}_{\calE}$ to $\overline{\im} \nabla^{i-1}_{\calE}$.
Moreover, we have $$(1-F^s/p^r)^{-1}=\sum_{n\in \bbN} (F^s/p^r)^n .$$ Using the fact that $\sum_{n=0}^N (F^s/p^r)^n$
preserves $\im \nabla^{i-1}_{\calE}$, then $(1-F^s/p^r)^{-1}$ sends 
$\im \nabla^{i-1}_{\calE}$ (and therefore by continuity $\overline{\im} \nabla^{i-1}_{\calE}$) into $\overline{\im} \nabla^{i-1}_{\calE}$. Hence, we are done.
\end{proof}
Thanks to Theorem \ref{thm_slopeiminus1-variant}, we shed new light on the examples in \cite{ES20}.
\begin{coro}\label{cor_ES20}
    Let  $X_0$ be a smooth affine $k$-scheme and let $\calE^\dagger$ be a unit-root overconvergent $F$-isocrystal over $X_0/K$. If $E^\dagger$ is an integral overconvergent flat connection such that $E_K^\dagger=\calE^\dagger$, there exists a natural $W$-linear injective morphism $$H^i_\mr{rig}(X_0,E^{\dag}_K)^{[0,i)}\hookrightarrow H^i_{\mr{MW}}(X_0,E^{\dag})/\mr{tors}$$
   for every $i\geq 1$. In addition, the rationalised cokernel is isomorphic to $H^i_\mr{rig}(X_0,E^{\dag}_K)^{[i]}$. \end{coro}
    \begin{proof}By Theorem \ref{thm_slopeiminus1-variant}, we know that $H^i_\mr{rig}(X_0,E^\dag_K)^{[0,i)}$ is sent to $N^i_\mr{conv}(X_0, \widehat{E}_K).$ By virtue of Proposition  \ref{prop_relationns}, we deduce the inclusion $$H^i_\mr{rig}(X_0,E^{\dag}_K)^{[0,i)}\subseteq N^i_\mr{rig}(X_0,E^{\dag}_K)\overset{\ref{lemmHirig}}{=}(H^i_{\mr{MW}}(X_0,E^{\dag})/\mr{tors})_{\mr{div}}.$$ This gives the desired result. For the last part, we use the slope bounds of Corollary \ref{cor_RationalBoundsvar}.
    \end{proof}

\begin{empt}Let us also explain a different proof of Theorem \ref{thm_slopeiminus1-variant} using \textit{extended syntomic cohomology}. Write $\fX=\Spf(\widehat{A})$ for a formal lift of $X_0$ over $W$ and write $F$ for a Frobenius lift $F\colon \fX\to \fX.$ Let $E$ be a unit-root $F$-crystal over $X_0/W$.  We keep the notation as in \S \ref{sec_FGauge}.
\end{empt}

\begin{defi}\label{dfn-TijE}
    We define \textit{extended syntomic cohomology} of $E(r)$ as the complex $$R\Gamma_\mr{esyn}(X_0,E(r))\coloneqq \Cocone\left( \mr{dR}^{\bullet,r}_E\xrightarrow{\varphi_r-\iota} \mr{dR}^{\bullet}_{E_K} \right).$$
We also write $$T_E^\bullet\coloneqq \Cocone\left( \mr{dR}^{\bullet}_E\to \mr{dR}^{\bullet}_{E_K} \right)\quad\text{and}\quad \TE^{i,j}\coloneqq H^i(\mr{dR}^{\bullet,j}_E)[p^\infty].$$ The group $\TE^{i,0}$ will also be denoted by $\TE^i$.  
\end{defi}
\begin{empt}
Note that for every $i$ we have 
\begin{equation}
    \label{exseq-Ci}0\to H^{i-1}_\cris(E)\otimes_{\Zp} \Qp/\Zp\to H^i(T^\bullet_E)\to \TE^i\to 0,
\end{equation}
where $\TE^i\coloneqq H^i_\cris(E)[p^\infty]$. In addition, there exists by the octahedron axiom a distinguished triangle
\begin{equation}\label{syn->esyn-tri}
    R\Gamma_\mr{syn}(X_0,E(r))\to R\Gamma_\mr{esyn}(X_0,E(r)) \to T^\bullet_E.
\end{equation}\end{empt}

\subsubsection{A different proof of Theorem \ref{thm_slopeiminus1-variant}.} Let us explain another proof of Theorem \ref{thm_slopeiminus1-variant}, at least for integral slopes. As before, we assume $k$ algebraically closed.

\begin{lemm}\label{lemmaHr+1phir-1=0}We have
$$H^{r+1}_\dR(X_n,E_n)^{\varphi_r-1}=0.$$
\end{lemm}

\begin{proof}
    We proceed by induction on $n$. Since $F/p^r \colon E_1\otimes \Omega_{X_1}^{r+1}\to E_1\otimes \Omega_{X_1}^{r+1}$ is trivial, then we get the case $n=1$.
    Thanks to the fact that  $$\mr{Im}(H^{r+1}_\dR(X_1,E_1)\xrightarrow{\underline{p}^{n-1}} H^{r+1}_\dR(X_n,E_n))$$
is killed by $\varphi_r$, it has trivial intersection with $H^{r+1}_\dR(X_n,E_n)^{\varphi_r-1}$. The vanishing then follows from the one of $H^1_\dR(X_{n-1},E_{n-1})^{\varphi_r-1}$
and the exact sequence
$$H^{r+1}_\dR(X_1,E_1)\xrightarrow{\underline{p}^{n-1}} H^{r+1}_\dR(X_n,E_n) \to H^{r+1}_\dR(X_{n-1},E_{n-1}).$$ 
\end{proof}

    Let $E$ be a unit-root $F$-crystal such that $E_K=\calE$. 

    \spa
(i) For $r= i-1$, let us  prove that $H^{r+1}_\syn(X_0,E(r))$ is sent to $p^nH^{r+1}_{\mr{esyn}}(X_0,E(r))$ for every $n\geq 0$.
In other words, we want to prove that the map 
$$H^{r+1}_\syn(X_0,E(r))/p^n\to H^{r+1}_\mr{esyn}(X_0,E(r))/p^n$$ is the zero map.
Using the change-of-coefficients exact sequence for  $R\Gamma_\mr{syn}(X_0,E(r))$ and $R\Gamma_\mr{esyn}(X_0,E(r))$ from $\Zp$ to $\Z/p^n\Z,$ we are reduced to prove that the left vertical arrow of the following diagram with exact rows vanishes
\begin{equation}\label{eq_nonsarated3} \xymatrix {&{H^{r+1} ( R\Gamma_\mr{syn}(X_0,E(r)) \otimes^L \Z/p^n\Z) } \ar[d]^-{} \ar[r] & {H^{r+1}(\mr{dR}^{\bullet,r}_{E}\otimes\Z/p^n\Z)} \ar@{=}[d] \ar[r]^-{\varphi_r-\iota} & {H^{r+1}(\mr{dR}^{\bullet}_{E}  \otimes\Z/p^n\Z)} \ar[d]
\\ {0}\ar[r]& {H^{r+1} ( R\Gamma_\mr{esyn}(X_0,E(r)) \otimes^L \Z/p^n\Z) } \ar[r] & {H^{r+1}(\mr{dR}^{\bullet,r}_{E}\otimes\Z/p^n\Z)}  \ar[r]^-{} &  {0}}.
\end{equation}
Hence, we are reduced to check that
$$H^{r+1}(\mr{dR}^{\bullet}_E\otimes\Z/p^n\Z)=H^{r+1}(\mr{dR}^{\bullet,r}_E\otimes \Z/p^n\Z)\xrightarrow{\varphi_r-1}H^{r+1}(\mr{dR}^{\bullet}_E\otimes \Z/p^n\Z)$$
is injective, which is proved in  Lemma \ref{lemmaHr+1phir-1=0}.
\spa


(ii) Consider the $\Zp$-module 
$$M\coloneqq \im \left( H^{r+1}_{\mr{esyn}}(X_0,E(r))\to H^{r+1}(\mr{dR}^{\bullet,r}_{E})\right).$$

Since $H^{r+1}_\syn(X_0,E(r))_\bbQ =H^{r+1}_{\mr{esyn}}(X_0,E(r))_\bbQ $ and  $H^{r+1}_\syn(X_0,E(r))$ is sent to $p^nH^{r+1}_{\mr{esyn}}(X_0,E(r))$ for every $n\geq 0$ (Step (i)),  then $M/\mr{tors}=M_\Q$.
Thanks to the fact that $M_\Q$ is $p$-divisible, this yields the inclusion $M_\bbQ \hookrightarrow  (H^{r+1}(\mr{dR}^{\bullet}_E)/\mr{tors})_\mr{div} $. We have that $$M= \mr{Ker} ( H^{r+1}(\mr{dR}^{\bullet,r}_{E}) \stackrel {\varphi_r-\iota} \longrightarrow H^{r+1}(\mr{dR}^{\bullet}_{E_K})),$$ which implies that $M_K=  H^{r+1}(\mr{dR}^{\bullet}_{E_K})^{[r]}$. Hence, we conclude thanks to Lemma \ref{lemmHiconv}.

\qed

\subsection{A separated subspace}
In contrast to the previous results, this section establishes positive results on injectivity (Theorems \ref{thm_InjectivitySlopei} and \ref{thm_InjectivitySlopeibis}). We first prove that the slope $r$ part of the $r$th rigid cohomology group is equal to the slope $r$ part of the $r$th convergent cohomology. The proof involves a reduction to the case of good compactification via (a refinement of) de Jong's alterations. Next, when $\gamma \in (q-1,q)$, we establish that the slope $\gamma$ part of the $q$th rigid cohomology group injects into convergent cohomology. We also check that this space is separated (Theorem \ref{thm_sloperseparated}).

\spa

Let $X_0$ be a smooth scheme over $k$ and suppose $\calV=W$ unramified.

\begin{theo}\label{thm_InjectivitySlopei} The map  $$H^r_\mr{rig}(X_0/K)^{[r]}\to H^r_\mr{conv}(X_0/K)^{[r]}$$ is bijective for every $r\geq 0$.
\end{theo}
\begin{proof}
Note that in the statement we may assume $k$ algebraically closed by faithfully flat descent.

\spa



(a) We reduce to the case where there exist a smooth projective $k$-scheme $Y_0$  and a strict normal crossing divisor $D_0$ such that $X_0 = Y_0 \setminus D_0$ is affine. Thanks to \cite[Thm. 1.2]{BS}, there exists an étale hypercovering $X_{\bullet ,0} '\to X_0$ 
such that $X '_{i,0}$ admits a smooth projective compactification $Y_{i,0} '$ so that $Y_{i,0} '\setminus X_{i,0} '$ is a strict normal crossing divisor. 
\begin{itemize}
    \item [{(i)}] By \cite[Rmk. 11.7.2]{CT03}), we have the spectral sequence 
$$E_1^{i,j}\coloneqq H^j_\mr{rig}(X'_{i,0} /K)\Rightarrow H^{i+j}_\mr{rig}(X_0/K).$$
By     \cite[Thm. 5.4.1]{KedWeil2}, we know that the rigid cohomology of $H^j_\mr{rig}(X'_{i, 0}/K)$ has slopes $\leq j$. Since the spectral sequence commutes with the action of Frobenius, thanks to Dieudonné--Manin classification, 
this implies that $(E_{\ell}^{i,j})^{[r]} =0$ if $j <r$ and $\ell\geq 1$. In particular, we deduce that $(E_{\infty}^{0,r})^{[r]} =(E_{2}^{0,r})^{[r]}$. 
Hence, we get the exact sequence: 
$$0 \to H^r_\mr{rig}(X_0/K)^{F -p^r} \to H^r_\mr{rig}(X'_{0,0}/K)^{F -p^r}
\to H^r_\mr{rig}(X'_{1,0}/K)^{F -p^r} .$$

\item[{(ii)}]  Since (partial) rigid cohomology satisfies cohomological étale descent (in the context of \cite[Thm. 7.1.1]{Tsu04}),  then we get the spectral sequence
$$\bbE_1^{i,j}\coloneqq H^j_\mr{conv}(X'_{i,0} /K)\Rightarrow H^{i+j}_\mr{conv}(X_0/K),$$
which commutes with the action of Frobenius.  For $i,j,\ell\geq 0$, we consider the condition  $$\calP(i,j,\ell)\coloneqq\begin{cases}
j\leq r-1, & \text{if $i+j-r\leq 0$}\\
j\leq r-\ell(i+j-r), & \text{if $i+j-r\geq 1$,}
\end{cases}$$ and we use the convention that $\calP(i,j,\ell)$ is true when $i$ or $j$ are negative. We want to show by induction on $\ell\geq 1$ that $$F -p^r \colon \bbE_{\ell}^{i,j} \to \bbE_{\ell}^{i,j}$$ is an isomorphism if $\calP(i,j,\ell)$ holds. When $\ell=1$, by Corollary \ref{cor_RationalBoundsvar}, the morphism $F -p^r \colon \bbE_{1}^{i,j} \to \bbE_{1}^{i,j}$ is an isomorphism for every $j\leq r-1$. For higher $\ell$, we use Lemma \ref{lemkercokerF}. Looking at 
$$\bbE_{\ell}^{i-\ell,j+\ell-1}\to\bbE_{\ell}^{i,j}\to \bbE_{\ell}^{i+\ell,j-\ell+1},$$
this reduces to prove that $$\calP(i,j,\ell+1)\Rightarrow \calP(i+\ell,j-\ell+1,\ell)\wedge \calP(i-\ell,j+\ell-1,\ell).$$

\noindent We show that
$$\calP\coloneq\calP(i,j,\ell+1)\Rightarrow \calP'\coloneqq\calP(i+\ell,j-\ell+1,\ell);$$ the other implication is left to the reader. Setting $k\coloneqq i+j-r$, we can rewrite

$$\calP'=\begin{cases}
j-\ell+1\leq r-1, & \text{if $k\leq -1$}\\
j-\ell+1\leq r-\ell(k+1), & \text{if $k\geq 0$.}
\end{cases}$$

\noindent We have three cases: if $k\leq -1$, then $\calP$ implies $j\leq r-1,$ which shows that $j-\ell+1\leq r-1$. If $k=0,$ we deduce again $j\leq r-1$, which coincides with $j-\ell +1\leq r-\ell(k+1).$ Finally, if $k\geq 1,$ we get from $\calP$ that $j\leq r-(\ell+1)k.$ Thus, we are left to check that $$(r-(\ell+1)k)-\ell+1\leq r-\ell(k+1),$$ which is true when $k\geq 1.$

\spa

\item[{(iii)}] Let $\bbE' $ be the kernel of the projection $H^{r}_\mr{conv}(X_0/K) \twoheadrightarrow \bbE_{\infty}^{0,r}$. 
By the previous step, $F -p^r \colon \bbE_{\infty}^{i,r-i} \to \bbE_{\infty}^{i,r-i}$ is an isomorphism for $i >0$. By the five lemma, we deduce that the map $F -p^r \colon \bbE ' \to \bbE '$ is an isomorphism.
Hence, the map $$(H^{r}_\mr{conv}(X_0/K))^{F -p^r}\to (\bbE_{\infty}^{0,r})^{F -p^r}$$ is an isomorphism.
For every $m \geq 2$, we have $(\bbE_{m+1}^{0,r})^{F -p^r} =(\bbE_{m}^{0,r})^{F -p^r}$ and therefore
$(\bbE_{\infty}^{0,r})^{F -p^r} =(\bbE_{2}^{0,r})^{F -p^r}$. 
By definition, $$\bbE_{2}^{0,r}= \Ker (H^r_\mr{conv}(X'_{0,0}/K)\to H^r_\mr{conv}(X'_{1,0}/K)),$$ 
thus $$(\bbE_{2}^{0,r})^{F -p^r} = \Ker (H^r_\mr{conv}(X'_{0,0}/K)^{F -p^r}\to H^r_\mr{conv}(X'_{1,0}/K)^{F -p^r}).$$
We then get the exact sequence:
$$0 \to H^r_\mr{conv}(X_0/K)^{F -p^r} \to H^r_\mr{conv}(X'_{0,0}/K)^{F -p^r}\to H^r_\mr{conv}(X'_{1,0}/K)^{F -p^r} .$$
\end{itemize}

\spa

(b)
We have the isomorphisms
\begin{align*}
R\Gamma_\mr{pro\et}(Y_0,W\Omega^r_{Y_0^\sharp,\mr{log}}[-r] )_{\Qp}
&\underset{\ref{prop_BMS8.7prelog}}{\riso}  \Cocone \left (R \Gamma (Y_0, W\Omega^\bullet_{Y^{\sharp}_{0}} )_K\xrightarrow{F /p^r-1}R \Gamma (Y_0, W\Omega^\bullet_{Y^{\sharp}_{0}} )_K \right )\\
&\underset{\eqref{log-cris-rig=}}{\riso}  \Cocone \left (R \Gamma_{\rig} (X_0/K)\xrightarrow{F /p^r-1}R \Gamma_{\rig} (X_0/K)  \right ).
\end{align*}
Thanks to
$$R\Gamma_\mr{pro\et}(X_0,W\Omega^r_{X_0,\mr{log}}[-r] )_{\Qp}
\underset{\ref{BMS8.7preUR}}{\riso}  
  \Cocone \left (R \Gamma_{\conv} (X_0/K)\xrightarrow{F /p^r-1}R \Gamma_{\conv} (X_0/K)  \right ),$$
this yields  the commutative diagram
\begin{equation}
\label{thm_InjectivitySlopei-diag}\xymatrix{{H^0_\mr{pro\et}(Y_0,W\Omega^r_{Y_0^\sharp,\mr{log}} )_K} \ar@{->>}[r]^-{} \ar[d]^-{}& { H^r_\mr{rig}(X_0/K)^{^{F-p^r}}} \ar[d]^-{}\\ {H^0_\mr{pro\et}(X_0,W\Omega^r_{X_0,\mr{log}})_K} \ar@{->>}[r]^-{}& { H^r_\mr{conv}(X_0/K)^{F-p^r}} }
\end{equation}
with surjective horizontal arrows. The bottom map of \eqref{thm_InjectivitySlopei-diag} is in fact an isomorphism:  $$
H^0_\mr{pro\et}(X_0,W\Omega^r_{X_0,\mr{log}})_K
\overset{\eqref{BMS8.21}}{=} H^{r}_{\mr{syn}}(X_0,\Qp(r))_K\overset{\ref{Prop_IntegralBounds}}{=}H^{r}_{\mr{conv}}(X_0/K)^{F-p^r}.$$  
Thanks to Lemma \ref{lem_compatibilitylogOmega}, its left vertical arrow is also an isomorphism. Hence we are done.

\end{proof}

\begin{theo}\label{thm_InjectivitySlopeibis} For $q\geq 0$, the morphism  $$H^q_\mr{rig}(X_0/K)^{(q-1,q)}\to H^q_\mr{conv}(X_0/K)$$ is injective.
\end{theo}

\begin{proof}
Again, to prove the statement we may assume $k$ algebraically closed by faithfully flat descent.
Let $\gamma \in (q-1,q)$. Arguing as in Step (a) of the proof of Theorem \ref{thm_InjectivitySlopei}, we reduce to the case where there exists $Y_0$ a smooth projective $k$-scheme and $D_0$ is a strict normal crossing divisor such that $X_0 = Y_0 \setminus D_0$.

\spa
Let $Y_{\bullet,0}\to Y_0$ be an étale hypercovering and $X_{i,0} \coloneqq Y_{i,0} \times_{Y_0} X_0$ for any $i$. Consider the diagram:
\begin{equation}
\label{eq_InjectivitySlopei-es2} 
\xymatrix{{0} \ar[r]^-{}  & {H^r_\mr{rig}(X_0/K)^{F^s-p^r}  }\ar[r]^-{} \ar[d]^-{} & {H^r_\mr{rig}((X_{0,0}, Y_{0,0})/K)^{F^s-p^r}  } \ar[r]^-{} \ar[d]^-{} & {H^r_\mr{rig}((X_{1,0}, Y_{1,0})/K)^{F^s-p^r}  } \ar[d]^-{}\\
{0} \ar[r]^-{} & {H^r_\mr{conv}(X_0/K)^{F^s-p^r}  }\ar[r]^-{} & {H^r_\mr{conv}(X_{0,0}/K)^{F^s-p^r} } \ar[r]^-{}& {H^r_\mr{conv}(X_{1,0}/K)^{F^s-p^r}  }. }
\end{equation}
Similarly to Step (a) of the proof of Theorem \ref{thm_InjectivitySlopei}, using étale cohomological descent (in the context of \cite[7.1.1]{Tsu04}), it follows from Corollary \ref{cor_RationalBoundsvar} that the rows are exact sequences. Hence, we reduce to prove that the middle vertical arrow is injective. 
As for the proof of Proposition \ref{prop_cohomologysheaves_var}, étale locally on $Y_0$ we are in the context of Corollary \ref{cor_InjectivityRetraction}.  
Using 
\cite[Thm. 1.3.6]{CT12} and 
Corollary \ref{cor_InjectivityRetraction} and proceeding by induction on the number of irreducible components of $D_0$, we are done.
\end{proof}
\begin{exam}Note that Theorem \ref{thm_InjectivitySlopei} is false for the subspaces of slope $\gamma\in (r-1,r).$ Indeed, if we take $X_0=\mathbb{A}^1_{\Fp}$ and we consider $\fX\coloneqq\mr{Spf}(\Zp\langle t\rangle)$ as a lift, then $$\omega\coloneqq\sum_{i=0}^\infty p^it^{p^{2i}} \mr{dlog}(t)$$ defines a non-torsion class in $H^1_\mr{dR}(\fX/\Zp).$ It satisfies the identity $$F^2(\omega)=F^2\left(\sum_{i=0}^\infty p^it^{p^{2i}} \mr{dlog}(t)\right)=\sum_{i=0}^\infty p^{i+2}t^{p^{2(i+1)}} \mr{dlog}(t)=p\omega-pdt.$$

    Therefore $F^2([\omega])=[\omega]$, which inplies that $[\omega]$ lies in $H^1_\mr{conv}(\mathbb{A}^1_{\Fp}/\Qp)^{[1/2]}$. On the other hand, $$H^1_\mr{rig}(\mathbb{A}^1_{\Fp}/\Qp)^{[1/2]}\subseteq H^1_\mr{rig}(\mathbb{A}^1_{\Fp}/\Qp)=0.$$
\end{exam}


\begin{prop}
 If $X_0=Y_0\setminus D_0$ with $X_0$ affine,  $Y_0$ is smooth and proper, and $D_0$ a strict normal crossing divisor, then $$H^{r+1}_\mr{syn}(Y_0^\sharp,\Zp(r))\to H^{r+1}_\mr{syn}(X_0,\Zp(r))$$ is injective for $r\geq 0$, where $Y_0^\sharp$ is the log scheme associated to $D_0$ (notation as in \S \ref{synE(r,s)} and \S \ref{Def_logsyn}).
\end{prop}
\begin{proof}
By Lemma \ref{lem_compatibilitylogOmega}, we have $j_*W\Omega^r_{X_0,\mr{log}}=W\Omega^r_{Y_0^\sharp,\mr{log}}.$ Thus, looking at the low degree exact sequence of Leray spectral sequence for $j\colon X_0\hookrightarrow Y_0$, we deduce that $$H^1_\proet(Y_0,j_*W\Omega^r_{X_0,\mr{log}})\to H^1_\proet(X_0,W\Omega^r_{X_0,\mr{log}})$$ is injective, as we wanted. 
\end{proof}

\begin{lemm}\label{lem_surjectivetorsion}
Let $E$ be a unit-root $F$-crystal on $X_0 /W$. 
The morphism $\TE^{r,r}\xrightarrow{\varphi_r-\iota} \TE^r$ is surjective (see Notation \ref{dfn-TijE}).
\end{lemm}
\begin{proof}
   Thanks to Proposition \ref{thm_Mazur}, there exists a morphism in $D(\Zp)$ defined by
    $$\psi_r\colon \tau_{\leq r}\mr{dR}^{\bullet}_E\xrightarrow{\cdot p^r} \tau_{\leq r} p^r\mr{dR}^{\bullet}_E\xleftarrow{F}\tau_{\leq r} \mr{dR}^{\bullet,r}_E.$$ This yields for every $n\geq 1$ a morphism 
    $$\psi_r\colon H^{r-1}(\mr{dR}^{\bullet}_E\otimes_{\Zp}\Zp/p^n\Zp)\to H^{r-1}(\mr{dR}^{\bullet,r}_E\otimes_{\Zp}\Zp/p^n\Zp),$$ which is the inverse of $$\varphi_r\colon H^{r-1}(\mr{dR}^{\bullet,r}_E\otimes_{\Zp}\Zp/p^n\Zp)\to  H^{r-1}(\mr{dR}^{\bullet}_E\otimes_{\Zp}\Zp/p^n\Zp).$$
 We want to prove that $\iota\circ \psi_r$ is a nilpotent $\Zp$-linear endomorphism of $H^{r-1}(\mr{dR}^{\bullet}_E\otimes_{\Zp}\Zp/p^n\Zp)$. For this purpose, we show by induction that $(\iota\circ \psi_r)^i$ factors through a certain morphism $$\widetilde{\psi}_r^i\colon H^{r-1}(\mr{dR}^{\bullet}_E\otimes_{\Zp}\Zp/p^n\Zp)\to H^{r-1}(\mr{dR}^{\bullet}_E\otimes_{\Zp}p^{i-1}\Zp/p^n\Zp) $$ for $i\geq 1$. For $i=1$ we just take $\widetilde{\psi}^1_r\coloneqq\psi_r.$ For $i\geq 2$ we note that, by definition, the morphism

    $$\iota\colon H^{r-1}(\mr{dR}^{\bullet,r}_E\otimes_{\Zp}p^m\Zp/p^n\Zp)\to H^{r-1}(\mr{dR}^{\bullet}_E\otimes_{\Zp}p^{m}\Zp/p^n\Zp)$$ factors through
    $$\epsilon \colon H^{r-1}(\mr{dR}^{\bullet,r}_E\otimes_{\Zp}p^m\Zp/p^n\Zp)\to H^{r-1}(\mr{dR}^{\bullet}_E\otimes_{\Zp}p^{m+1}\Zp/p^n\Zp)$$ for $m\leq n.$ If we set $\widetilde{\psi}^i_r\coloneqq \epsilon \circ\psi_r\circ  \widetilde{\psi}^{i-1}_r,$ we obtain a well-defined morphism $$\widetilde{\psi}_r^i\colon H^{r-1}(\mr{dR}^{\bullet}_E\otimes_{\Zp}\Zp/p^n\Zp)\to H^{r-1}(\mr{dR}^{\bullet}_E\otimes_{\Zp}p^{i-1}\Zp/p^n\Zp) $$ satisfying the desired property.

\spa
    
  The sum $$\sum_{i=1}^{n} (\iota \circ \psi_r)^i\colon H^{r-1}(\mr{dR}^{\bullet}_E\otimes_{\Zp}\Zp/p^n\Zp)\to H^{r-1}(\mr{dR}^{\bullet,r}_E\otimes_{\Zp}\Zp/p^n\Zp)$$ defines a right inverse of $$H^{r-1}(\mr{dR}^{\bullet,r}_E\otimes_{\Zp}\Zp/p^n\Zp)\xrightarrow{\varphi_r-\iota} H^{r-1}(\mr{dR}^{\bullet}_E\otimes_{\Zp}\Zp/p^n\Zp).$$ By the change-of-coefficients exact sequence, we deduce that the two groups surject respectively into $\TE^{r,r}$ and $\TE^r$. This ends the proof.
\end{proof}


\begin{theo}\label{thm_sloperseparated}If $X_0$ is a smooth affine scheme over $k$, and $\calE$ is a unit-root $F$-isocrystal over $X_0/K$, then $H^r_\mr{conv}(X_0,\calE)^{[r]}$ intersects trivially $N^r_\mr{conv}(X_0,\calE).$
\end{theo}
\begin{proof}
Let $E$ be a unit-root $F$-crystal such that $E_K=\calE$. 
%
Denote by $$\alpha\colon H^{r}(\mr{dR}^{\bullet,r}_{E}) \overset{\varphi_r-\iota}{\longrightarrow} H^{r}(\mr{dR}^{\bullet}_{E})$$ 
and $$\beta\colon H^{r}(\mr{dR}^{\bullet,r}_{E}) \overset{\varphi_r-\iota}{\longrightarrow} H^{r}(\mr{dR}^{\bullet}_{E_K}).$$
Set $N\coloneqq \mr{Ker} (\alpha)$ and $M\coloneqq \mr{Ker}(\beta)$.
Since $M=\mr{Ker} (\beta)$, we have $\alpha (M) \subseteq  \TE^r$.
According to Lemma \ref{lem_surjectivetorsion}, we have $\alpha  (\TE^{r,r}) =  \TE^r$.
Since $M[p^\infty]$ contains $\TE^{r,r}$, then the inclusions $$\alpha (M[p^\infty]) \subseteq \alpha (M)\subseteq  \TE^r$$ are equalities.
This implies that the composition $N\subseteq M \twoheadrightarrow M/\mr{tors}$ is surjective because $N = \mr{Ker}(\alpha)$.
Thanks to the fact that $$N = \im (H^{r}_\syn(X_0,E(r))\to H^{r}(\mr{dR}^{\bullet,r}_{E}))$$ and 
the group $H^{r}_\syn(X_0,E(r))$ is finitely generated by Proposition \ref{prop_Bosco}, this implies that $M/\mr{tors}$ is a finite type $\Zp$-module. This gives $(M/\mr{tors})_\mr{div}=0.$ 
Observe that $$M/\tors = H^r_\mr{conv}(X_0,\calE)^{F-p^r}\cap (H^{r}(\mr{dR}^{\bullet,r}_{E})/\tors),$$ which implies that
$$H^r_\mr{conv}(X_0,\calE)^{F-p^r}\cap (H^{r}(\mr{dR}^{\bullet,r}_{E,r})/\tors)_{\rm{div}}= (M/\tors)_{\rm{div}}=0.$$ 
We conclude thanks to the canonical isomorphism $$(H^{r}(\mr{dR}^{\bullet,r}_{E})/\tors)_{\rm{div}}= N^r_\mr{conv}(X_0,\calE)$$ of Lemma \ref{lemmHiconv}.
\end{proof}
\begin{coro}\label{cor_RigidSlopeRSeparated}
    If $X_0$ is a smooth affine scheme over $X_0$, then $H^r_\mr{rig}(X_0/K)^{[r]}$ intersects trivially $$N^r_\mr{rig}(X_0/K)=(H^i_\mr{MW}(X_0/W)/\mr{tors})_\mr{div}.$$ Equivalently, $$H^r_\mr{rig}(X_0/K)^{[r]}\cap H^i_\mr{MW}(X_0/W)/\mr{tors}$$ is a free $W$-module of finite rank.
\end{coro}
\begin{proof}
    Thanks to Theorem \ref{thm_InjectivitySlopei}, we have that $H^r_\mr{rig}(X_0/K)^{[r]}$ embeds into  $H^r_\mr{conv}(X_0/K)^{[r]}$. The result then follows from Theorem \ref{thm_sloperseparated} thanks to Proposition \ref{prop_relationns}.
\end{proof}

\begin{theo}\label{thm_DeterminationAlgebraicdR}
     Let $Y$ be a smooth proper scheme over $W$, let $D$ be a relative normal crossing divisor, and let $X\subseteq Y$ be the complement of $D$. If $X$ is affine, then
     $$H^i_\dR(X/W)/\mr{tors}\simeq K^{\oplus a}\oplus W^{\oplus b}$$ where $a$ is the dimension of $H^i_\mr{rig}(X_0/K)^{[0,i)}$ and $b$ is the dimension of $H^i_\mr{rig}(X_0/K)^{[i]}.$
\end{theo}
\begin{proof}
    This follows from Corollary \ref{cor_ES20} and Corollary \ref{cor_RigidSlopeRSeparated} thanks to the comparison with integral Monsky--Washnitzer cohomology of Theorem \ref{thm_comparisonMWdR}.
\end{proof}

\section{On the injectivity for arithmetic $D$-modules}
\label{Sec_InjectivityArithmeticDmodules}
\subsection{Arithmetic $D$-modules and a counterexample}\label{Sec_FailureExtInjectivity}
The purpose of the subsection is to give geometric counterexamples of the noncommutative version of the following proposition.
\begin{prop}\label{prop-ffExt}
Let $A \to B$ be a left faithfully flat morphism of coherent rings. 
Let $M$ be a coherent left $A$-module and $N$ be a left $A$-module. 
Suppose either $A$ and $B$ are commutative or $N$ is flat. Then the canonical map 
	$$\operatorname{Ext}^i_A(M, N) \hookrightarrow \operatorname{Ext}^i_B(B\otimes_A M, B \otimes_A N)$$ is injective for every $i\geq 0$.
\end{prop}

\begin{proof}
The commutative case is well-known. Consider the case when $N$ is flat. By taking a coherent projective resolution of $M$, we deduce that the flat ($A$-bilinear) morphism $A\to B$ induces maps
$$\Ext^i_A(M, A)\to \Ext^i_A(M, B)\riso \Ext^i_A(M, A)\otimes_A B,$$ where the second one is an isomorphism. In addition, since $A\to B$ is fully faithful, we deduce that the first map is injective. Thanks to the fact that $N$ is a flat $A$-module and $M$ is coherent, then the map
$$\Ext^i_A(M, N) = \Ext^i_A(M, A)\otimes_A N \to \Ext^i_A(M, B)\otimes_A N  =\operatorname{Ext}^i_B(B\otimes_A M, B \otimes_A N)$$ is injective, as we wanted.
\end{proof}

\begin{empt}
    
Let $\cV$ be a local ring in $\mr{CDVR}^\star(\Zp)$ and let $\fY$ be a smooth $\cV$-formal scheme. We choose a divisor $T_0\subseteq Y_0$, we write $\fX$ for the open of $\fY$ complementary to $T_0$, and we denote by $j\colon \fX \hookrightarrow \fY$ the open immersion. Berthelot proved that the ring homomorphism
$\cD^\dag_{\fY} (\hdag T_0)_{\bbQ} \to j_*\cD^\dag_{\fX, \bbQ}$ is faithfully flat. 
But, in order to give a counterexample of Proposition \ref{prop-ffExt}, we need a version of such 
faithfully flatness for the global sections in the ample case (see Theorem \ref{Gamma5.3.3-Huyghe}). For completeness, we also add the case where $\fY$ is affine (see Lemma \ref{4.3.10Be1-Gamma}). 
\end{empt}


\spa 
The following theorem is a slight complement of \cite[Thm. 5.3.3]{Huyghe98}.
\begin{theo}[Huyghe]\label{Gamma5.3.3-Huyghe}
Let $\fY$ be a projective and smooth $\cV$-formal scheme, $\fT \subseteq \fY$ be a relative ample divisor over $\fS$. Let $\fX \coloneqq \fY \setminus \fT$ and $\cA$ be either $\widehat{\cB}^{(m)}_{\fY} (T_0)_{\bbQ} $, $\cO _{\fY} (\hdag T_0)_{\bbQ}$, 
or $ \cD^\dag_{\fY} (\hdag T_0)_{\bbQ}$ and write $A$ for $\Gamma (\fY, \cA)$.
\begin{enumerate}[(1)]
\item The ring $A$ is coherent in the last case and Noetherian otherwise.  The functors\footnote{By abuse of notation, we also denote by $A$ the constant presheaf on $\fY$ associated to $A$.}  $\Gamma(\mathfrak{Y}, -)$ and $ \cA\otimes_{A}-$
	induce quasi-inverse  exact equivalences
	$$\left\{\text{Coherent }\cA\text{-modules}\right\} \simeq \left\{\text{Coherent }A\text{-modules}\right\}.$$
	Moreover, any coherent $\cA$-module $\cM$  admits a resolution by free $\cA$-modules of finite rank
	and $H^n (\fY, \cM)=0$ for every $n \geq 1$.  
    \spa
\item The extension $A\to \cA$ is flat. In addition, when $\cA= 	\cO _{\fY} (\hdag T_0)_{\bbQ}$ or $\cD^\dag_{\fY} (\hdag T_0)_{\bbQ}$,
the ring homomorphism  $A \to \Gamma (\fX ,\cA)$ is faithfully flat.
\end{enumerate}
\end{theo}

\begin{proof}
We denote by straight letters the global sections of our sheaves.

\spa
{(a)} Let us treat the case  when $\cA= \widehat{\cB}^{(m)}_{\fY} ( T_0)_{\bbQ} $.
The ring $\Gamma (\fY, \widehat{\cB}^{(m)}_{\fY} ( T_0) )$ is Noetherian, as explained in \cite{Huyghe98}, just before Prop. 2.3.2. Since the functor $\Gamma (\fY, -)$  commutes to $-\otimes \bbQ$, then 
$A$ is Noetherian as well. Thanks to [\textit{ibid.}, Prop. 2.3.1.(i)] (resp. [\textit{ibid.}, Prop. 2.3.1.(ii)]) and the beginning of the proof of [\textit{ibid.}, Cor. 4.2.2], we can prove that a coherent $\cA $-module $\cM$ admits a resolution by free $\cA $-modules of finite rank 
(resp. $H^n (\fY, \cM)=0$ for any integer $n \geq 1$). 
Moreover, thanks to [\textit{ibid.}, Prop. 2.3.1] and the five lemma, as in 
\cite[Prop. 5.2.1 and Cor. 5.2.3]{Noot-Huyghe-affinite-proj}, we can check that the functors 
$\cA  \otimes_{A }-$
and $\Gamma (\fY, -)$ induce exact equivalences between the category of 
coherent $\cA $-modules and that of  $A $-modules of finite type. 
Since $A$ is Noetherian, this implies that $\cA  \otimes_{A }-$ is exact on the category of $A$-modules and     then $A \to \cA$ is flat.

\spa

(b) Let us treat now the case  when $\cA= \cO _{\fY} (\hdag T_0)_{\bbQ}$.
\begin{itemize}
    \item[{(i)}]We claim that for any integers $m ' \geq m \geq 0$, the extension $\widehat{B}^{(m)}_{\fY} ( T_0)_{\bbQ} \to \widehat{B}^{(m')}_{\fY} ( T_0)_{\bbQ}$
is flat. Set $\cC \coloneqq \widehat{\cB}^{(m)}_{\fY} ( T_0)_{\bbQ}$, $\cC '\coloneqq \widehat{\cB}^{(m')}_{\fY} ( T_0)_{\bbQ}$, 
$C\coloneqq \Gamma (\fY, \cC)$ and $C'\coloneqq \Gamma (\fY, \cC')$.
Let $I$ be an ideal of $C$. From (a), we get the injection $\cC \otimes_{C} I \subseteq \cC$.
Since $\cC \to \cC'$ is flat (see \cite[Prop. 4.3.2]{Be1}), then 
$\cC' \otimes_{C} I \subseteq \cC'$. 
Hence, $\Gamma (\fY, \cC' \otimes_{C} I)  \subseteq \Gamma (\fY, \cC')=C'$.
Since  $\cC' \otimes_{C} I \riso \cC' \otimes_{C'}  (C' \otimes_C I)$ and
since $C' \otimes_C I$ is a coherent $C'$-module, then using (a), 
$\Gamma (\fY, \cC' \otimes_{C} I) \riso C' \otimes_C I$. Hence, we are done.

\spa
\item[{(ii)}]As $\Gamma (\fY, -)$ commutes with filtered colimits, then $O _{\fY} (\hdag T_0)_{\bbQ}
\riso \varinjlim_m \widehat{B}^{(m)}_{\fY} ( T_0)_{\bbQ}$. 
Using (i) this implies that $O _{\fY} (\hdag T_0)_{\bbQ}$ is coherent. 

\spa
\item[{(iii)}]Since $\widehat{B}^{(m)}_{\fY} ( T_0)_{\bbQ} \to \widehat{\cB}^{(m)}_{\fY} ( T_0)_{\bbQ}$ is flat (see Step (a)),
then by taking the inductive limits so is 
$O _{\fY} (\hdag T_0)_{\bbQ} \to \cO _{\fY} (\hdag T_0)_{\bbQ}$. 
Since $\widehat{\cB}^{(m)}_{\fY} ( T_0)_{\bbQ}\to \cO _{\fY} (\hdag T_0)_{\bbQ}$ is flat, then 
using \cite[Prop. 3.6.2]{Be1}, we get from (a) that 
any coherent $\cO _{\fY} (\hdag T_0)_{\bbQ}$-module $\cM$ admits a resolution by free $\cO _{\fY} (\hdag T_0)_{\bbQ}$-modules of finite rank
(in particular $\cM$ is globally of finite presentation) and that  $H^n (\fY, \cM ) =0$.
Using five lemma, this yields that the functors $\cO _{\fY} (\hdag T_0)_{\bbQ}\otimes_{O _{\fY} (\hdag T_0)_{\bbQ}} -$
and $\Gamma (\fY, -)$ induce exact quasi-inverse equivalence between the category
of coherent $\cO _{\fY} (\hdag T_0)_{\bbQ}$-module and
that of coherent $O _{\fY} (\hdag T_0)_{\bbQ}$-modules.

\spa
\item[{(iv)}]Set $\cB \coloneqq \cA |_{\fX}= \cO_{\fX,\bbQ}$,  $B\coloneqq\Gamma (\fX ,\cA)$. Let us check that $A\to B$ is flat. 
Let $I$ be a finitely generated ideal of $A$. We have to check that the morphism of $B$-modules of finite type
$B \otimes_{A} I\to B$ is injective. 
Since $A\to \cA$ is flat, then $\cA \otimes_{A} I \to \cA$ is injective. 
Hence, so is $\Gamma (\fX, \cA \otimes_{A} I)\to \Gamma (\fX, \cA)= B$.
Since $\fX$ is affine, then the functors 
$\cB \otimes_{B}-$ and $\Gamma (\fX, -)$ induce exact equivalences between the category of 
coherent $\cB$-modules and that of coherent $B$-modules. 
Combining the fact that the functor $ \Gamma (\fX, -)$ commutes with filtered colimits and any $B$-module is a filtered colimits
of $B$-module of finite presentation, then for any $B$-module $N$, 
the canonical morphism $N\to \Gamma (\fX, \cB  \otimes_{B} N) $ is an isomorphism.
This yields the last isomorphism $$\Gamma (\fX, \cA \otimes_{A} I) \riso \Gamma (\fX, \cB \otimes_{A} I) 
\riso \Gamma (\fX, \cB  \otimes_{B} (B  \otimes_{A} I)) \riso  B \otimes_{A} I$$ and we are done. 

\spa

\item[{(v)}]Let us now check that $A \to B$ 
is faithfully flat. Let $M$ be a monogenous $A$-module of finite presentation such that 
$B \otimes_{A} M=0$. 
Following  \cite[Lem. 3.3.5]{Be1}, thanks to (iv), we reduce to prove that  $M=0$. Hence, we get a coherent $\cA$-module by setting
$\cM \coloneqq \cA \otimes_{A} M$.
We have $\cM |_{\fX} \riso   \cB \otimes_{B} (B\otimes_{A} M)$. 
Since $\fX$ is affine, by using the theorem of type $A$ for coherent $\cO _{\fX,\bbQ}$-modules, we get
$\Gamma (\fX, \cM) \riso B\otimes_{A} M=0$
and therefore $\cM |_{\fX}=0$. Since $\cO _{\fY} (\hdag T_0)_{\bbQ} \to j_* \cO _{\fX,\bbQ}$ is faithfully flat (see \cite[Thm. 4.3.10]{Be1}),
then similarly to \cite[Prop. 4.3.12]{Be1}, this yields that $\cM=0$. 
From (iii), we get $M=0$. 
\spa
\item[{(vi)}] Since $A\to B$ is faithfully flat and $B=O_{\fX,\bbQ}$ is Noetherian, then so is $A=O _{\fY} (\hdag T_0)_{\bbQ}$.
\end{itemize}

\spa
(c) Let us prove the case when $\cA=   \cD^\dag_{\fY} (\hdag T_0)_{\bbQ}$.
Set $\cB \coloneqq \cA |_{\fX}$,  $B\coloneqq\Gamma (\fX ,\cA)$. 
Part (1) is \cite[Thm. 5.3.3]{Huyghe98}.

\begin{itemize}

\item[{(i)}] Let us prove that $A \to \cA$ is flat. If $\fa$ is a left ideal of $A$, then it is the filtered colimit of left $A$-modules of finite presentation $(M_i)_{i\in I}$ with $I$ a filtered set.
Let $K_i \coloneqq \Ker (M_i \to A)$ for every $i\in I$. 
Thanks to the exactness of filtered colimits, we get the exact sequence
$$0\to \varinjlim_i K_i \to  \varinjlim_i M_i \to  A.$$ By definition, $\varinjlim_i M_i \to  A$ is injective, which implies the vanishing of $\varinjlim_i K_i$. Set $\cK_i \coloneqq \cA \otimes_A K_i$ and $\cM_i \coloneqq \cA \otimes_A M_i$.
Since $A$ is coherent, each $K_i$ is a coherent $A$-module.
Moreover, in light of the exactness of $\cA \otimes_A -$ for coherent $A$-modules, then 
we get the exact sequence $$0\to \cK_i \to\cM_i \to  \cA.$$
Again, since filtered colimits are exact, this yields the exact sequence
$$0\to \varinjlim_i \cK_i \to  \varinjlim_i \cM_i \to  \cA.$$
In addition, thanks to the isomorphisms $\varinjlim_i \cK_i= \varinjlim_i ( \cA \otimes_A K_i )\riso  \cA \otimes_A  \varinjlim_i K_i  =0$ and $\varinjlim_i \cM_i\riso  \cA \otimes_A  \varinjlim_i M_i  =
\cA \otimes_A \fa$, we deduce the desired result.
\spa

\item[{(ii)}] Using (i), we can check that $A \to B$ is flat as in (b.iv).
Using the faithful flatness of $\cA \to j_* \cB$ (see \cite[Thm. 4.3.10]{Be1}), we can then prove that $A \to B$ is faithfully flat as in (b.v).
\end{itemize}
\end{proof}

\begin{rema}Denote by straight letters the global sections of our sheaves.
With Notation \ref{Gamma5.3.3-Huyghe},  the ring extension 
$\Gamma (\fY, \cO _{\fY} (\hdag T_0)_{\bbQ}\otimes_{\cO_{\fY,\bbQ}} \cD _{\fY,\bbQ})
\to \Gamma (\fX, \cD _{\fX,\bbQ})$ is faithfully flat. 
Indeed, since $\cO _{\fY} (\hdag T_0)_{\bbQ}\otimes_{\cO_{\fY,\bbQ}} \cD _{\fY,n,\bbQ}$ is a coherent
$\cO _{\fY} (\hdag T_0)_{\bbQ}$-module, then the canonical morphism
$$O_{\fX,\bbQ} \otimes _{O _{\fY} (\hdag T_0)_{\bbQ}} \Gamma (\fY, \cO _{\fY} (\hdag T_0)_{\bbQ}\otimes_{\cO_{\fY,\bbQ}} \cD _{\fY,n,\bbQ})
\to  \Gamma (\fX, \cD _{\fX,n,\bbQ})$$
is an isomorphism.  Taking the colimit on $n$, we are done.
\end{rema}

\begin{empt}\label{4.3.10Be1-Gammapre}
Let $\fY$ be a smooth $\cV$-formal scheme, $T_0$ be a divisor of $Y_0$.
Theorems   of type $A$ and $B$ for right or left coherent $\cD^\dag_{\fY} (\hdag T_0)_{\bbQ}$-modules hold (e.g., see \cite[8.7.5.5]{Car25}). In particular, 
for any affine open $\fU' \subseteq \fU$, the morphisms
$\Gamma (\fU, \cD^\dag_{\fY} (\hdag T_0)_{\bbQ}) \to \Gamma (\fU', \cD^\dag_{\fY} (\hdag T_0)_{\bbQ})$
and
$\Gamma (\fU, \cD^\dag_{\fY} (\hdag T_0)_{\bbQ})
\to \cD^\dag_{\fY} (\hdag T_0)_{\bbQ}|_{\fU} $ are flat and the functors
 $\cD^\dag_{\fY} (\hdag T_0)_{\bbQ}|_{\fU} \otimes_{\Gamma (\fU, \cD^\dag_{\fY} (\hdag T_0)_{\bbQ})}-$ 
 and $\Gamma (\fU,-)$ are quasi-inverse exact equivalences between 
 coherent $\cD^\dag_{\fY} (\hdag T_0)_{\bbQ}|_{\fU}$-modules and 
 coherent $\Gamma (\fU, \cD^\dag_{\fY} (\hdag T_0)_{\bbQ})$-modules. 
\end{empt}

\begin{lemm}
\label{4.3.10Be1-Gamma}
Let $\fY$ be an affine smooth $\cV$-formal scheme, $T_0$ be a divisor of $Y_0$, and $\fX\subseteq \fY$ the complement of $T_0$. The extension $\Gamma (\fY, \cD^\dag_{\fY} (\hdag T_0)_{\bbQ}) \to \Gamma (\fX, \cD^\dag_{\fY} (\hdag T_0)_{\bbQ})$ is faithfully flat.      
\end{lemm}

\begin{proof}
    Set $\cA\coloneqq  \cD^\dag_{\fY} (\hdag T_0)_{\bbQ}$, $A\coloneqq\Gamma (\fY ,\cA)$,
$\cB \coloneqq \cA |_{\fX}$,  $B\coloneqq\Gamma (\fX ,\cA)$. 
Since $T_0$ is a divisor, then  $\fX$ is affine.
Hence, it follows from \S \ref{4.3.10Be1-Gammapre} that $A\to B$ is flat. 
Let $M$ be a monogenous $A$-module of finite presentation such that 
$B \otimes_{A} M=0$. 
Following  \cite[Lem. 3.3.5]{Be1}, we reduce to prove that  $M=0$. 
We get a coherent $\cA$-module by setting 
$\cM \coloneqq \cA \otimes_{A} M$.
We have $\cM |_{\fX} \riso   \cB \otimes_{B} (B\otimes_{A} M)$. 
Since $\fX$ is affine, by using the theorem of type $A$ for coherent $\cB$-modules (see \S\ref{4.3.10Be1-Gammapre}), we get
$\Gamma (\fX, \cM) \riso B\otimes_{A} M=0$
and therefore $\cM |_{\fX}=0$.
Following \cite[Prop. 4.3.12]{Be1}, this implies that $\cM=0$. 
Since $\fY$ is affine, then it follows from theorem of type $A$ of coherent $\cA$-modules (see \S\ref{4.3.10Be1-Gammapre}) that $M=0$. 

\end{proof}

\begin{prop}\label{preinjExt-Ddag}
Let $\fY$ be a projective and smooth $\cV$-formal scheme, $\fT \subseteq \fY$ be a relative to $\fS$ ample divisor
(resp. $\fY$ is an affine smooth $\cV$-formal scheme and $T_0$ is a divisor of $Y_0$). Let $\fX \coloneqq \fY \setminus T_0$.
Set $D^\dag_{\fY} (\hdag T_0)_{\bbQ}\coloneqq \Gamma (\fY, \cD^\dag_{\fY} (\hdag T_0)_{\bbQ})$,
$D^\dag_{\fX,\bbQ}\coloneqq \Gamma (\fX, \cD^\dag_{\fX,\bbQ})$.
Let $\cM,\cN $ be two coherent $\cD^\dag_{\fY} (\hdag T_0)_{\bbQ}$-modules and  
$M \coloneqq \Gamma (\fY, \cM)$, $N \coloneqq \Gamma (\fY, \cN)$  be the corresponding coherent $D^\dag_{\fY} (\hdag T_0)_{\bbQ}$-modules.
The following properties are equivalent.
\begin{enumerate}
\item The canonical morphism 
$$\Ext^r_{\cD^\dag_{\fY} (\hdag T_0)_{\bbQ}} (\cM, \cN) \to \Ext^r_{\cD^\dag_{\fX,\bbQ}} (\cM |_{\fX} , \cN |_{\fX} )$$
is injective.
\item The canonical morphism 
$$\Ext^r_{D^\dag_{\fY} (\hdag T_0)_{\bbQ}} (M, N) \to \Ext^r_{D^\dag_{\fX,\bbQ}} (D^\dag_{\fX,\bbQ} \otimes_{D^\dag_{\fY} (\hdag T_0)_{\bbQ} } M, D^\dag_{\fX,\bbQ} \otimes_{D^\dag_{\fY} (\hdag T_0)_{\bbQ} }N  )$$
is injective.
\end{enumerate}
In particular, when $\cN$ is flat, then both maps are injective.
\end{prop}

\begin{proof}
It follows from Theorem \ref{Gamma5.3.3-Huyghe}.(1) (resp. \cite[4.3.6]{Be1}) that the functor $R\Gamma (\fY, -)$ induces the isomorphism
$$R \Hom_{\cD^\dag_{\fY} (\hdag T_0)_{\bbQ}} (\cM, \cN)
\riso R \Hom_{D^\dag_{\fY} (\hdag T_0)_{\bbQ}} (M, N).$$
Similarly, since $\fX$ is affine, using theorem of type $A$ and $B$ for coherent 
$D^\dag_{\fX,\bbQ}$-modules, we get 
$$R \Hom_{\cD^\dag_{\fX,\bbQ}} (\cM |_{\fX}, \cN |_{\fX})
\riso R \Hom_{D^\dag_{\fX,\bbQ}} (\Gamma (\fX, \cM), \Gamma (\fX, \cN)),$$
with $\Gamma (\fX, \cM) \riso D^\dag_{\fX,\bbQ}\otimes_{D^\dag_{\fY} (\hdag T_0)_{\bbQ}} M$
and $\Gamma (\fX, \cN) \riso D^\dag_{\fX,\bbQ}\otimes_{D^\dag_{\fY} (\hdag T_0)_{\bbQ}} N$.
\end{proof}

\begin{exam}\label{CounterexampleExt}
    In the noncommutative case of Proposition \ref{prop-ffExt}, the hypothesis that $N$ is flat is crucial. 
Indeed, thanks to Theorem \ref{Gamma5.3.3-Huyghe},
setting $B=D^\dag_{\fX,\bbQ}$ and $A= D^\dag_{\fY} (\hdag T_0)_{\bbQ} $,  
$A \to B$ is fully faithful. In the geometric context of Corollary \ref{countex}, for example, by taking $\cM=\cN=\calO_{\mathfrak{Y}}({}^\dagger T_0)_\Q$, 
the comparison morphism
$$\Ext^i_{\cD^\dag_{\fY} (\hdag T_0)_{\bbQ}} (\calO_{\mathfrak{Y}}({}^\dagger T_0)_\Q, \calO_{\mathfrak{Y}}({}^\dagger T_0)_\Q)= H^i_{\mr{rig}}(X_0/K)\to H^i_{\mr{conv}}(X_0/K)
	=\Ext^i_{\cD^\dag_{\fX,\bbQ}} (\calO_{\fX,\bbQ}, \calO_{\fX,\bbQ})$$ is not injective for $i=2.$ Hence, thanks to Proposition \ref{preinjExt-Ddag}, this yields the desired counterexample. 
\end{exam}

	

\subsection{A fully faithful functor}
In this section, we prove Theorem \ref{Hol-ff}, which is an extension of Kedlaya's full faithfulness theorem to overholonomic $D$-modules. As a consequence, we get a generalisation of Theorem \ref{thm_RC} with coefficients in overholonomic $D$-modules.

\begin{nota}\label{ntn-Df-h} We recall some notation of \cite[\S 19]{Car25}. A {\it data of coefficients $\mathfrak{C}$ over $\cV$} is the data for any object
$\cW$ of $\mathrm{CDVR}^\star(\cV)$, 
for any smooth formal scheme $\fX$ over $\cW$ of a strictly full subcategory of
$\underrightarrow{LD}^{\rmb}_{\bbQ,\coh} ( \widehat{\cD}_{\fX}^{(\bullet)})$,
which will be denoted by $\mathfrak{C} (\fX/\cW)$,
or simply $\mathfrak{C} (\fX)$ if there is no ambiguity with the base $\cW$.
If there is no ambiguity with $\cV$,
we simply say a {\it data of coefficients}.

\spa

Let $\cW$ be an object of $\mathrm{CDVR}^\star(\cV)$ and $\fX$ be a smooth formal scheme over $\cW$.
  Let $F\text{-}\mathrm{H} (\cD^\dag_{\fX,\bbQ})$ be the category of overholonomic $\cD^\dag_{\fX,\bbQ}$-modules endowed with a Frobenius structure. 
  We get a data of coefficients $\fA$ over $\cV$ so that $\underrvec{l}_{\bbQ}^{*} (\fA (\fX))= F\text{-}\mathrm{H} (\cD^\dag_{\fX,\bbQ})$. 
  We get  $\underrightarrow{LD}^{\rmb}_{\bbQ,F\text{-}\h} \coloneqq\Delta (\overline{\fA}_{\loc})$ the smallest data of coefficients over $\cV$ which is local and    stable under devissage and  cohomology and containing $\fA$ (see notation  \cite[19.0.2.18]{Car25}).
 The data $\underrightarrow{LD}^{\rmb}_{\bbQ,F\text{-}\h}$ is local, stable by devissages, direct summands,
local cohomological functors, pushforwards, extraordinary pullbacks, base change, tensor products, duals, cohomology.
Such objects are called of Frobenius type. 
\end{nota}

\begin{nota}\label{notaHF}
Let $\fP$ be a smooth $\cV$-formal scheme, $P_0$ be its special fibre, $Y_0,Z_0 \subseteq P_0$ be two closed subscheme,
$X_0\coloneqq Y_0 \setminus Z_0$, and $\fU\coloneqq \fP\setminus Z_0$. We denote by $\underrightarrow{LD}^{\rmb}_{\bbQ,F\text{-}\h} (Y_0, \fP, Z_0)$ the full subcategory of  $\underrightarrow{LD}^{\rmb}_{\bbQ,F\text{-}\h} (\fP)$ consisting of 
objects $\cE$ so that $\cE \to \cE (\hdag Z_0)$ is an isomorphism and $\cE$ has his support in $Y_0$  (which is equivalent to saying
the canonical morphism $R \underline{\Gamma}^\dag_{Y_0} (\cE) \to \cE $ is an isomorphism).
When $Y_0$ is empty, this is written as $\mathrm{H}_F ( \fP, Z_0)$.

\spa

We denote by $\underrightarrow{LD}^{\rmb}_{\bbQ,F\text{-}\h} (X_0, \fP/\cW)$
the full subcategory of
$\underrightarrow{LD}^{\rmb}_{\bbQ,F\text{-}\h}  (\fP)$
of objects $\cE$ such that there exists an isomorphism of the form
$\cE \riso R \underline{\Gamma}^\dag_{X_0} (\cE)$.
We have $$\underrightarrow{LD}^{\rmb}_{\bbQ,F\text{-}\h} (X_0, \fP/\cW)= \underrightarrow{LD}^{\rmb}_{\bbQ,F\text{-}\h} (Y_0, \fP, Z_0).$$ In particular, this latter category is independent on the choice of closed subschemes such that $X_0= Y_0\setminus Z_0$. It follows from \cite[19.2.1.6]{Car25} that there exists a canonical t-structure on  $\underrightarrow{LD}^{\rmb}_{\bbQ,F\text{-}\h} (X_0, \fP/\cW)$. We denote by 
$\underrightarrow{LM}_{\bbQ,F\text{-}\h} (X_0, \fP/\cW)$ or by
$\underrightarrow{LM}_{\bbQ,F\text{-}\h} (Y_0, \fP,Z_0/\cW)$ its heart (beware this is not a subcategory of 
$\underrightarrow{LM}_{\bbQ,\coh} ( \widehat{\cD}_{\fP}^{(\bullet)})$).

\spa

We set $D^{\rmb}_{F\text{-}\h} (Y_0, \fP, Z_0)\coloneqq \underrvec{l}_{\bbQ}^{*} \underrightarrow{LD}^{\rmb}_{\bbQ,F\text{-}\h} (Y_0, \fP, Z_0)$ and
$\mathrm{H}_F (Y_0, \fP, Z_0)\coloneqq \underrvec{l}_{\bbQ}^{*}  \underrightarrow{LM}_{\bbQ,F\text{-}\h} (Y_0, \fP, Z_0)$. We get a t-structure on $D^{\rmb}_{F\text{-}\h} (Y_0, \fP, Z_0)$ so that $\mathrm{H}_F (Y_0, \fP, Z_0)$ is its heart. For any integer $n$, we denote by $H^n \colon D^{\rmb}_{F\text{-}\h} (Y_0, \fP, Z_0) \to \mathrm{H}_F (Y_0, \fP, Z_0) $ the $n$th cohomological space.
Beware that the inclusion $D^{\rmb}_{F\text{-}\h} (Y_0, \fP, Z_0) \to D^{\rmb}_{\coh} ( \cD^\dag_{\fP,\bbQ})$ is fully faithful but not t-exact (except when $Z_0$ is a divisor). On the other hand, the inclusion $\iota\colon D^{\rmb}_{F\text{-}\h} (Y_0, \fP, Z_0) \to D^{\rmb}_{F\text{-}\h} (\fP, Z_0) $ is t-exact.
We have the functor 
$$R \underline{\Gamma}^{\dag}_{Y_0}\colon  D^{\rmb}_{F\text{-}\h} (\fP, Z_0) \to D^{\rmb}_{F\text{-}\h} (Y_0, \fP, Z_0).$$
For any $n \in \bbN$, we set 
$\cH^{\dag , n}_{Y_0}\coloneqq H^n \circ R \underline{\Gamma}^{\dag}_{Y_0} 
\colon  D^{\rmb}_{F\text{-}\h} (\fP, Z_0) \to \mathrm{H}_F (Y_0, \fP, Z_0)$.
The functor $R \underline{\Gamma}^{\dag}_{Y_0}$ is a right adjoint of $\iota$.
Indeed, let $\cE_{Y_0} \in D^{\rmb}_{F\text{-}\h} (Y_0,\fP, Z_0)$ and $\cE \in D^{\rmb}_{F\text{-}\h} (\fP, Z_0)$.
Since $\Hom_{ \cD^\dag_{\fP,\bbQ}}  ( \cE_{Y_0} , (\hdag Y_0) (\cE) ) =0$
(see \cite[13.1.4.1]{Car25}), then the map 
$\Hom_{ \cD^\dag_{\fP,\bbQ}}  ( \cE_{Y_0} ,\cE )  \to \Hom_{ \cD^\dag_{\fP,\bbQ}}  ( \cE_{Y_0} , R \underline{\Gamma}^{\dag}_{Y_0} \cE )$
induced by the canonical morphism
$ R \underline{\Gamma}^{\dag}_{Y_0}  \cE \to \cE$ is an isomorphism.

\spa

We denote by $$\bbD_{\fP, Z_0}\coloneqq (\hdag Z_0) \circ \bbD_{\fP}\colon D^{\rmb}_{F\text{-}\h} (Y_0, \fP, Z_0)  \to D^{\rmb}_{F\text{-}\h} (Y_0, \fP, Z_0) $$ the dual functor
which induces the exact dual functor $(\hdag Z_0) \circ \bbD_{\fP}\colon  \mathrm{H}_F (Y_0, \fP, Z_0)  \to  \mathrm{H}_F (Y_0, \fP, Z_0) $
that we simply denote by $\stackrel {*}{}$. If $\fP$ is proper, we can simply set $\mathrm{H}_F (X_0)\coloneqq \mathrm{H}_F (Y_0, \fP, Z_0)$ because this is independent on the choice of the immersion
$X_0\hookrightarrow \fP$ with $\fP$ a proper smooth $\cV$-formal scheme (see \cite[19.2.3]{Car25}).
\end{nota}

\begin{lemm}
Let $\cG \in \rmH_F (\fP, Z_0) $ and $\cG_{Y_0} \in \rmH_F ( Y_0, \fP, Z_0) $.
The canonical morphism 
\begin{equation}
\label{Hol-ff-proofpre1} 
\Hom_{ \cD^\dag_{\fP,\bbQ}} (\cG_{Y_0} , \cH^{\dag, 0}_{Y_0} (\cG) ) 
\to \Hom_{ \cD^\dag_{\fP,\bbQ}} (\cG_{Y_0} , \cG )
\end{equation}
is an isomorphism.
Denoting by $\overset{*}{}$ the dual functor $\bbD_{\fP, Z_0}$,  we have moreover the canonical isomorphism
\begin{gather} \label{Hol-ff-proofpre2}  
\Hom_{ \cD^\dag_{\fP,\bbQ}} (\cG, \cG_{Y_0}) \riso  \Hom_{ \cD^\dag_{\fP,\bbQ}} ((\cH^{\dag, 0}_{Y_0} ( \cG^*) )^*,\cG_{Y_0}  ).
\end{gather}
\end{lemm}

\begin{proof}

Since the canonical map $ \cH^{\dag ,0}_{Y_0} (\cG_{Y_0}) \to \cG_{Y_0}$ is an isomorphism, then the map 
\eqref{Hol-ff-proofpre1}  is surjective. Moreover,  \eqref{Hol-ff-proofpre1} is an injection because $\cH^{\dag, 0}_{Y_0} (\cG ) \to \cG$ is a monomorphism of 
$\rmH_F (\fP, Z_0) $. The dual functor induces the first and last canonical isomorphisms
\begin{gather}
\notag
\Hom_{ \cD^\dag_{\fP,\bbQ}} (\cG, \cG_{Y_0}) 
\riso \Hom_{ \cD^\dag_{\fP,\bbQ}} (\cG_{Y_0}^* , \cG^*)
\\ \notag  \underset {\eqref{Hol-ff-proofpre1} }\riso 
\Hom_{ \cD^\dag_{\fP,\bbQ}} (\cG_{Y_0}^* ,\cH^{\dag, 0}_{Y_0} ( \cG^*))
\riso  \Hom_{ \cD^\dag_{\fP,\bbQ}} (\cH^{\dag, 0}_{Y_0} ( \cG^*)^*,\cG_{Y_0}  ).
\end{gather}

\end{proof}

\begin{theo}\label{Hol-ff}
    If $\fU \coloneqq \fP \setminus Z_0$, the functor $|_{\fU} \colon \rmH_F ( Y_0, \fP, Z_0) \to \rmH_F ( Y_0, \fU) $ is fully faithful. 
\end{theo}

\begin{proof}

(i) Let $\cE,\cF $ be two objects of $\rmH_F ( Y_0, \fP, Z_0) $. Using \cite[15.3.7.7]{Car25}, we deduce that the canonical map $$\Hom_{ \cD^\dag_{\fP,\bbQ}} (\cE, \cF ) \to 
\Hom_{ \cD^\dag_{\fU,\bbQ}} (\cE|_{\fU} , \cF|_{\fU} )$$ is injective. To check the surjectivity, we proceed by induction on the dimension of $Y_0$ and the number of connected components of maximal degree. 
Let $\psi \colon \cE |_\fU\to \cF |_\fU$ be an element of  $\Hom_{ \cD^\dag_{\fU,\bbQ}} (\cE|_{\fU} , \cF|_{\fU} )$.
Let $D_0$ be a divisor of $P_0$ such that $X_0 \setminus D_0$ is smooth and dense in a connected component of 
$X_0$ of dimension $\dim(X_0)$ and such that both $\cE (\hdag D_0)$ and $\cF  (\hdag D_0) $ belong to 
$\mathrm{MIC}^{\dag\dag} _F ( Y_0, \fP, Z_0\cup D_0) $.
We have the exact triangle of $\rmH_F ( Y_0, \fP, Z_0) $ of the form
\begin{equation}
\label{Hol-ff-proof1}
0 \to \cH^{\dag, 0}_{D_0} (\cE) \stackrel {a_{\cE}} \to\cE  \stackrel {b_{\cE}} \to   \cE (\hdag D_0)  \stackrel {c_{\cE}}  \to  \cH^{\dag, 1}_{D_0} (\cE) \to 0
\end{equation}
and similarly for $\cE$ replaced by $\cF$.

\spa

(ii) Arguing as in \cite[Cor. 5.7]{DE20}, the functor $$|_{\fU} \colon \mathrm{MIC}^{\dag\dag} _F ( Y_0, \fP, Z_0\cup D_0)\to \mathrm{MIC}^{\dag\dag} _F ( Y_0, \fU, D_0\cap U_0)$$ is fully faithful.
Hence,  there exists a (unique) morphism 
$\theta_1 \colon \cE (\hdag D_0) \to  \cF (\hdag D_0)$ such that $\theta_1 |_\fU=  (\hdag D_0) (\psi)$.
By the induction hypothesis, there exists a (unique) morphism
$\theta_2 \colon\cH^{\dag, 1}_{D_0} (\cE) \to  \cH^{\dag, 1}_{D_0} (\cF)$ of $\rmH_F ( Y_0\cap D_0, \fP, Z_0) $
such that $ {\theta_2} |_ \fU=  \cH^{\dag, 1}_{D_0\cap U_0} (\psi)$.

\spa

Consider the diagram of $\rmH_F ( Y_0, \fP, Z_0) $
\begin{equation}
\label{Hol-ff-proof2}
\xymatrix {
{  \cE (\hdag D_0)} \ar[d]^-{\theta_1} \ar[r]^-{ c_{\cE}}
& { \cH^{\dag, 1}_{D_0} (\cE) }\ar[d]^-{\theta_2}
\\ {  \cF (\hdag D_0)}  \ar[r]^-{c_{\cF}}
& { \cH^{\dag, 1}_{D_0} (\cF) .}
}
\end{equation}

Using \cite[15.3.7.7]{Car25}, the diagram \eqref{Hol-ff-proof2} is commutative because so is on $\fU$.
Since $\mr{Ker} c_{\cE}= \mr{Im}\, b_{\cE}$ and $\mr{Ker} c_{\cF}= \mr{Im}\, b_{\cF}$, we get therefore 
a morphism $\theta \colon \mr{Im}\, b_{\cE}\to \mr{Im}\, b_{\cF}$ such that 
$\theta |_{\fU}$ is the morphism induced by $\psi$. 

\spa

(iii) Let $\cG$ be an object of $\rmH_F ( \fP, Z_0) $. 
Consider the commutative diagram:
$$ \xymatrix{ 
{\Hom_{ \cD^\dag_{\fP,\bbQ}} (\cG , \cH^{\dag, 0}_{D_0} ( \cF)) } 
\ar[r]^-{\sim}_-{\eqref{Hol-ff-proofpre2}}\ar[d]^-{}
& {\Hom_{ \cD^\dag_{\fP,\bbQ}} (\cH^{\dag, 0}_{D_0} ( \cG^*)^*,\cH^{\dag, 0}_{D_0} ( \cF)  )}\ar[d]^-{}
\\ {\Hom_{ \cD^\dag_{\fU,\bbQ}} (\cG |_{\fU}  , \cH^{\dag, 0}_{D_0\cap U_0} ( \cF |_{\fU}) ) } 
\ar[r]^-{\sim}_-{\eqref{Hol-ff-proofpre2}}
& {\Hom_{ \cD^\dag_{\fU,\bbQ}} (\cH^{\dag, 0}_{D_0\cap U_0} ( (\cG |_{\fU} )^*)^* ,\cH^{\dag, 0}_{D_0\cap U_0} ( \cF|_{\fU})   )} .}$$
By induction hypothesis, the right arrow is an isomorphism. Then so is the left one.

\spa

(iv) The map 
$$ \Ext^1_{ \cD^\dag_{\fP,\bbQ}} ( \cE , \cH^{\dag, 0}_{D_0} (\cF)) \to \Ext^1_{ \cD^\dag_{\fU,\bbQ}} ( \cE|_{\fU} , \cH^{\dag, 0}_{D_0} (\cF)  |_{\fU}) $$
is injective. 
Indeed, let $$0\to \cH^{\dag, 0}_{D_0} (\cF) \to \cG \to \cE \to 0$$ be an exact sequence of $\rmH_F ( \fP, Z_0) $. If 
$\cH^{\dag, 0}_{D_0} (\cF) |_\cU\to \cG |_\cU$ has a retract, then it follows from (iii) and the faithfulness of $\cU$ that
$ \cH^{\dag, 0}_{D_0} (\cF) \to \cG$ has also a retract and we are done.

\spa

(v) Consider the commutative diagram of $K$-vector spaces: 
\begin{equation}
\label{Hol-ff-proof3}
\xymatrix { {\Hom_{ \cD^\dag_{\fP,\bbQ}} (\cE, \cF )  } \ar[r]^-{} \ar[d]^-{}
& {\Hom_{ \cD^\dag_{\fP,\bbQ}} (\cE, \mr{Im}\, b_{\cF})  } \ar[r]^-{}\ar[d]^-{}
& { \Ext^1_{ \cD^\dag_{\fP,\bbQ}} ( \cE , \cH^{\dag, 0}_{D_0} (\cF))} \ar@{^{(}->}[d]^-{}
\\ {\Hom_{ \cD^\dag_{\fU,\bbQ}} (\cE|_{\fU} , \cF|_{\fU} )  } \ar[r]^-{}
& {\Hom_{ \cD^\dag_{\fU,\bbQ}} (\cE|_{\fU}, \mr{Im}\, b_{\cF} |_{\fU})  } \ar[r]^-{}
& { \Ext^1_{ \cD^\dag_{\fU,\bbQ}} ( \cE|_{\fU} , \cH^{\dag, 0}_{D_0} (\cF)  |_{\fU})   } } 
\end{equation}
whose rows are exact. According to (iv), the right arrow of \eqref{Hol-ff-proof3} is injective.

\spa

Let $\eta\colon \cE \to \mr{Im}\, b_{\cF}$ be the composition of $\cE \to \mr{Im}\, b_{\cE}$ with the morphism $\theta$ constructed in (ii).
Since the bottom row of \eqref{Hol-ff-proof3} is exact, since $\eta |_{\fU}$ is the composition of $\psi$ with $\cF  |_{\fU}\to \mr{Im}\, b_{\cF} |_{\fU}$, then 
the image of $\eta |_{\fU}$ in $\Ext^1_{ \cD^\dag_{\fU,\bbQ}} ( \cE|_{\fU} , \cH^{\dag, 0}_{D_0} (\cF)  |_{\fU})  $ vanishes. 
Since the right arrow of \eqref{Hol-ff-proof3}  is injective, this yields that 
the image of $\eta$ in $\Ext^1_{ \cD^\dag_{\fP,\bbQ}} ( \cE , \cH^{\dag, 0}_{D_0} (\cF))$ vanishes. 
By the exactness of the top row, we get that $\eta$ comes from an element $\phi$ of $\Hom_{ \cD^\dag_{\fP,\bbQ}} (\cE, \cF )  $.
Since the image of $\psi - \phi |_{\fU}$ in $\Hom_{ \cD^\dag_{\fU,\bbQ}} (\cE|_{\fU}, \mr{Im}\, b_{\cF} |_{\fU})  $ is null,
then it comes from an element $\beta$ of $\Hom_{ \cD^\dag_{\fU,\bbQ}} (\cE |_{\fU}  , \cH^{\dag, 0}_{D_0\cap U_0} ( \cF |_{\fU}) )   $.
It follows from (iii) that there exists $\alpha \in \Hom_{ \cD^\dag_{\fP,\bbQ}} (\cE , \cH^{\dag, 0}_{D_0} ( \cF))$
such that $\alpha |_{\fU} = \beta$. Hence, $\psi - \phi |_{\fU} = (a_{\cE}\circ \alpha) |_{\fU}$, i.e., 
$\psi = (\phi +a_{\cE}\circ \alpha) |_{\fU}$. Hence, we are done.

\end{proof}

\begin{coro}\label{cor_Hol-ff}
Let $\cE,\cF $ be two objects of $\rmH_F ( Y_0, \fP, Z_0) $. The map 
$$\Ext^1_{ \cD^\dag_{\fP,\bbQ}} ( \cE , \cF) \to \Ext^1_{ \cD^\dag_{\fU,\bbQ}} ( \cE|_{\fU} , \cF  |_{\fU})  $$
is injective. 
\end{coro}
\begin{proof}
    This follows from Theorem \ref{Hol-ff} as in Theorem \ref{thm_RC}.
\end{proof}


\bibliographystyle{ams-alpha}
\bibliographystyle{ams-alpha}

\end{document}